\documentclass[a4paper,headings = small, abstract]{scrartcl}
\usepackage{amsmath}
\usepackage{amscd,amsthm,array,bbm,hhline, enumerate,dsfont, booktabs,fancybox,calc,textcomp,xcolor,graphicx,bbm,xspace,nicefrac,url,arcs,listings,nameref,float}
\usepackage{subcaption}

\makeatletter
\newcommand{\leqnomode}{\tagsleft@true}
\newcommand{\reqnomode}{\tagsleft@false}
\makeatother

\usepackage[theoremfont]{newpxtext}
\usepackage[vvarbb, upint, bigdelims]{newpxmath}
\usepackage[scr=boondoxupr, scrscaled = 1.05, cal = euler, calscaled =1.05]{mathalpha}
\usepackage[scaled=0.95]{inconsolata}
\DeclareMathAlphabet{\mathsf}{OT1}{\sfdefault}{m}{n}
\SetMathAlphabet{\mathsf}{bold}{OT1}{\sfdefault}{b}{n}
\usepackage[utf8]{inputenc}
\usepackage[shortlabels]{enumitem}
\usepackage{etoolbox}
\usepackage[normalem]{ulem}
\usepackage{mathtools}
\usepackage{geometry}
\usepackage{float}

\usepackage{bm}
\usepackage{tikz}
\usepackage{comment}

\usepackage[outdir =./]{epstopdf}
\usepackage{xurl} 
\usepackage[citestyle = numeric-comp,
bibstyle = numeric, 
giveninits=true,
natbib = true, 
backend = biber,
maxnames=99,
minnames=1,
url=false,
doi=false]{biblatex}

\numberwithin{equation}{section}
\graphicspath{{./simulationen/}}
\usepackage{chngcntr}
\counterwithin{figure}{section}

\usepackage{xcolor}
\usepackage[colorlinks=true, allcolors=myteal]{hyperref}
\definecolor{WIMgreen}{RGB}{60 134 132}
\definecolor{red_pers}{RGB}{204 37 41}
\definecolor{UMblue}{RGB}{4 47 86}
\definecolor{myteal}{RGB}{0 123 137}
\definecolor{nd}{RGB}{0 0 0}
\definecolor{dartmouthgreen}{rgb}{0.05, 0.5, 0.06}\definecolor{cobalt}{rgb}{0.0, 0.28, 0.67}\definecolor{coolblack}{rgb}{0.0, 0.18, 0.39}
\definecolor{glaucous}{rgb}{0.38, 0.51, 0.71}\definecolor{hooker\'sgreen}{rgb}{0.0, 0.44, 0.0}\definecolor{lemonchiffon}{rgb}{1.0, 0.98, 0.8}\definecolor{oucrimsonred}{rgb}{0.6, 0.0, 0.0}\definecolor{radicalred}{rgb}{1.0, 0.21, 0.37}\definecolor{raspberry}{rgb}{0.89, 0.04, 0.36}\definecolor{royalazure}{rgb}{0.0, 0.22, 0.66}
\definecolor{dex}{RGB}{138 18 34}
\definecolor{cs}{rgb}{0.0, 0.44, 1.0}
\definecolor{Purple}{RGB}{103 58 183}

\theoremstyle{plain}
\newtheorem{theorem}{Theorem}[section]
\newtheorem{proposition}[theorem]{Proposition}
\newtheorem{lemma}[theorem]{Lemma}
\newtheorem{corollary}[theorem]{Corollary}

\theoremstyle{definition}

\theoremstyle{remark}
\newtheorem{remark}[theorem]{Remark}

\SetLabelAlign{center}{\hss#1\hss}

\def\bbeta{\boldsymbol{\beta}}

\def\tr{\operatorname{tr}}

\def\C{\mathbf{C}}

\def\A{\boldsymbol{A}}
\def\bSigma{\mathbf{\Sigma}}

\def\E{\mathbb{E}}

\def\M{\mathbf{M}}
\def\N{\mathbb{N}}

\def\R{\mathbb{R}}

\definecolor{RTGblue}{HTML}{003466}
\def\Z{\mathbf{Z}}

\def\Rr{\mathbf{R}}

\def\X{\mathbf{X}}
\def\Z{\mathbf{Z}}

\def\N{\mathbb{N}}
\def\P{\mathbb{P}}

\newcommand{\ep}{\varepsilon}

\newcommand{\1}{\mathds{1}}

\def\vec{\operatorname{vec}}
\newcommand{\e}{\mathrm{e}}
\newcommand{\md}{\mathrm{d}}

\newcommand{\mix}{\mathrm{mix}}

\renewcommand{\tilde}{\widetilde}%
\newcommand{\T}{\mathbb{T}}\renewcommand{\d}{\,\mathrm{d}}

\newcommand*\samethanks[1][\value{footnote}]{\footnotemark[#1]}
\setkomafont{sectioning}{\normalcolor\bfseries}
\setkomafont{author}{\large}
\setkomafont{date}{\normalsize}

\newlist{todolist}{itemize}{2}
\setlist[todolist]{label=$\square$}

\let\originalleft\left
\let\originalright\right
\renewcommand{\left}{\mathopen{}\mathclose\bgroup\originalleft}
\renewcommand{\right}{\aftergroup\egroup\originalright}

\definecolor{myteal}{RGB}{0 123 137}
\definecolor{radicalred}{rgb}{1.0, 0.21, 0.37}
\allowdisplaybreaks
	\bibliography{lasso}

\makeatother
\title{\fontsize{16}{19} \selectfont Sparse Estimation for High-Dimensional L{\'e}vy-driven Ornstein--Uhlenbeck Processes from Discrete Observations}
\author{Niklas Dexheimer\thanks{ University of Twente \newline Department of Applied Mathematics \newline Drienerlolaan 5, 7522 NB Enschede, The Netherlands.\newline Email: \href{mailto:n.dexheimer@utwente.nl}{n.dexheimer@utwente.nl}/\href{mailto:n.jeszka@utwente.nl}{n.jeszka@utwente.nl}
}
\qquad Natalia Jeszka\samethanks
}

\begin{document}
	\maketitle
	\begin{abstract}
We study high-dimensional drift estimation for L{\'e}vy-driven Ornstein--Uhlenbeck processes based on discrete observations. Assuming sparsity of the drift matrix, we analyze Lasso and Slope estimators constructed from approximate likelihoods and derive sharp nonasymptotic oracle inequalities. Our bounds disentangle the contributions of discretization error and stochastic fluctuations, and establish minimax optimal convergence rates under suitable choices of tuning parameters in a high-frequency regime. We further quantify the sample complexity required to attain these rates depending on the L{\'e}vy noise. The results extend the theory of high-dimensional statistics for stochastic processes to a substantially broader class of noise mechanisms, in particular pure jump processes. They also demonstrate that Lasso and Slope remain competitive for jump-driven systems, providing practical guidance for inference in applications where L{\'e}vy processes are a natural modeling choice.
	\end{abstract}

	\section{Introduction}\label{sec: intro}
 High-dimensional statistics is a central field of modern statistical research, with a particular emphasis on settings where the number of parameters significantly exceeds the sample size. One common assumption in this setting is sparsity of the underlying parameter. Traditional estimation techniques are generally inadequate in this regime, prompting the introduction of penalized estimators to overcome these limitations. The Lasso estimator \cite{tibshirani96} is one of the most widely studied examples. It is specifically tailored for sparse settings by adding a $L_1$ penalty to the estimation objective. A generalization of the Lasso estimator is given by the Slope estimator \cite{bog15}, which introduces additional weight parameters to the $L_1$ penalty. The behavior of these estimators is by now well understood in classical settings, such as sparse linear regression where minimax optimality has been verified in \cite{belets18}.  In recent years, growing interest has focused on extending high-dimensional statistical theory to continuous-time settings, especially diffusion processes \cite{gama19,cmp20,dex24,amorino2025sampling,degreg25}. A prominent class of models in this context are multivariate Ornstein--Uhlenbeck (OU) processes. These are $d$-dimensional stochastic processes $\X=(X_t)_{t\geq0},$ which satisfy a stochastic differential equation (SDE) of the following form
 \begin{equation}\label{eq: sde intro}
     \d X_t=-\A_0 X_t\d t+\d W_t,\quad t\geq0,  
 \end{equation}
 where $\A_0\in\R^{d\times d}$ is called the drift matrix and $(W_t)_{t\geq0}$ is a standard $d$-dimensional Brownian motion. Scalar OU processes are often used for modeling interest rates, which is known as the Vasicek model. Therefore, multivariate OU processes are natural candidates for modeling interbank lending \cite{carmona15,fouque13}. The interactions between banks are encoded in the drift matrix $\A_0$ in this setting and for a very high number of banks it is natural to assume sparsity of $\A_0$. A straightforward way to generalize this model is to incorporate random shocks, such as changes to the key interest rate, by replacing the Brownian motion with a general $d$-dimensional L{\'e}vy process. Other applications of L{\'e}vy-driven OU processes are for example in computational neuroscience, where they are used for modeling the postsynaptic membrane potential in biological neural networks, see Section 2.1 in \cite{dexheimer2025spiketimingdependent}. A particularly interesting point of this model is that the driving L{\'e}vy process is a pure jump process. From a theoretical point of view, extending the noise in \eqref{eq: sde intro} to general L{\'e}vy processes also broadens the class of possible limiting distributions of $\X$ from normal distributions to the class of self-decomposable distributions, for details see \cite{mas04,sato84}.

\paragraph{Main Contributions}\hspace{-6pt}
This work investigates estimation of the drift matrix $\A_0$ of a $d$-dimensional L{\'e}vy-driven OU process $\X$, based on $n\in\N$ equidistant discrete observations of $\X$ separated by $\Delta_n>0$. We are especially interested in the case where $\A_0$ is sparse, which is why we employ Lasso and Slope estimators. For this we define a localized and truncated likelihood respectively contrast function see Section \ref{subsec: est}. In summary, our main contributions are as follows. 
 \begin{enumerate}
 \item We show sharp oracle inequalities for the $L_2$-error of the estimators, see Theorem \ref{thm: main}. These disentangle the contribution of the discretization, truncation and stochastic fluctuations to the overall error.
    \item  The stochastic error of both estimators is shown to be of the order \[\frac{s\log\Big(\frac{\e d^2}{s}\Big)}{T}, \]
    where $s$ denotes the sparsity of $\A_0$ and $T=n\Delta_n$ the observation length. This is the minimax optimal rate of convergence under sparsity constraints if a continuous record of observations of $\X$ up to time $T$ is available (see Theorem 2.7 in \cite{dex24}). Additionally, the discretization error is bounded by $d^2\Delta_n^2,$ which improves previous results in the literature. Optimal choices of the truncation level, depending on the background driving L{\'e}vy process, are presented.    To the best of our knowledge, this is the first result guaranteeing minimax optimality for sparse regression with high-frequency discrete observations of a diffusion process.
    \item The results hold true under minimal assumptions on the background driving L{\'e}vy process, namely as soon as the BDLP admits a $p$-th moment with $p>2$. Our analysis in particular permits pure jump processes and anisotropic noise, which to the best of our knowledge has not been investigated before in high-dimensional statistics for stochastic processes.
    \item We also prove a novel matrix Bernstein-type concentration inequality for the localized empirical covariance matrix of $\X$ (see Proposition \ref{prop: REP disc}) to verify an analogue to the restricted eigenvalue property needed for the theoretical analysis of the Lasso estimator (see Sections 3.1 of \cite{cmp20,dex24}). This leads to a sample complexity which grows at most polynomial in terms of the sparsity $s$ depending on the behavior of the BDLP.
\end{enumerate}
\paragraph{Related Work}
 The same setting as in this paper was investigated in \cite{dex24} except for the highly idealized assumption of a continuous record of observations of $\X$ being available. The authors were able to show that the proposed estimators achieve the minimax optimal rate of convergence for suitably chosen tuning parameters. Nevertheless, the paper has multiple restrictive assumptions in addition to the unrealistic continuous observation scheme. First, \cite{dex24} assumes knowledge of the continuous martingale part of the observed L{\'e}vy-driven OU process, since this is needed in order to compute the likelihood function (see \cite{sorensen91}). However, even given continuous observations of $\X,$ said continuous martingale part is not easily identifiable as soon as the background driving L{\'e}vy process (BDLP) has infinite jump activity. Moreover, the estimators cannot be applied if the BDLP is a pure jump process since they are based on the likelihood function and consequently on the continuous martingale part, which does not exist for pure jump processes. The analysis of \cite{dex24} also relies on a sufficiently strong concentration of the empirical covariance matrix of $\X$ (see Assumption $(\mathcal{H})$ in \cite{dex24}), as is usual in the analysis of the Lasso estimator (for more details see Section \ref{subsec: cov}). However, in the presence of jumps only a very weak concentration is verified (see \cite[Proposition 4.2]{dex24}). In the case of a heavy-tailed L{\'e}vy measure of the BDLP this leads to a sample complexity growing exponentially in terms of the dimension $d,$ which is highly unsatisfactory. 
 
 The traditional way to transform the results of \cite{dex24} to a discrete observation scheme is to use a jump filtering approach in order to approximate the continuous martingale part in the likelihood function. This method was used in \cite{mai14} for maximum likelihood estimation of the drift parameter of a discretely observed scalar L{\'e}vy-driven OU process. However, as stated in Remark 3.4 in \cite{mai14}, there is ``no hope'' to recover the continuous martingale part as soon as the drift of the BDLP is non-zero, which excludes many types of anisotropic martingale noise. As discussed above, jump filtering is in addition also not possible for pure jump processes, which is why we use a different approach.
 
 The first paper to investigate high-dimensional estimation for discretely observed diffusion processes was \cite{amorino2025sampling}. Compared to our work, the authors investigate a more general class of parametric diffusion processes with a standard Brownian motion as noise. In addition, classical OU processes are investigated since those are not contained in the considered general framework. A detailed comparison of the results to our paper is given in Remark \ref{rem:1}; in short, through our detailed analysis of the OU case we improve both the upper bound on the stochastic error and on the discretization error. \cite{degreg25} investigates an adaptive Elastic-Net estimator for very general diffusion processes given discrete observations. The paper is not only interested in parametric estimation of the drift but also of the volatility. The estimator is defined as a penalized version of a given original estimator, which is the MLE in most cases. Due to the high generality of the paper, the authors are mostly interested in showing that the Elastic-Net estimator matches the performance of the original estimator. Another related article is given by \cite{wong20} which investigates sparse regression for heavy-tailed $\beta$-mixing time series. Indeed, under the imposed assumptions (see Assumption \ref{ass: ergodicity}) the L{\'e}vy-driven OU process $\X$ is $\beta$-mixing by the results of \cite{mas04}, and allowing general L{\'e}vy process as noise in \eqref{eq: sde intro} corresponds to allowing heavy tailed noise in the time series dynamics investigated in \cite{wong20}. Comparing our results to \cite{wong20} we improve the rate of convergence of the Lasso estimator and also allow for heavy-tailed noise which only admits a polynomial moment instead of sub-Weibull noise.
 \paragraph{Structure}
The paper is structured as follows. We start by introducing the mathematical framework, notation and estimators in Section \ref{sec: preliminaries}. Thereafter, we present our main results in Section \ref{sec: results}, whose proofs are contained in Section \ref{sec: proof}. Subsequently, we conduct a simulation study on synthetic data in Section \ref{sec: sim}. The paper closes with a conclusion and future outlook in Section \ref{sec: conclusion}. Appendix \ref{app} contains some useful auxiliary results.

\section{Preliminaries and Notation} \label{sec: preliminaries}
Before we rigorously introduce our mathematical model we fix some basic notation. For $x\in\R^d,a_0>0,$ we denote the open ball in $\R^d$ around $x$ with radius $a_0$ by $B(x,a_0),$ for $n\in\N$ we set $[n]\coloneq\{1,\ldots,n\},$ and for a set $S$ we denote its complement by $S^{\mathsf c}$. For a symmetric matrix $\M\in\R^{d\times d}$, we write $\lambda_{\max}(\M)$ and $\lambda_{\min}(\M)$ for the largest and the smallest eigenvalue of $\M$, and for any positive definite matrix $\M\in\R^{d\times d}$ we denote by $\kappa(\M)\coloneq\lambda_{\max}(\M)/\lambda_{\min}(\M)$ the condition number of $\M$. For general $\M\in\R^{d_1\times d_2}$, we define
\[
\|\M\|_p\coloneq \left(\sum_{1\le i\le d_1,1\le j\le d_2}|M_{ij}|^p\right)^{1/p}, \quad p\ge 1,
\]
and $\Vert\M\Vert_\infty\coloneq \max_{1 \leq i\leq d_1, 1 \leq j\leq d_2}\vert \M_{ij}\vert$. By a slight abuse of notation we denote the number of non-zero entries of the matrix $\M$ by $\Vert \M\Vert_0,$ i.e.\ $\Vert \M\Vert_0\coloneq \sum_{1\le i\le d_1,1\le j\le d_2} \1(M_{ij}\neq 0)$.
Furthermore we set $\Vert \M\Vert$ to be the spectral norm, i.e.\ $\Vert \M\Vert^2=\lambda_{\max}(\M^\top \M)$.

The set of all real $d\times d$ matrices with the real parts of all eigenvalues positive is denoted by $M_+(\R^d)$, i.e., $\M\in M_+(\R^d)$ if and only if $\|\e^{-t\M}\|\to0$ as $t\to\infty$. 

Given $\bbeta=(\beta_1,\ldots,\beta_d)\in\R^d$, we let $(\beta_1^\#,\ldots,\beta_d^\#)$ denote a nonincreasing rearrangement of $|\beta_1|,\ldots,|\beta_d|$, and set
\[
\|\bbeta\|_\star\coloneq \sum_{j=1}^d\sqrt{\log\left(\frac{2d}{j}\right)} \beta_j^\#.\]
Then by Proposition 1.2 in \cite{bog15} $\Vert \cdot\Vert_\star$ defines a norm on $\R^d.$
For $\M\in \R^{d_1\times d_2}$, we set (by a slight abuse of notation) $\Vert \M\Vert_\star\coloneq \Vert \vec(\M)\Vert_\star,$ where $\vec$ denotes the usual vectorization operator. Then it holds
\begin{equation}\label{def:normstar}
	\Vert \M\Vert_\star=\sum_{j=1}^{d_1d_2}\vec(\M)^\#_j\sqrt{\log\left(\frac{2d_1d_2}{j}\right)}. 
\end{equation}

\subsection{Mathematical Model} \label{subsec: mat mod}

Throughout the paper, we let $\Z=(Z_t)_{t\ge0}$ denote a $d$-dimensional L{\'e}vy process defined on a filtered probability space $(\Omega,\mathcal F,(\mathcal F_t),\P)$ and adapted to the filtration $(\mathcal F_t)_{t\ge0}$. We will consider the $d$-dimensional L{\'e}vy-driven Ornstein--Uhlenbeck (OU) process, defined as a strong solution $\X=(X_t)_{t\ge0}$ to the stochastic differential equation
\begin{equation}\label{sde:ou}
\d X_t=-\A_0 X_t\d t+\mathrm{d} Z_t,\quad t>0,\qquad X_0\sim\pi,
\end{equation}
where $\A_0\in\R^{d\times d}$ is the drift parameter and $\pi$ the initial distribution of $\X$. The process $\Z$ is referred to as the background driving L{\'e}vy process (BDLP). We assume that the initial condition $X_0$ is independent of $\Z$.

An explicit representation of the solution to the SDE \eqref{sde:ou} is given by
\begin{equation}\label{eq: ou explicit}
X_t=\e^{-t\A_0}X_0+\int_0^t\e^{-(t-s)\A_0}\d Z_s,\quad t>0,
\end{equation}
which follows from It{\^o}'s formula (see e.g.\ equations (1.1) and (1.2) in \cite{mas04}).

Denote by $(b,\C=\bSigma\bSigma^\top,\nu)$ the generating triplet of $\Z$, i.e., $b\in\R^d$, $\C=\bSigma\bSigma^\top\in\R^{d\times d}$ is a symmetric, positive semidefinite matrix and $\nu$ is a Lévy measure, i.e.\ a $\sigma$-finite measure on $\R^d$ satisfying 
\[\nu(\{\boldsymbol 0\})=0\quad\text{ and }\quad \int_{\R^d}\min\{1,\|z\|^2\}\,\nu(\mathrm{d} z)<\infty.\]
The L{\'e}vy--It\={o} decomposition (see e.g.\ Theorem 2.4.16 in \cite{applebaum09}) then implies
\[
Z_t=bt+\Sigma W_t+\int_0^t\int_{\Vert z\Vert\geq1} zN(\md s,\md z)+\int_0^t\int_{\Vert z\Vert<1} z\tilde{N}(\md s,\md z), \quad t\geq 0,
\]
where $(W_s)_{s\geq 0}$ is a $d$-dimensional Brownian motion, $N$ is a Poisson random measure on $[0,\infty)\times \R^d$ with intensity measure given by $\mathbf{\lambda}\otimes \nu$, and $\tilde{N}$ denotes its compensated counterpart. 

Regarding the BDLP $\Z$, or equivalently its L{\'e}vy measure  $\nu$, we distinguish between the following cases.

\begin{itemize}
    \item The BDLP $\Z$ is continuous, which happens if and only if $\nu(\R^d)=0,$ i.e\ if $\Z$ is a Brownian motion.
    
    \item  The process $\Z$ has bounded jumps, i.e., there exists $a_0>0,$ such that 
\begin{equation*}
    \nu(B(0,a_0)^\mathsf{c})=0.
\end{equation*}

\item For $p\geq 2,$ the process $\Z,$ respectively its L{\'e}vy measure  $\nu$,  admits a $p$-th moment, i.e., 
\[\nu_p\coloneq \int \Vert z\Vert^p\,\nu(\mathrm{d} z)<\infty.\]
\item Additionally, we say that $\Z,$ respectively $\nu,$ is sub-Weibull with parameter $\alpha>0,$ if $\Z$ admits a second moment and if there exists a constant $c_\alpha>0,$ such that 
\[\int \Vert x\Vert^2\exp(c_\alpha \Vert x\Vert^\alpha)\,\nu(\mathrm{d} x)\leq 2\nu_2.\] 
By continuity such a constant always exists if $\Vert x\Vert^2\exp(c\Vert x\Vert^\alpha)$ is $\nu$-integrable for some $c>0.$ Compared to the usual sub-Gaussian norm for $\alpha=2$ (see Definition 2.5.6 in \cite{vers18}), the term $\Vert x\Vert^2$ is added in order to ensure integrability around the origin.
\end{itemize}
\subsection{The Lasso and Slope Estimators}\label{subsec: est}
In the following, our goal will be to estimate the drift parameter $\A_0\in\R^{d\times d}$ of the L{\'e}vy-driven OU process $\X$ defined in \eqref{sde:ou} based on discrete-time observations of the process. More specifically, for $n\in\N,T>0,$ we assume that observations $\left( X_{t_i}\right)_{i=0, \ldots, n}$ of $\X$ are available  where $0=t_0<t_1<\ldots<t_n=T.$ For simplicity, we assume that the observation times $(t_i)_{i\in \{0,\ldots,n\}}$ are equidistant, i.e.\ it holds $t_i = i T / n,i=0,\ldots,n$, and we denote the inter observation distance by $\Delta_n = T/ n$.

One of the classical approaches in parametric statistics is to apply the maximum likelihood estimator (MLE). Given a continuous record of observations of $\X$ up to time $T>0$ the negative log-likelihood function $\mathcal{L}_T\colon \R^{d\times d}\to \R$ is, under suitable assumptions on the BDLP $\Z$, explicitly computable by an application of Girsanov's theorem (see \cite{sorensen91})\begin{equation*}
    \mathcal{L}_T(\A) = \frac 1 T\int_0^T (\A X_{s-})^\top \d X_s^c + \frac{1}{2T} \int_0^T \| \A X_{s-} \|^2 \d s,\quad \A\in\R^{d\times d},
\end{equation*} 
where $(X^c_t)_{t\in[0,T]}$ denotes the continuous martingale part of $\X$ and $X_{s-}$ the left limit of $\X$ at time $s\in[0,T]$. A natural version of the likelihood function based on discrete observations is given by 
\begin{equation}\label{eq: disc like untrunc}
    \frac{1}{T}\sum_{i=1}^n \left(\A X_{t_{i-1}}\right)^\top \Delta X_i   +\frac{\Delta_n}{2T}\sum_{i=1}^n \Vert \A X_{t_{i-1}}\Vert^2
\end{equation}
where
\[
\Delta X_i = X_{t_{i}} - X_{t_{i-1}}.
\]
A crucial difference from the true likelihood function is that the continuous martingale part is replaced by the increment of the original process $\X$. Straightforward computations show that minimizing \eqref{eq: disc like untrunc} with respect to $\A$ is equivalent to minimizing the contrast function
\begin{equation}\label{eq: contrast untrunc} 
\frac{1}{T}\sum_{i=1}^n \| \Delta X_i - \Delta_n \A X_{t_{i-1}} \|^2,
\end{equation}
i.e., both are minimized by the same value of the parameter $\A$. 

As we are interested in a high-dimensional regime with possibly sparse drift matrix $\A_0$ we intend to use a penalized version of the likelihood as minimization  objective for our estimation procedure, as for example in the classical Lasso estimator \cite{tibshirani96}. Given a continuous record of observations of $\X$ and knowledge of the continuous martingale part, \cite{dex24} used penalized versions of $\mathcal{L}_T$ to define Lasso (with $\Vert\cdot\Vert_1$ penalization) and Slope (with $\Vert\cdot\Vert_\star$ penalization) estimators of the drift matrix of $\X$. Based on discrete observations of a continuous OU process, i.e.\ where the BDLP $\Z$ is a standard Brownian motion, \cite{amorino2025sampling} defined a Lasso estimator as minimizer of the contrast function  \eqref{eq: contrast untrunc} with $\Vert\cdot\Vert_1$ penalization. 
However, this approach is not suitable for the general L{\'e}vy-driven case, which will be illustrated below. We therefore propose the following pseudo-likelihood function
\begin{equation}
    \label{eq: like disc}
    \mathcal{L}^D_n(\A)\coloneq \frac{1}{T}\sum_{i=1}^n \left(\A X_{t_{i-1}}\right)^\top \Delta X_i \1_B\left(X_{t_{i-1}}\right) \1_{\{\Vert \Delta X_i \Vert< \eta\}} +\frac{\Delta_n}{2T}\sum_{i=1}^n \Vert \A X_{t_{i-1}}\Vert^2 \1_B\left(X_{t_{i-1}}\right)\1_{\{\Vert \Delta X_i \Vert< \eta\}},
\end{equation} where $B \subseteq \R^d$ is a bounded set with radius $b\coloneq\operatorname{rad}(B)<\infty$, and $\eta > 0$ is referred to as truncation level. Canonical choices of $B$ are discussed after Assumption \ref{ass: ergodicity}, explicit values of $\eta$ depending on the BDLP $\Z$ are given in Table \ref{table:1}.
As above, minimizing the pseudo-likelihood $\mathcal{L}^D_n(\A)$ corresponds to minimizing the localized and truncated contrast function 
\begin{equation*}
    R_T(\A)\coloneq \frac{1}{T}\sum_{i=1}^n \| \Delta X_i - \Delta_n \A X_{t_{i-1}} \|^2 \1_B\left(X_{t_{i-1}}\right)\1_{\{\Vert \Delta X_i \Vert< \eta\}}.
\end{equation*}

The first reason for defining the pseudo-likelihood $\mathcal{L}^D_n$ is that the continuous martingale part is replaced by the increment of the original process $\X$ in the discretized negative log-likelihood function \eqref{eq: disc like untrunc}. This is only well-motivated if the BDLP $\Z$ is a Brownian motion since the continuous martingale part of $\X$ is then $\X$ itself. The classical approach to overcome this issue is to try to recover the continuous martingale part by jump filtering as applied for maximum likelihood estimation for scalar Lévy-driven OU processes in \cite{mai14}. More specifically, summands of the contrast function \eqref{eq: contrast untrunc} are excluded if the norm of the corresponding increment exceeds the threshold $\eta>0,$ with $\eta$ tending to $0$ as the distance between observations $\Delta_n$ tends to $0$. However, as explained in the introduction, this excludes many types of Lévy process, in particular pure jump processes. To filter out excessively large jumps of $\Z$ we still disregard excessively large increments, but instead of letting the truncation level $\eta$ tend to $0$ with $\Delta_n$ we let it grow together with the observation length $T$. This also leads to fewer observations being ignored, which is natural in practical applications. 

Another change we make to the contrast function \eqref{eq: contrast untrunc} is to disregard excessively large observations of $\X$. This is motivated by the fact that the form of the contrast function stems from the sub-Gaussianity of the continuous martingale part. In particular, it resembles the classical negative log-likelihood function for the mean of a normal distribution. It is well-known that such $L_2$-criteria are not suitable in heavy-tailed settings, such as the general Lévy-driven case, since they put too much emphasis on extreme outliers. 

 Based on the pseudo-likelihood $\mathcal{L}^D_n$ we define the Lasso estimator $\hat{\A}_{\mathsf{L}}$ with tuning parameter $\lambda_{\mathsf{L}}>0$ as
\begin{align}\label{def: lasso discrete}
    \hat{\A}_{\mathsf{L}} \in \operatorname{argmin}_{\A\in\R^{d\times d}} (\mathcal{L}^D_n(\A)+\lambda_{\mathsf{L}}\Vert \A\Vert_1),
\end{align}
and similarly introduce the Slope estimator $\hat{\A}_{\mathsf{S}}$ with tuning parameter $\lambda_{\mathsf{S}}>0$ as 
\begin{equation}\label{def: slope discrete}
	\hat{\A}_{\mathsf{S}}\in \operatorname{argmin}_{\A\in\R^{d\times d}}(\mathcal{L}^D_n(\A)+\lambda_{\mathsf{S}}\|\A\|_\star),
\end{equation}
where $\Vert \cdot\Vert_\star$ is defined in \eqref{def:normstar}.
Our analysis requires a suitably stable behavior of $\X,$ in particular the existence of an invariant distribution $\mu$. Theorems 4.1 and 4.2 in \cite{sato84} or Proposition 2.2 in \cite{mas04} show that $\X$ admits an invariant distribution $\mu$ as soon as $\A_0\in M_+(\R^d)$ and $\E[(1\vee\log(\Vert Z_1\Vert))] < \infty$. In case an invariant distribution $\mu$ exists we denote its truncated covariance matrix for a set $B\subseteq\R^d$ by \[\C_{\infty}(B) \coloneq \E [Y Y^\top \1_B(Y)],\quad \textrm{where}\quad Y\sim\mu,\]
and denote its true covariance matrix by $\C_{\infty}\coloneq \C_{\infty}(\R^d).$ Throughout the remainder of the paper we impose the following assumption on the considered OU process, which implies sufficient stability of $\X$ and identifiability of $\A_0$.

\begin{enumerate}[$(\mathcal{A}_0$)]
	\item \label{ass: ergodicity}
	 $\A_0\in M_+(\R^d)$, the BDLP $\Z$ is a square-integrable martingale and $\lambda_{\min}(\C_\infty(B))>0$. 
	Moreover, the initial distribution coincides with the invariant distribution $\mu$, i.e., $\pi=\mu,$ so that $\X$ is stationary.
\end{enumerate}
 Assumption \ref{ass: ergodicity} is well-posed by the aforementioned results of \cite{sato84,mas04}. In fact, \cite{mas04} even shows that $\X$ is exponentially ergodic as soon as $\A_0\in M_+(\R^d)$ and $\Z$ admits a $p$-th moment for some $p>0$, which is fulfilled under \ref{ass: ergodicity}. The stationarity assumption is therefore only for convenience and can be circumvented in practice by allowing a large enough burn-in, for more details see e.g.\ Lemma 4.2 in \cite{chris24}. Under \ref{ass: ergodicity} also the matrix 
 \begin{equation}\label{eq: def nu mat}
      \boldsymbol{\nu}_2\coloneq \int zz^\top\,\nu(\md z)
 \end{equation}
 which will be important in our statements, is well-defined. We want to emphasize that we assume positive definiteness of $\C_\infty(B)$ rather than $\C$. This assumption is weaker than requiring $\lambda_{\min}(\C)>0,$ since $\lambda_{\min}(\C)>0,$ implies $\lambda_{\min}(\C_\infty(B))>0$ for any $B$ with Lebesgue measure larger than $0$ by the results of \cite{sato84,mas04}. This condition also allows $\Z$ to be a pure jump process, contrasting the assumptions of \cite{dex24}. Said assumption acts as an identifiability condition for $\A_0$ since $\C_\infty(B)$ can be interpreted as the Fisher information in the given model.

Additionally to the tuning parameters and the truncation level $\eta,$ our estimators also require us to choose the set $B\subset \R^d$. For this we recall the so-called thin shell phenomenon for high-dimensional random variables, see e.g. Section 3.1 in \cite{vers18}. This states that the squared norm of high-dimensional random vectors concentrates strongly around the trace of their covariance matrix. Consequently, in order to have a sufficient amount of observations in the set $B,$ its radius $b$ should grow with the dimension like $\sqrt{d}.$ The canonical choice for this is $B=B(0,c\sqrt{d}),$ the open ball with radius $b=c\sqrt{d},$ where $c>0$ is a dimension independent constant.

For the presentation of our main results we also introduce the following quantities. For any stochastic process $(Y_t)_{t\in[0,T]},$ and $B\subseteq\R^d,\eta>0,  T>0,$  we define the empirical norms 
\begin{align*}
   \Vert Y \Vert^2_{l^2_n(B,X)} &\coloneq \frac{\Delta_n}{T} \sum_{i=1}^n \Vert Y_{t_{i-1}} \Vert^2 \1_B\left(X_{t_{i-1}}\right),
   \\
   \Vert Y \Vert^2_{l^2_n(B,X,\eta)} &\coloneq \frac{\Delta_n}{T} \sum_{i=1}^n \Vert Y_{t_{i-1}} \Vert^2 \1_B\left(X_{t_{i-1}}\right)\1_{\{\Vert \Delta X_i \Vert< \eta\}}.
\end{align*}
Moreover, we denote the truncated analogues of the empirical covariance matrix of $\X$ by
\begin{align*}
\hat{\C}_n(B) &\coloneq\frac{\Delta_n}{T} \sum_{i=1}^n X_{t_{i-1}} X_{t_{i-1}}^\top \1_B\left(X_{t_{i-1}}\right),
\\
\hat{\C}^\eta_n(B) &\coloneq\frac{\Delta_n}{T} \sum_{i=1}^n X_{t_{i-1}} X_{t_{i-1}}^\top \1_B\left(X_{t_{i-1}}\right)\1_{\{\Vert \Delta X_i \Vert< \eta\}}.
\end{align*}
Lastly, for $i\in[n]$, we also introduce
\begin{equation}\label{eq: def z tilde}
  \Delta\tilde{Z}_i\coloneq  \int_{t_{i-1}}^{t_{i}}\exp(-(t_{i}-u)\A_0)\d Z_u .
\end{equation}

\section{Main Results}\label{sec: results}
The main result of this paper is given by the following oracle inequalities for the Lasso and Slope estimators. Its proof together with all necessary auxiliary results can be found in Section \ref{sec: proof}. The statements require an observation horizon of length $T_\star(\ep_0,\eta),$ which acts as the sample complexity of the given problem. It is defined in detail in Section \ref{sec: proof} (see Proposition \ref{prop: qt}). Naturally, the sample complexity depends on the confidence level $\ep_0\in(0,1),$ but also on the truncation parameter $\eta>0$. Since it is beneficial to choose $\eta$ differently depending on the tails of the L{\'e}vy measure $\nu$, we choose to first give a general statement, suppressing the dependence on $\eta$. An explicit statement for the cases introduced in Section \ref{subsec: mat mod} can be found in Corollary \ref{cor: main}.
\begin{theorem}\label{thm: main}
Assume \ref{ass: ergodicity}. Let 
$\ep_0\in(0,1)$ be given, assume $b\geq1,\eta\geq 2b(\exp(\Delta_n\Vert\A_0\Vert)-1)$ and define
\[\gamma(\Delta_n)\coloneq\lambda_{\max}(\C+\boldsymbol{\nu}_2)(\lambda_{\max}(\C_\infty(B))\lor 1)\exp(\Delta_n\Vert\A_0\Vert),\]
where $\boldsymbol{\nu}_2$ is defined in \eqref{eq: def nu mat}. Then, there exists a universal constant $c_{\star}>0$ such that for any $T\geq T_\star(\ep_0,\eta)$ the following statements hold true.
\begin{enumerate}
     \item[a)] Let $\hat{\A}_{\mathsf{L}}$ be the Lasso estimator with tuning parameter $\lambda_{\mathsf{L}}$ satisfying
     \begin{equation}\label{eq:lassolambda}
           \lambda_{\mathsf{L}}\geq 2c_\star\Big(\frac{\log(2\e d^2/s)\gamma(\Delta_n)}{T}\Big)^{1/2}.  
     \end{equation}
 Then for any $\ep\in(0,1)$ and any $s$-sparse matrix $\A\in\R^{d\times d}$ the bound
    \begin{align*}
                 &\Vert(\hat{\A}_{\mathsf{L}}-\A_0)X\|_{l^2_n(B,X,\eta)}^2 +\lambda_{\mathsf{L}}\Vert\hat{\A}_{\mathsf{L}}-\A\Vert_1\\
    &\leq\|(\A-\A_0) X\|_{l^2_n(B,X,\eta)}^2 +3\Delta_n^2\lambda_{\max}(\C_\infty(B))\Vert\A_0\Vert_{2}^4\Big(\frac{1}{2}+\frac{1}{6}\exp(\Delta_n\Vert\A_0\Vert_{2})\Delta_n\Vert\A_0\Vert_{2}\Big)^2 
    \\
&\quad +\frac{3}{\Delta_n} \lambda_{\max}(\C+\boldsymbol{\nu}_2)\exp(\Delta_n\Vert\A_0\Vert)\P(\Vert \Delta  \tilde{Z}_1\Vert\geq \eta/2 )  
    \\&\quad+ \frac{48\lambda^2_{\mathsf{L}}}{\lambda_{\min}(\C_\infty(B))\P(\Vert\Delta\tilde{Z}_1\Vert\leq\eta/2)}\Bigg(s\lor\frac{\log(2/\ep)}{\log(2\e d^2/s)}\Bigg),
    \end{align*}
    holds true with probability larger than $1-\ep_0-\ep$.
    \item[b)] Let $\hat{\A}_{\mathsf{S}}$ be the Slope estimator with tuning parameter $\lambda_{\mathsf{S}}$ satisfying
    \begin{equation}\label{eq:slopelambda}
        \lambda_{\mathsf{S}}\geq 2c_\star\Big(\frac{\gamma(\Delta_n)}{T}\Big)^{1/2}. 
    \end{equation}
    Then for any $\ep\in(0,1)$ and any $s$-sparse matrix $\A\in\R^{d\times d}$ the bound
    \begin{align*}
                 &\Vert(\hat{\A}_{\mathsf{S}}-\A_0)X\|_{l^2_n(B,X,\eta)}^2 +\lambda_{\mathsf{S}}\Vert\hat{\A}_{\mathsf{S}}-\A\Vert_\star\\
    &\leq\|(\A-\A_0) X\|_{l^2_n(B,X,\eta)}^2 +3\Delta_n^2\lambda_{\max}(\C_\infty(B))\Vert\A_0\Vert_{2}^4\Big(\frac{1}{2}+\frac{1}{6}\exp(\Delta_n\Vert\A_0\Vert_{2})\Delta_n\Vert\A_0\Vert_{2}\Big)^2 
    \\
&\quad +\frac{3}{\Delta_n} \lambda_{\max}(\C+\boldsymbol{\nu}_2)\exp(\Delta_n\Vert\A_0\Vert)\P(\Vert \Delta  \tilde{Z}_1\Vert\geq \eta/2 )  
    \\&\quad+ \frac{48\lambda^2_{\mathsf{S}}}{\lambda_{\min}(\C_\infty(B))\P(\Vert\Delta\tilde{Z}_1\Vert\leq\eta/2)}\Bigg(s\log\Big(\frac{2\e d^2}{s}\Big)\lor\log(2/\ep)\Bigg),
    \end{align*}
    holds true with probability larger than $1-\ep_0-\ep$.
\end{enumerate}
\end{theorem}
Before we continue with stating the upper bounds under different assumptions on the tails of $\nu$, we discuss several aspects of the above result.
\begin{remark}\label{rem:1}
    \begin{enumerate}[a)]
    \item Both of the above statements are sharp oracle inequalities in the empirical norm $\|\cdot\|_{l^2_n(B,X,\eta)}$. This mirrors Theorems 4.2 and 6.1 in \cite{belets18}, where Lasso and Slope estimators in the linear regression setting are investigated.
    \item  We see that the error decomposes into four components.
    \begin{itemize}
    \item A bias term, which measures how close the true drift matrix $\A_0$ is to a $s$-sparse matrix $\A$. Naturally, this term disappears if $\A_0$ itself is $s$-sparse.
        \item The discretization error, which is governed by the magnitude of the inter observation distance $\Delta_n$. For $\Delta_n\lesssim \Vert \A_0\Vert_0^{-1/2},$ this error is of the order $\Delta_n^2\Vert \A_0\Vert_2^4\lesssim \Delta_n^2\Vert \A_0\Vert_0^2$. In contrast, Corollary 4 in \cite{amorino2025sampling}, which concerns the Lasso estimator for continuous OU processes, shows a discretization error of the order $\Delta_nsd^4\log(d)$. This result in particular assumes $\Delta_n\lesssim s^{-2}$ implicitly, see the proof of Proposition 6.9, in particular the reference to the derivation of equation (49) in \cite{amorino2025sampling}. The reason for this difference, is the different approach to the estimation problem. Whereas we investigate L{\'e}vy-driven OU processes, \cite{amorino2025sampling} focuses on continuous diffusion processes with a general parametric structure in the drift function. To bound the discretization error in this setting, the authors obtain a high probability bound through Talagrand's generic chaining device, and apply this result to continuous OU processes. We, on the other hand, directly exploit the fact that L{\'e}vy-driven OU processes admit an explicit solution for bounding the discretization error by a deterministic bound through a Taylor approximation (see Proposition \ref{prop: discretizationdex}). 
        \item The truncation error, which depends directly on the tails of the L{\'e}vy measure $\nu$ and the truncation level $\eta$. In the following, we will choose $\eta$ such that it diverges as $T$ tends to infinity. Since $\Delta\tilde{Z}_1$ is a stochastic integral of length $\Delta_n,$ see \eqref{eq: def z tilde}, this leads to a sufficiently fast decrease of this part of the error. 
        \item The stochastic error, which is of the form $s\log(2\e d^2/s)/T,$ for suitably chosen tuning parameters. This convergence rate is known to be minimax optimal for sparse linear regression with $d^2$ parameters (see Theorem 7.1 in \cite{belets18}) and L{\'e}vy-driven OU processes 
        if continuous observations are available (see Theorem 2.7 in \cite{dex24}). The difference to the commonly obtained rate of convergence $s\log(d^2)/T$ (see e.g. \cite{amorino2025sampling,wong20}) lies in the logarithmic term. Although this term generally has a minimal impact, the more refined bound shows that the Lasso and Slope estimators can also achieve the minimax optimal rate for non-sparse estimation problems, i.e.\ if $s=d^2$.
    \end{itemize}
    \item The choice of the tuning parameters for the estimators only depends on $\Delta_n$ through the term $\exp(\Delta_n\Vert \A_0\Vert),$ which is negligible in a high-frequency framework. This is significantly different to \cite{amorino2025sampling}, where $\lambda$ has to be chosen larger than $ d^2\Delta_n^{3/2}\sqrt{\log(d^2)+\log(1/\ep)},$ see the definition of $\lambda_2^{OU}$ in \cite{amorino2025sampling}. The reason for this is the different approach to bounding the discretization error, which has already been discussed in point b). Moreover, the optimal choice of $\lambda$ for the Slope estimator is completely independent of the dimension and sparsity, whereas the optimal tuning parameter for the Lasso estimator depends on the a priori unknown sparsity $s$. This is also the case in \cite{belets18} and \cite{dex24}. However, since the ergodicity assumption \ref{ass: ergodicity} implies $s\geq d,$ one can always choose the logarithmic term as $\log(2\e d),$ which then leads to the commonly obtained rate of convergence $s\log(d)/T,$ for the Lasso estimator (see point b)). Therefore, the optimal choices of $\lambda$ essentially only depend on the observation horizon $T$ and are in particular completely independent of the confidence level $\ep_0+\ep$, which is for example not the case in \cite{amorino2025sampling,wong20}. Contrary to \cite{wong20}, the tuning parameters are also independent of the unknown mixing coefficients.
    \end{enumerate}
\end{remark}
In order to ease the presentation of the next result, we introduce the following condition on the truncation level $\eta$.
\begin{enumerate}[$(\mathcal{N})$]
	\item \label{ass: eta}
	There exists a constant $c_\eta>0,$ such that
    \[\P(\Vert \Delta\tilde{Z}_1\Vert\geq \eta/2)\leq c_\eta\frac{d\Delta_n}{T}. \]
\end{enumerate}
The truncation error in Theorem \ref{thm: main} is of the order $d/T,$ as soon as $\eta$ satisfies condition \ref{ass: eta}. In this case it is negligible compared to the stochastic error of order $s\log(2\e d^2/s)/T,$ since the ergodicity assumption \ref{ass: ergodicity} implies $s\geq d.$ Sufficient conditions on the magnitude of $\eta$ in order to ensure that assumption \ref{ass: eta} is satisfied are given in Proposition \ref{prop: eta} for the case where $\Z$ is continuous; has bounded jumps; is sub-Weibull; admits a polynomial moment. We summarize these in Table \ref{table:1}, together with their implications on the sample complexity $T_\star$. Not to overburden notation, we introduce for $\ep\in(0,1),$
\[T_{\star,0}=T_{\star,0}(\ep)\coloneq d\log(d/\ep)^2\lor \frac{d}{c_{\mathrm{mix},2}}\log\Big(\frac{c_{\mathrm{mix},1}d}{c_{\mathrm{mix},2}\ep}\Big)^2,\] where $c_{\mathrm{mix},1},c_{\mathrm{mix},2}>0$ are the constants appearing in the $\beta$-mixing coefficients, see \eqref{eq: mixing}. Explicit expressions for these constants are generally hard to obtain. However for the case where $\Z$ is continuous, we show in Proposition \ref{prop: mix}, that $c_{\mathrm{mix},1}$ is of order $d,$ whereas $c_{\mathrm{mix},2}$ only depends on the eigenvalues of $\C$ and $\C_\infty$. Hence, in this case $T_{\star,0}$ is of the order $d\log(d/\ep)^2$. Furthermore we only focus on the case of small confidence levels $\ep$ and omit terms without a direct dimensional dependence in Table \ref{table:1}; the explicit expressions can be found in Section \ref{sec: proof} and Proposition \ref{prop: eta}. We discuss these results in detail in Remark \ref{rem:cor}.
\begin{table}[h!]
    \centering
 \begin{tabular}{l|l|l}
  &Minimal order of $\eta$   & Minimal order of $T_{\star}(\ep,\eta)$
  \\ \hline
  $\Z$ is continuous   &$\sqrt{\log(T)d\Delta_n}$&$d^2\Delta_n\log(T)\log(1/\ep)\lor T_{\star,0} $
  \\
    $\Z$ has bounded jumps  &$\sqrt{d}\log(T)$&$d^2\log(T)\log(1/\ep)\lor T_{\star,0}$
    \\
      $\Z$ is $\alpha$ sub-Weibull   &$\sqrt{d}\log(T)^{1+1/\alpha}$&$d^2\log(T)^{2+2/\alpha}\log(1/\ep)\lor T_{\star,0}$
      \\
        $\Z$ admits a $p$-th moment  &$T^{1/p}d^{1/2-1/p}$&$d^{2-2/p}T^{2/p}\log(1/\ep)\lor T_{\star,0}$
\end{tabular}
    \caption{Minimal orders of the truncation level $\eta$ to satisfy condition \ref{ass: eta} under different assumptions on the BDLP $\Z$. Additionally, the implications on the sample complexity $T_\star$ are added.}
    \label{table:1}
\end{table}

We are now ready to state the next result, which in particular gives a bound for the error of the investigated estimators in the Frobenius norm.
\begin{corollary}\label{cor: main}
    Let everything be given as in Theorem \ref{thm: main}. Assume additionally that $\eta$ satisfies condition \ref{ass: eta}, that $T\geq 2c_\eta d\Delta_n$ and that $\A_0$ is $s$-sparse. Then there exists a universal constant $c>0,$ such that the following statements hold true.
    \begin{enumerate}[a)]
        \item For any $\ep\in(0,1)$ the bound
            \begin{align*}
                 &\Vert(\hat{\A}_{\mathsf{L}}-\A_0)X\|_{2}^2 +\frac{\lambda_{\mathsf{L}}}{\lambda_{\min}(\C_\infty(B))}\Vert\hat{\A}_{\mathsf{L}}-\A\Vert_1\\
    &\leq c\Delta_n^2\kappa(\C_\infty(B))\Vert\A_0\Vert_{2}^4\Big(1+\exp(\Delta_n\Vert\A_0\Vert_{2})\Delta_n\Vert\A_0\Vert_{2}\Big)^2 
    \\&\quad+ c\frac{\lambda^2_{\mathsf{L}}}{(\lambda_{\min}(\C_\infty(B))\land 1)^2}\Bigg(s\lor\frac{\log(2/\ep)}{\log(2\e d^2/s)}\Bigg),
    \end{align*}
     holds true with probability larger than $1-\ep_0-\ep$.
                \item For any $\ep\in(0,1)$ the bound
                 \begin{align*}
                       &\Vert(\hat{\A}_{\mathsf{S}}-\A_0)X\|_{2}^2 +\frac{\lambda_{\mathsf{S}}}{\lambda_{\min}(\C_\infty(B))}\Vert\hat{\A}_{\mathsf{S}}-\A_0\Vert_1
                 \\
    &\leq c\Delta_n^2\kappa(\C_\infty(B))\Vert\A_0\Vert_{2}^4\Big(1+\exp(\Delta_n\Vert\A_0\Vert_{2})\Delta_n\Vert\A_0\Vert_{2}\Big)^2 
    \\
 &\quad+c\frac{\lambda_{\mathsf{S}}^2}{(\lambda_{\min}(\C_\infty(B))\land 1)^2}\Bigg(s\log\Big(\frac{2\e d^2}{s}\Big)\lor\log(2/\ep)\Bigg),
    \end{align*}
    holds true with probability larger than $1-\ep_0-\ep$.
    \end{enumerate}
\end{corollary}
    The proof of Corollary \ref{cor: main} is given in Section \ref{sec: proof}. We close this section with a brief discussion of the above result and the sample complexities depicted in Table \ref{table:1}.
    \begin{remark}\label{rem:cor}
        \begin{enumerate}[a)]
            \item If the tuning parameters $\lambda_{\mathsf{L}}$ and $\lambda_{\mathsf{S}}$ are chosen such that they satisfy inequalities \eqref{eq:lassolambda}, respectively \eqref{eq:slopelambda}, with equality, and $\Delta_n\Vert \A_0\Vert_2\in O(1)$, the upper bounds of Corollary \ref{cor: main} are essentially given by 
            \[  \Delta_n^2d^2\Vert \A_0\Vert_2^4+\frac{s\log(2\e d^2/s)}{T}\leq   \Delta_n^2(d^2\Vert\A_0\Vert^4\land s^2\Vert \A_0\Vert_\infty^4)+\frac{s\log(2\e d^2/s)}{T}.\]
          We see that there are two different regimes. Ignoring the logarithmic term, we obtain that the discretization error dominates asymptotically if $s\Delta_n^2T\to \infty,$ whereas the stochastic error dominates if $s\Delta_n^2T\to 0$. This indicates that the same minimax optimal rate of convergence as for continuous observations, see \cite{dex24}, is achieved as soon as $\Delta_n\in o((sT)^{-1/2})$. The analogous sufficient condition obtained in Corollary 4 of \cite{amorino2025sampling} for continuous OU processes is of the form $\Delta_n\in o((d^4T)^{-1}),$ which is more restrictive. 
            \item The sample complexity required in Corollary \ref{cor: main} is given as the maximum of two terms: One which resembles the assumptions of Theorem \ref{thm: main} given by $T_{\star,0}$ and one required for condition \ref{ass: eta} to be satisfied. 
            The latter naturally depends on the tails of the L{\'e}vy-measure $\nu,$ see Table \ref{table:1}. In the case where $\Z$ is continuous, the sample complexity $T_\star$ is up to logarithmic factors equal to $T_{\star,0}$ as soon as $\Delta_n\in O(1/d)$. The term $T_{\star,0}$ is up to logarithmic factors identical to the sample complexity required for continuous BDLPs $\Z$ under continuous observations (see Section 3.1 in \cite{dex24}). We therefore recover the result of \cite{dex24} as $\Delta_n\to 0$. However, if the BDLP $\Z$ is not continuous, $T_{\star}$ can be dominated by the sample complexity required through condition \ref{ass: eta}, depending on the magnitude of the confidence level. In the sub-Weibull case we then require an observation horizon of length $\gtrsim d^2,$ up to logarithmic terms. In contrast, \cite{wong20}, where the Lasso estimator for exponentially $\beta$-mixing time series with sub-Weibull noise with parameter $\alpha\in(0,2)$ is investigated, derive a sample complexity of order $s^{1+2/\alpha},$ (see Corollary 9 in \cite{wong20} together with the definition of $\gamma$ therein), which strictly dominates $d^2$ if $\alpha\in(0,2)$ for the given problem due to assumption \ref{ass: ergodicity}. 
        \end{enumerate}
    \end{remark}
\section{Proof of the Main Result}\label{sec: proof}
Throughout this whole section we will assume that \ref{ass: ergodicity} holds true without further mention.
The first step in proving the main results of this paper is the following inequality for penalized estimators. It relies on deterministic arguments from convex analysis and is an adaptation of Lemma A.2 in \cite{belets18}. In high-dimensional statistics, such inequalities are often referred to as ``basic inequality''. 
\begin{lemma}\label{lemma: L2 frobenius disc}

Let $h\colon \R^{d\times d} \to \R$ be a convex function
and $\hat{\A}$ be a solution to the minimization problem $\min_{\A\in\R^{d\times d}}\left(\mathcal L^D_n(\A)+h(\A)\right)$. Then $\hat{\A}$ satisfies for all $\A\in\R^{d\times d},$
 \begin{align*}
    &\Vert(\hat{\A}-\A_0)X\|_{l^2_n(B,X,\eta)}^2 \\
    &\leq\|(\A-\A_0) X\|_{l^2_n(B,X,\eta)}^2-\|(\hat{\A}-\A)X\|_{l^2_n(B,X,\eta)}^2 + 2 \left(h(\A)-h(\hat{\A}) \right) \\ 
    & \quad + \frac{2}{T} \left(  \sum_{i=1}^n \left((\A - \hat{\A}) X_{t_{i-1}}\right)^\top (\exp(-\Delta_n\A_0)-\mathbb{I}_d+\Delta_n\A_0)X_{t_{i-1}} \1_{\{\Vert \Delta X_i \Vert< \eta\}} \1_B \left( X_{t_{i-1}} \right)\right.
    \\
    & \quad + \left. \sum_{i=1}^n \left((\A - \hat{\A}) X_{t_{i-1}}\right)^\top \Delta \tilde{Z}_i \1_{\{\Vert \Delta X_i \Vert< \eta\}} \1_B \left( X_{t_{i-1}} \right) \right),
\end{align*}
where
$\Delta \tilde{Z}_i$ is defined in \eqref{eq: def z tilde}.
\end{lemma}

\begin{proof}[Proof of Lemma \ref{lemma: L2 frobenius disc}]
     We start by computing the gradient of $\mathcal{L}^D_n$. For this, we first note that Lemma \ref{lemma: basic ou} implies $\Delta X_i = (\exp(-\Delta_n\A_0)-\mathbb{I}_d)X_{t_{i-1}} + \Delta \tilde{Z}_i $. Therefore, we can rewrite
 \begin{align*}
    \mathcal{L}^D_n(\A) = & \frac{1}{T} \sum_{i=1}^n \left(\A X_{t_{i-1}}\right)^\top  (\exp(-\Delta_n\A_0)-\mathbb{I}_d+\Delta_n\A_0)X_{t_{i-1}} \1_{\{\Vert \Delta X_i \Vert< \eta\}} \\
     & +  \frac{1}{T} \sum_{i=1}^n \left(\A X_{t_{i-1}}\right)^\top \Delta \tilde{Z}_i \1_B\left(X_{t_{i-1}}\right) \1_{\{\Vert \Delta X_i \Vert< \eta\}} \\
     & -  \frac{\Delta_n}{T} \sum_{i=1}^n \left(\A X_{t_{i-1}}\right)^\top \A_0  X_{t_{i-1}} \1_B\left(X_{t_{i-1}}\right) \1_{\{\Vert \Delta X_i \Vert< \eta\}}  
       +  \frac{\Delta_n}{2T}\sum_{i=1}^n \Vert \A X_{t_{i-1}}\Vert^2 \1_B\left(X_{t_{i-1}}\right)\1_{\{\Vert \Delta X_i \Vert< \eta\}} \\
 = & \tr\left(\A\ep_n(\A_0, B, \eta)^\top - \A \hat{\C}^\eta_n(B)\A_0^\top +\frac{1}{2} \A \hat{\C}^\eta_n(B)\A^\top \right),
\end{align*}
 where
\begin{align*}
     \ep_n(\A_0,B,\eta)^\top \coloneq \frac{1}{T}  
     & \left( \sum_{i=1}^n X_{t_{i-1}}  ((\exp(-\Delta_n\A_0)-\mathbb{I}_d+\Delta_n\A_0)X_{t_{i-1}})^\top \1_{\{\Vert \Delta X_i \Vert< \eta\}} \1_B \left( X_{t_{i-1}} \right)\right.\\
     & \left. + \sum_{i=1}^n X_{t_{i-1}}
     \left( \Delta \tilde{Z}_i\right)^\top \1_{\{\Vert \Delta X_i \Vert< \eta\}} \1_B \left( X_{t_{i-1}} \right)  \right).
 \end{align*}
    Through the above expression a straightforward computation gives \[ \nabla \mathcal{L}^D_n(\A) = (\A-\A_0)\hat{\C}_n^{\eta}(B)+\ep_n(\A_0, B, \eta).\]
    For notational simplicity, we now define the function $ f\equiv \mathcal{L}^D_n+h$.
	Since $f$ is convex, it follows that $\mathbf{0}$ is in the subdifferential of $f$ at $\hat{\A}$.
	The Moreau--Rockafellar theorem then provides the existence of a matrix $\mathbf B$ in the subdifferential of $h$ at  $\hat{\A}$ such that $$\mathbf{0}=(\hat{\A}-\A_0)\hat{\C}_n^{\eta}+ \ep_n(\A_0, B, \eta) +\mathbf{B}.$$ 
    Then it holds $\tr( \mathbf{B}^\top( \A-\hat{\A}))\leq h(\A)-h(\hat{\A}),$  since the trace is the inner product
    associated with the Frobenius norm $\|\cdot\|_2$. 
	Consequently,
	\begin{align*}
		&\Vert(\hat{\A}-\A_0)X\|_{l^2_n(B,X,\eta)}^2-\|(\A-\A_0) X\|_{l^2_n(B,X,\eta)}^2+\|(\hat{\A}-\A)X\|_{l^2_n(B,X,\eta)}^2
		\\
		& =\tr\left(\left[(\hat{\A}-\A_0)^\top(\hat{\A}-\A_0)-(\A-\A_0)^\top(\A-\A_0)+(\hat{\A}-\A)^\top(\hat{\A}-\A)\right] \hat{\C}^\eta_n(B)\right)
		\\
		& =2\tr\left( (\hat{\A}-\A_0)^\top(\hat{\A}-\A)\hat{\C}^\eta_n(B)\right)
		\\
		& =2\tr\left((\hat{\A}-\A )^\top(\hat{\A}-\A_0)\hat{\C}^\eta_n(B)\right)
		\\
		& =2\tr\left((\A-\hat{\A})^\top(\ep_n\left(\A_0, B, \eta \right)+\mathbf{B})\right)
		\\
		& \leq 2 \left (\tr \left((\A-\hat{\A})^\top\ep_n(\A_0, B, \eta)\right)+h(\A)-h(\hat{\A})\right ),
	\end{align*}
 concluding the proof.
\end{proof} 
The basic bound provided by Lemma \ref{lemma: L2 frobenius disc} shows that in order to control the error of the estimators, we have to bound the following error terms
 \begin{align}
 \begin{split}\label{eq:def errs}
     &\mathcal{D}(\M) \coloneq \sum_{i=1}^n \left(\M X_{t_{i-1}}\right)^\top (\exp(-\Delta_n\A_0)-\mathbb{I}_d+\Delta_n\A_0)X_{t_{i-1}} \1_{\{\Vert \Delta X_i \Vert< \eta\}} \1_B \left( X_{t_{i-1}} \right)\\
     &\mathcal{S}(\M) \coloneq \sum_{i=1}^n \left(\M X_{t_{i-1}}\right)^\top \Delta \tilde{Z}_i \1_{\{\Vert \Delta X_i \Vert< \eta\}} \1_B \left( X_{t_{i-1}} \right),    
     \end{split}
 \end{align}
 where we recall that
 $$\Delta \tilde{Z}_i=\int_{t_{i-1}}^{t_i}\e^{-(t_i-u)\A_0}\d Z_u.$$
 In the following we will refer to $\mathcal{D}$ as ``discretization error'' since it tends to $0$ as the inter-observation distance $\Delta_n$ tends to $0$. We call $\mathcal{S}$ the ``stochastic error,'' which can be seen as discrete time analogue of the stochastic error observed under continuous observations, see \cite{dex24}. In particular, the stochastic error is essentially given as a martingale transform, which is the discrete counterpart of the It{\^o} integral $\ep_T$ observed for continuous observations in Lemma 2.1 of \cite{dex24}.
 \subsection{Discretization Error}
 Using a Taylor expansion together with classical inequalities, we obtain the following deterministic bound on the discretization error.

\begin{proposition}\label{prop: discretizationdex}
For any $\M\in\R^{d\times d},c>0,$ the following bound for the discretization error holds true 
\begin{align*}
\frac{2}{T}\mathcal{D}(\M)
&\leq\frac{1}{c} \Vert\M X\Vert_{l_n^2(B,X,\eta)}^2+c\Delta_n^2\Vert \hat{\C}^\eta_n(B)\Vert\Vert\A_0\Vert_{2}^4\Big(\frac{1}{2}+\frac{1}{6}\exp(\Delta_n\Vert\A_0\Vert_{2})\Delta_n\Vert\A_0\Vert_{2}\Big)^2 .
\end{align*}
\end{proposition}
\begin{proof}
For the remainder of the proof, we denote
$$\Rr_n \coloneq \Delta_n\A_0-\mathbb{I}_d+\exp(-\Delta_n\A_0).$$

Applying the Cauchy--Schwarz inequality together with Young's inequality, we get that for any~$c>0$,
\begin{align}
\begin{split}\label{eq:discerr}
\frac{2}{T}\mathcal{D}(\M)
&=
\frac{2}{T}\sum_{i=1}^n \left(\M X_{t_{i-1}}\right)^\top \Rr_nX_{t_{i-1}} \1_{\{\Vert \Delta X_i \Vert< \eta\}} \1_B \left( X_{t_{i-1}} \right)
\\
&\leq \frac{2}{T}\sum_{i=1}^n \Vert\M X_{t_{i-1}}\Vert \Vert\Rr_nX_{t_{i-1}}\Vert \1_{\{\Vert \Delta X_i \Vert< \eta\}} \1_B \left( X_{t_{i-1}} \right)
\\
&\leq\frac{\Delta_n}{Tc} \sum_{i=1}^n \Vert\M X_{t_{i-1}}\Vert^2\1_B \left( X_{t_{i-1}} \right)\1_{\{\Vert \Delta X_i \Vert< \eta\}}+\frac{c}{T\Delta_n}\sum_{i=1}^n\Vert\Rr_nX_{t_{i-1}}\Vert^2  \1_B \left( X_{t_{i-1}} \right)\1_{\{\Vert \Delta X_i \Vert< \eta\}}
\\
&=\frac{1}{c} \Vert\M X\Vert_{l_n^2(B,X,\eta)}^2+\frac{c}{T\Delta_n}\sum_{i=1}^n\tr(X_{t_{i-1}}^\top \Rr^\top_n\Rr_nX_{t_{i-1}})1_B \left( X_{t_{i-1}} \right)\1_{\{\Vert \Delta X_i \Vert< \eta\}}
\\
&=\frac{1}{c} \Vert\M X\Vert_{l_n^2(B,X,\eta)}^2+\frac{c}{\Delta^2_n}\tr(\Rr_n\hat{\C}^\eta_n(B)\Rr_n^\top)
\\
&\leq \frac{1}{c} \Vert\M X\Vert_{l_n^2(B,X,\eta)}^2+\frac{c}{\Delta^2_n}\Vert \hat{\C}^\eta_n(B)\Vert\Vert\Rr_n\Vert^2_2
\end{split}
\end{align}
where we also used the cyclic property of the trace operator together with the well-known inequality $\tr(\A\C\A^\top)\leq \Vert \C\Vert\Vert \A\Vert^2_2,$ for any $\A,\C\in\R^{d\times d},$ such that $\C$ is positive semidefinite. Now applying a Taylor approximation together with the triangle inequality gives
\begin{align*}
    \Vert \Rr_n\Vert_2&=\Vert \Delta_n\A_0-\mathbb{I}_d+\exp(-\Delta_n\A_0)\Vert_{2}
    \\
    &=\Big\Vert \sum_{k=2}^\infty\frac{(-\Delta_n\A_0)^k}{k!}\Big\Vert_{2}
       \\
    &\leq \sum_{k=2}^\infty\frac{\Vert\Delta_n\A_0\Vert_{2}^k}{k!}
    \\
    &=\exp(\Delta_n\Vert\A_0\Vert_{2})-1-\Delta_n\Vert\A_0\Vert_{2}
    \\
    &\leq \frac{1}{2}(\Delta_n\Vert\A_0\Vert_{2})^2+\frac{1}{6}\exp(\Delta_n\Vert\A_0\Vert_{2})(\Delta_n\Vert\A_0\Vert_{2})^3
    \\
    &=(\Delta_n\Vert\A_0\Vert_{2})^2\Big(\frac{1}{2}+\frac{1}{6}\exp(\Delta_n\Vert\A_0\Vert_{2})\Delta_n\Vert\A_0\Vert_{2}\Big).
\end{align*}
Plugging this bound into \eqref{eq:discerr} completes the proof.
\end{proof}
We see that the upper bound in Proposition \ref{prop: discretizationdex} depends on the operator norm of the truncated empirical covariance matrix $\hat{\C}^\eta_n(B).$ In order to control this term we require a sufficiently strong concentration of $\hat{\C}^\eta_n(B).$ This is the subject of the next section.
\subsection{Concentration of the Empirical Covariance Matrix}\label{subsec: cov}
Proposition \ref{prop: discretizationdex} implies that a sufficiently strong concentration of the empirical covariance matrix $\hat{\C}^\eta_n(B)$ suffices to control the discretization error. Such a concentration behavior is often required for the theoretical analysis of the Lasso and Slope estimators in the form of so-called restricted eigenvalue conditions (see e.g. Section 3.1 in \cite{cmp20} and \cite{dex24}). For continuous OU processes \cite{cmp20} obtained concentration of the empirical covariance matrix through Malliavin calculus (see Proposition 3.2 in \cite{cmp20}) and \cite{amorino2025sampling} achieves similar results through the same technique (see Proposition 3.3 in \cite{amorino2025sampling}). 

However, for discontinuous OU processes, this approach is not available. In \cite{dex24} this issue was addressed by assuming a suitable behavior of the empirical covariance matrix (see assumption $(\mathcal{H})$ in \cite{dex24}). This assumption was also verified by an application of Markov's inequality. Unfortunately, this result leads to an extremely large sample complexity that grows exponentially with the dimension. 

We therefore follow another approach which relies on the fact that the L{\'e}vy-driven OU process $\X$ is exponentially $\beta$-mixing as soon as the L{\'e}vy-measure $\nu$ admits a polynomial moment (see \cite{mas04}). Since we assume $\nu$ to admit a second moment (see \ref{ass: ergodicity}), there then exist constants $c_{\mix,1},c_{\mix,2}>0,$ such that the $\beta$-mixing coefficients of $\X,$ denoted by $\beta_{\X}$ are bounded by
\begin{equation}\label{eq: mixing}
\beta_{\X}(t)\leq c_{\mix,1}\exp(-c_{\mix,2}t), 
\end{equation}
for all $t\geq0$. Explicit values for the constants $c_{\mix,1},c_{\mix,2}$ are generally hard to obtain. We provide upper bounds for the continuous case in Proposition \ref{prop: mix}.

Using a Berbee coupling as in the proof of Theorem 3.1 in \cite{dex22}, the mixing property allows us to approximate the empirical covariance matrix by a sum of independent random variables. Since we employ a truncated version of the likelihood function, these random variables are also bounded. We can thus apply the classical Bernstein inequality to the independent approximation of the empirical covariance matrix, which leads to the following result.

\begin{proposition}\label{prop: REP disc}
For any $u>0,m_T\in(\Delta_n,T/6),$ it holds
$$\P(\Vert\hat{\C}_n(B)- \C_\infty(B)\Vert>u ) \leq4d\exp\Big(-\frac{T u^2}{4b^2m_T(6\lambda_{\max}(\C_\infty(B))+u)}\Big) +2c_{\mix,1}\frac{T}{m_T}\exp(-c_{\mix,2}m_T). $$
\end{proposition}

\begin{proof}
We first assume that $m_T$ is such that  $n_T\coloneq (T+\Delta_n)/(2m_T)\in\N.$ For $j\in[n_T]$ we define 
\begin{align}
\begin{split}
    \label{eq: blocks disc}
      X^{j,1}&\coloneq (X_{t_i})_{t_i\in [2(j-1)m_T,(2j-1)m_T)\cap\T},\quad X^{j,2}\coloneq (X_{t_i})_{t_i\in[(2j-1)m_T,2jm_T)\cap\T},  
      \\
      n^{j,1}&\coloneq \vert [2(j-1)m_T,(2j-1)m_T)\cap\T\vert,\quad n^{j,2}\coloneq \vert [(2j-1)m_T,2jm_T)\cap\T\vert,       
\end{split}
\end{align}
where $\T \coloneq \{i\Delta_n: i\in[n]\}$ denotes the observation times. 
Then $X^{j,1}$ is a measurable function of $X_{[2(j-1)m_T,(2j-1)m_T)},$ for all $j\in[n_T],$ and an analogous result holds true for $X^{j,2}$. Hence, we can argue as in the proof of Theorem 3.1 in \cite{dex22}, which allows us to construct random variables $\hat{X}^{j,k},k\in[2],j\in[n_T],$ such that
\begin{equation}\label{eq:mix disc}
\begin{array}{l}
    (1) \quad X^{j,k}\overset{(d)}{=}\hat{X}^{j,k}, \quad \text{for all } j\in[n_T], \\[6pt]
    (2) \quad \P(X^{j,k}\neq \hat{X}^{j,k})\leq \beta_{\X}(m_T), \quad \text{for all } j\in[n_T], \\[6pt]
    (3) \quad \hat{X}^{1,k},\ldots,\hat{X}^{n_T,k} \text{ are independent}.
\end{array}
\end{equation}
for $k\in[2]$.

We continue by splitting the empirical covariance matrix into the parts contributed by $X^{j,1}$ and $X^{j,2}$. More precisely, for $j\in[n_T],$ we define
\[S(X^{j,1})\coloneq \sum_{t_i\in [2(j-1)m_T,(2j-1)m_T)\cap\T} X_{t_i}X_{t_i}^\top \1_{B}(X_{t_i}), 
\quad
S(X^{j,2})\coloneq  \sum_{t_i\in[(2j-1)m_T,2jm_T)\cap\T} X_{t_i}X_{t_i}^\top \1_{B}(X_{t_i}), \]
and analogously we define $S(\hat{X}^{j,k})$ for $j\in[n_T],k\in[2]$. Then it holds by construction that $\hat{\C}_n(B) = \frac{1}{n} \sum_{j=1}^{n_T} (S(X^{j, 1})+ S(X^{j, 2})).$ Additionally, the properties in \eqref{eq:mix disc} imply
\begin{align}
\begin{split}\label{eq:berndis}
   & \P(\Vert\hat{\C}_n(B)- \C_\infty(B)\Vert>u ) 
  \\
  &\leq \P\Bigg(\Vert\frac{1}{n}\sum_{j=1}^{n_T} (S(X^{j, 1})- n^{j,1}\C_\infty(B)\Vert)+\Vert\frac{1}{n} \sum_{j=1}^{n_T} (S(X^{j, 2})-n^{j,2}\C_\infty(B)\Vert)>u \Bigg) 
    \\
  &\leq \P\Bigg(\Vert\frac{1}{n}\sum_{j=1}^{n_T} (S(X^{j, 1})-n^{j,1}\C_\infty(B))\Vert>\frac u 2\Bigg) +\P\Bigg(\Vert\frac{1}{n} \sum_{j=1}^{n_T} (S(X^{j, 2})-n^{j,2}\C_\infty(B))\Vert>\frac u 2 \Bigg) 
      \\
  &\leq \P\Bigg(\Vert\frac{1}{n} \sum_{j=1}^{n_T} (S(\hat{X}^{j, 1})-n^{j,1} \C_\infty(B))\Vert>\frac u 2\Bigg) +\P\Bigg(\Vert\frac{1}{n}\sum_{j=1}^{n_T} (S(\hat{X}^{j, 2})-n^{j,2} \C_\infty(B))\Vert>\frac u 2 \Bigg) 
  \\
  &\quad +\P(\exists j\in[n_T]:X^{j,1}\neq \hat{X}^{j,1})+\P(\exists j\in[n_T]:X^{j,2}\neq \hat{X}^{j,2})
        \\
  &\leq \P\Bigg(\Vert\sum_{j=1}^{n_T}\Big(S(\hat{X}^{j, 1})-n^{j,1}\C_\infty(B)\Big)\Vert>\frac{nu}{2}\Bigg) +\P\Bigg(\Vert\sum_{j=1}^{n_T}\Big( S(\hat{X}^{j, 2})- n^{j,2}\C_\infty(B)\Big)\Vert>\frac{nu}{2} \Bigg) 
  \\
  &\quad +2\frac{T}{m_T}\beta_X(m_T),   
\end{split}
\end{align}
 where we applied a union bound in the last step together with the fact that $n_T=(T+\Delta_n)/(2m_T)\leq T/m_T$.
Now, by definition it holds that $\E[S(\hat{X}^{j,k})]=n^{j,k}\C_\infty(B),$ for all $j\in[\nu_n],k\in[2].$ Additionally we note that $n^{j,k}\leq\lceil m_T/\Delta_n\rceil\leq 2m_T/\Delta_n,$ since $m_T\geq \Delta_n,$ and hence
\begin{align*}
\Vert  S(\hat{X}^{j, k})- n^{j,k}\C_\infty(B)\Vert&\leq 2n^{j,k}b^2
\\
&\leq 4\frac{m_T}{\Delta_n}b^2
\end{align*}
and since $\Big( S(\hat{X}^{j, k})- n^{j,k}\C_\infty(B)\Big)^2$ is a positive semidefinite matrix, the Courant--Fischer theorem gives 
\begin{align*}
    &\Vert \sum_{j=1}^{n_T}\E\Big[\Big( S(\hat{X}^{j, 1})- n^{j,1}\C_\infty(B)\Big)^2\Big]\Vert
    \\
    &\leq \sum_{j=1}^{n_T}\max_{u\in\R^d:\Vert u\Vert=1}u^\top\E\Big[\Big( S(\hat{X}^{j, 1})- n^{j,1}\C_\infty(B)\Big)^2\Big]u
    \\
    &\leq\sum_{j=1}^{n_T}\max_{u\in\R^d:\Vert u\Vert=1}u^\top\E\Big[(S(\hat{X}^{j, 1}))^2\Big]u
    \\
    &= \sum_{j=1}^{n_T}\max_{u\in\R^d:\Vert u\Vert=1}\sum_{t_i,t_j\in [2(j-1)m_T,(2j-1)m_T)\cap\T} \E\Big[\vert u^\top X_{t_i}X_{t_i}^\top X_{t_j}X_{t_j}^\top u\1_{B}(X_{t_i})\1_{B}(X_{t_j})\vert\Big]
    \\
    &\leq b^2\sum_{j=1}^{n_T}\max_{u\in\R^d:\Vert u\Vert=1}\sum_{t_i,t_j\in [2(j-1)m_T,(2j-1)m_T)\cap\T} \E\Big[\vert u^\top X_{t_i}\vert\vert X_{t_j}^\top u\vert\1_{B}(X_{t_i})\1_{B}(X_{t_j})\Big]
       \\
    &\leq b^2n_T(n^{j,k})^2\lambda_{\max}(\C_\infty(B))
    \\
    &\leq4 b^2n_T\frac{m_T^2}{\Delta^2_n}\lambda_{\max}(\C_\infty(B)),
\end{align*}
where we used Cauchy--Schwarz and H{\"o}lder's inequality. Analogous arguments give the same bounds for $k=2$. Since the $\hat{X}^{j,k}$ are independent for fixed $k=1,2,$ we can thus apply the matrix Bernstein inequality (Theorem 1.4 in \cite{tropp12}), which gives
\begin{align*}
&\P\Bigg(\Vert\sum_{j=1}^{n_T}\Big(S(\hat{X}^{j, 1})-n^{j,1}\C_\infty(B)\Big)\Vert>\frac{nu}{2}\Bigg) +\P\Bigg(\Vert\sum_{j=1}^{n_T}\Big( S(\hat{X}^{j, 2})- n^{j,2}\C_\infty(B)\Big)\Vert>\frac{nu}{2} \Bigg) 
\\
&\leq \P\Bigg(\lambda_{\max}\Big(\sum_{j=1}^{n_T}\Big(S(\hat{X}^{j, 1})-n^{j,1}\C_\infty(B)\Big)\Big)>\frac{nu}{2}\Bigg)+\P\Bigg(\lambda_{\max}\Big(\sum_{j=1}^{n_T}\Big(n^{j,1}\C_\infty(B)-S(\hat{X}^{j, 1})\Big)\Big)>\frac{nu}{2}\Bigg)
\\
&\quad+\P\Bigg(\lambda_{\max}\Big(\sum_{j=1}^{n_T}\Big(S(\hat{X}^{j, 2})-n^{j,2}\C_\infty(B)\Big)\Big)>\frac{nu}{2}\Bigg)+\P\Bigg(\lambda_{\max}\Big(\sum_{j=1}^{n_T}\Big(n^{j,2}\C_\infty(B)-S(\hat{X}^{j, 2})\Big)\Big)>\frac{nu}{2}\Bigg)
\\
&\leq 4d\exp\Big(-\frac{(nu)^2}{8b^2n_T\frac{m_T^2}{\Delta^2_n}\lambda_{\max}(\C_\infty(B))+\frac{4num_T}{3\Delta_n}b^2}\Big)
\\
&= 4d\exp\Big(-\frac{(nu)^2\Delta_n^2}{8b^2n_Tm_T^2\lambda_{\max}(\C_\infty(B))+\frac{4n\Delta_num_T}{3}b^2}\Big)
\\
&= 4d\exp\Big(-\frac{T^2 u^2}{4b^2m_T(T+\Delta_n)\lambda_{\max}(\C_\infty(B))+\frac{4Tum_T}{3}b^2}\Big)
\\
&\leq 4d\exp\Big(-\frac{3T u^2}{4b^2m_T(6\lambda_{\max}(\C_\infty(B))+u)}\Big),
\end{align*}
where we used that $n_Tm_T=(T+\Delta_n)/2\leq T$ again.
Combining this with \eqref{eq:berndis} implies
\begin{equation}\label{eq:berdiscint}
       \P(\Vert\hat{\C}_n(B)- \C_\infty(B)\Vert>u ) 
  \leq 4d\exp\Big(-\frac{3T u^2}{4b^2m_T(6\lambda_{\max}(\C_\infty(B))+u)}\Big) +2\frac{T}{m_T}\beta_X(m_T).
\end{equation}
Now let $\Delta_n<m_T\leq T/6$ be given and set $\tilde{n}_T\coloneq \lfloor (T+\Delta_n)/(2m_T)\rfloor,\tilde{m}_T\coloneq (T+\Delta_n)/(2\tilde{n}_T).$ Then by definition $\tilde{m}_T\geq m_T,$ and
\begin{align*}
   \tilde{m}_T&= (T+\Delta_n)/(2\tilde{n}_T)
   \\
   &\leq \frac{T+\Delta_n}{(T+\Delta_n)/(2m_T)-1}
   \\
   &=2m_T\frac{T+\Delta_n}{T+\Delta_n-2m_T}
   \\
   &\leq 3m_T,
\end{align*}
where we used $m_T\leq T/6<(T+\Delta_n)/6$. 
Thus \eqref{eq:berdiscint} holds true with $\tilde{m}_T.$
Applying the bounds on $\tilde{m}_T$ in \eqref{eq:berdiscint} together with the mixing property \eqref{eq: mixing} completes the proof.
\end{proof}
For the derivation of our results we require the following event to occur \begin{align*}
 \mathfrak{C}_0\coloneq \{\Vert\hat{\C}_n(B)- \C_\infty(B)\Vert\leq \frac{\lambda_{\min}(\C_\infty(B))}{2}\}. 
\end{align*} Straightforward calculations show that $\lambda_{\min}(\hat{\C}_n(B))\geq \lambda_{\min}(\C_\infty(B))/2$ holds true on $\mathfrak{C}_0,$ which is one of the reasons why we are interested in this event.
The probability of $\mathfrak{C}_0$ is simple to bound by Proposition \ref{prop: REP disc} for large enough values of $T$. The lower bound on $T$ depends on $b$, the radius of the localizing set $B$. As explained in Section \ref{subsec: est}, $b$ should be chosen of the order $\sqrt{d}$.
\begin{corollary}\label{cor:rep disc}
Assume $b\geq 1,$ and let $\ep\in(0,1)$ be given. 
Then for any
\[T\geq T_0(\ep)\coloneq \Bigg(\frac{52b\lambda_{\max}(\C_\infty(B))}{3 \lambda_{\min}(\C_\infty(B))^2}\log\Big(\frac{8d}{\ep}\Big)\lor\frac{12b}{c_{\mix,2}}\log\Big(\frac{144b^2c_{\mix,1}}{c_{\mix,2}\ep}\Big)\Bigg)^2\lor 36\Delta_n^2b^2, \]
the following bound holds true
  \begin{align*}
      \P\Big(\mathfrak{C}_0\Big) \geq 1-\ep.
  \end{align*}
\end{corollary}
\begin{proof}
Since $b\geq 1,T\geq36\Delta_n^2b^2 ,$ we can apply Proposition \ref{prop: REP disc} with $m_T=\sqrt{T}/(6b),$ as soon as $T\geq 1$. Then we obtain
  \begin{align*}
      &\P\Big(\Vert\hat{\C}_n(B)- \C_\infty(B)\Vert>\frac{\lambda_{\min}(\C_\infty(B))}{2} \Big) 
      \\
      &\leq4d\exp\Big(-\frac{T \lambda_{\min}(\C_\infty(B))^2}{8b^2m_T(12\lambda_{\max}(\C_\infty(B))+\lambda_{\min}(\C_\infty(B)))}\Big) +2c_{\mix,1}\frac{T}{m_T}\exp(-c_{\mix,2}m_T)
         \\
      &\leq4d\exp\Big(-\frac{T \lambda_{\min}(\C_\infty(B))^2}{104b^2m_T\lambda_{\max}(\C_\infty(B))}\Big) +2c_{\mix,1}\frac{T}{m_T}\exp(-c_{\mix,2}m_T)
             \\
      &=4d\exp\Big(-\frac{3\sqrt{T} \lambda_{\min}(\C_\infty(B))^2}{52b\lambda_{\max}(\C_\infty(B))}\Big) +12bc_{\mix,1}\sqrt{T}\exp\Big(-c_{\mix,2}\frac{\sqrt{T}}{6b}\Big)
            \\
      &\leq4d\exp\Big(-\frac{3\sqrt{T} \lambda_{\min}(\C_\infty(B))^2}{52b\lambda_{\max}(\C_\infty(B))}\Big) +144b^2\frac{c_{\mix,1}}{c_{\mix,2}}\exp\Big(-c_{\mix,2}\frac{\sqrt{T}}{12b}\Big)
  \end{align*}  
  where we used the elementary inequality $x\leq\exp(x),$ which is valid for all $x\in\R,$ in the last step. Hence, if $T\geq T_0(\ep)\,$ it holds 
  \begin{align*}
    \P(\mathfrak{C}_0)=  \P\Big(\Vert\hat{\C}_n(B)- \C_\infty(B)\Vert\leq \frac{\lambda_{\min}(\C_\infty(B))}{2} \Big) \geq 1-\ep,
  \end{align*}
  which concludes the proof.
\end{proof}
Corollary \ref{cor:rep disc} concerns the empirical covariance matrix $\hat{\C}_n(B)$ without truncation of the increments. However, our results depend on the truncated empirical covariance matrix $\hat{\C}^\eta_n(B).$ The following result uses the matrix Freedman inequality (see \cite{matfreed}) to show that a sufficient concentration of $\hat{C}^\eta_n(B)$ follows under mild assumptions if $\mathfrak{C}_0$ holds true.
\begin{proposition}\label{prop:rep disc trunc}
Assume $b\geq 1$ and let $\ep\in(0,1)$ be given.  Define the events 
  \begin{align*}
       \mathfrak{C}^+_\eta&\coloneq \Big\{\lambda_{\max}(\hat{\C}^\eta_n(B))\leq 2\lambda_{\max}(\C_\infty(B))\P(\Vert \Delta \tilde{Z}_1 \Vert \leq \eta_+))\Big\}\cap \mathfrak{C}_0 ,
       ,
       \\
       \mathfrak{C}^-_\eta&\coloneq \Big\{\lambda_{\min}(\hat{\C}^\eta_n(B))\geq \frac{1}{4}\lambda_{\min}(\C_\infty(B))\P(\Vert \Delta \tilde{Z}_1 \Vert \leq \eta_-)\Big\}\cap \mathfrak{C}_0,
       \\
       \mathfrak{C}&\coloneq \mathfrak{C}^+_\eta\cap \mathfrak{C}^-_\eta,
  \end{align*}
  where $\eta_+$ and $\eta_-$ are defined in Lemma \ref{lemma: ou jumps, levy jumps}.
 Then for any 
 \[T\geq T_1(\ep,\eta)\coloneq T_0(\ep/2)\lor 48\Delta_nb^2 \log(2d/\ep)\frac{\P(\Vert \Delta \tilde{Z}_1 \Vert > \eta_-)\kappa(\C_\infty(B))+\P(\Vert \Delta \tilde{Z}_1 \Vert \leq \eta_-)}{\lambda_{\min}(\C_\infty(B))\P(\Vert \Delta \tilde{Z}_1 \Vert \leq \eta_-))} \]
 it holds
 \[\P(\mathfrak{C})\geq 1-\ep.\]
\end{proposition}
\begin{proof}
In order to ease the presentation, we introduce the following notation for $a\in\R$
  \begin{align*}
  \1^c_{\{\Vert \Delta \tilde{Z}_i \Vert > a\}} &\coloneq \1_{\{\Vert \Delta \tilde{Z}_i \Vert > a\}}-\P(\Vert \Delta \tilde{Z}_i \Vert > a),\quad i\in[n].
  \end{align*}
  Since the $\Delta\tilde{Z}_i$ are i.i.d.\ random variables, the triangle inequality implies together with Lemma \ref{lemma: ou jumps, levy jumps}
  \begin{align*}
   &\lambda_{\max}(\hat{\C}^\eta_n(B))
   \\
   &\leq \lambda_{\max}\Big(\frac{1}{n}
\sum_{i=1}^n
X_{t_{i-1}} X_{t_{i-1}}^\top
\1^c_{\{\Vert \Delta \tilde{Z}_i \Vert \leq \eta_+\}} \1_B \left( X_{t_{i-1}} \right)\Big) +\P(\Vert \Delta \tilde{Z}_1 \Vert \leq  \eta_+))\lambda_{\max}(\hat{\C}_n(B))
\\
   &\leq \lambda_{\max}\Big(\frac{1}{n}
\sum_{i=1}^n
X_{t_{i-1}} X_{t_{i-1}}^\top
\1^c_{\{\Vert \Delta \tilde{Z}_i \Vert \leq \eta_+\}} \1_B \left( X_{t_{i-1}} \right)\Big) \\&\quad+\P(\Vert \Delta \tilde{Z}_1 \Vert \leq  \eta_+))(\Vert \hat{\C}_n(B)-\C_\infty(B)\Vert+\lambda_{\max}(\C_\infty(B)))
  \end{align*}
    and similarly
    \begin{align*}
   &\lambda_{\min}(\hat{\C}^\eta_n(B))
   \\
   &\geq \lambda_{\min}\Big(\frac{1}{n}
\sum_{i=1}^n
X_{t_{i-1}} X_{t_{i-1}}^\top
\1^c_{\{\Vert \Delta \tilde{Z}_i \Vert \leq \eta_-\}}\1_B \left( X_{t_{i-1}} \right)\Big) +\P(\Vert \Delta \tilde{Z}_1 \Vert \leq  \eta_-)\lambda_{\min}(\hat{\C}_n(B))
  \\
   &\geq \lambda_{\min}\Big(\frac{1}{n}
\sum_{i=1}^n
X_{t_{i-1}} X_{t_{i-1}}^\top
\1^c_{\{\Vert \Delta \tilde{Z}_i \Vert \leq \eta_-\}}\1_B \left( X_{t_{i-1}} \right)\Big) \\&\quad+\P(\Vert \Delta \tilde{Z}_1 \Vert \leq  \eta_-)(\lambda_{\min}(\C_\infty(B))-\Vert \hat{\C}_n(B)-\C_\infty(B)\Vert).
  \end{align*}
By the above bounds we get for any $u>0$
  \begin{align}
  \begin{split}\label{eq: tr err 1}
     &\P\Big(\lambda_{\max}(\hat{\C}^\eta_n(B))>u+\frac{3}{2}\lambda_{\max}(\C_\infty(B))\P(\Vert \Delta \tilde{Z}_1 \Vert \leq \eta_+)),\mathfrak{C}_0 \Big)
  \\
  &\leq \P\Bigg(\lambda_{\max}\Big(
\sum_{i=1}^n
X_{t_{i-1}} X_{t_{i-1}}^\top
\1^c_{\{\Vert \Delta \tilde{Z}_i \Vert \leq \eta_+\}}\1_B \left( X_{t_{i-1}} \right)\Big)>nu,\mathfrak{C}_0\Bigg), 
  \end{split}
  \end{align}
  and
   \begin{align}
  \begin{split}\label{eq: tr err 2}
     &\P\Big(\lambda_{\min}(\hat{\C}^\eta_n(B))<-u+\frac{1}{2}\lambda_{\min}(\C_\infty(B))\P(\Vert \Delta \tilde{Z}_1 \Vert \leq \eta_-)),\mathfrak{C}_0 \Big)
  \\
  &\leq \P\Bigg(\lambda_{\min}\Big(
\sum_{i=1}^n
X_{t_{i-1}} X_{t_{i-1}}^\top
\1^c_{\{\Vert \Delta \tilde{Z}_i \Vert \leq \eta_-\}}\1_B \left( X_{t_{i-1}} \right)\Big)<-nu,\mathfrak{C}_0 \Bigg)
  \\
  &= \P\Bigg(\lambda_{\max}\Big(-
\sum_{i=1}^n
X_{t_{i-1}} X_{t_{i-1}}^\top
\1^c_{\{\Vert \Delta \tilde{Z}_i \Vert \leq \eta_-\}}\1_B \left( X_{t_{i-1}} \right)\Big)>nu,\mathfrak{C}_0\Bigg).
  \end{split}
  \end{align} 
  Since  $\Delta\tilde{Z}_i$ is independent of $\mathcal{F}_{t_{i-1}}$ for any $i\in[n],$ and because the $\Delta\tilde{Z}_i$ are i.i.d.\ it follows that $\sum_{i=1}^n
X_{t_{i-1}} X_{t_{i-1}}^\top
\1^c_{\{\Vert \Delta \tilde{Z}_i \Vert > a\}}\1_B \left( X_{t_{i-1}} \right)$ is a matrix martingale for any choice of $a\in\R$. Furthermore, for any $i\in[n]$ it holds almost surely that \[\lambda_{\max}(X_{t_{i-1}} X_{t_{i-1}}^\top
\1^c_{\{\Vert \Delta \tilde{Z}_i \Vert \leq \eta_+\}}\1_B \left( X_{t_{i-1}} \right))\leq b^2(1-\P(\Vert \Delta \tilde{Z}_1 \Vert \leq \eta_+))= b^2\P(\Vert \Delta \tilde{Z}_1 \Vert > \eta_+),\] 
and similarly
\[\lambda_{\max}(-X_{t_{i-1}} X_{t_{i-1}}^\top
\1^c_{\{\Vert \Delta \tilde{Z}_i \Vert \leq \eta_-\}}\1_B \left( X_{t_{i-1}} \right))\leq b^2\P(\Vert \Delta \tilde{Z}_1 \Vert \leq \eta_-).\] 
Additionally we obtain for any $a\in\R$
\begin{align*}
 &   \Vert\sum_{i=1}^n
\E[(X_{t_{i-1}} X_{t_{i-1}}^\top
\1^c_{\{\Vert \Delta \tilde{Z}_i \Vert \leq  a\}}\1_B \left( X_{t_{i-1}} \right))^2\vert\mathcal{F}_{t_{i-1}}]\Vert
\\
&\leq  b^2\E[
(\1^c_{\{\Vert \Delta \tilde{Z}_1 \Vert \leq a \}})^2]\Vert\sum_{i=1}^nX_{t_{i-1}} X_{t_{i-1}}^\top\1_B \left( X_{t_{i-1}} \right)\Vert
\\
&= nb^2\P(\Vert \Delta \tilde{Z}_1 \Vert > a)\P(\Vert \Delta \tilde{Z}_1 \Vert \leq a)\Vert\hat{\C}_n(B)\Vert,
\end{align*}
and thus the set inclusion
\begin{align*}
   \mathfrak{C}_0&\subseteq \Big\{\Vert\sum_{i=1}^n
\E[(X_{t_{i-1}} X_{t_{i-1}}^\top
\1^c_{\{\Vert \Delta \tilde{Z}_i \Vert \leq a \}}\1_B \left( X_{t_{i-1}} \right))^2\vert\mathcal{F}_{t_{i-1}}]\Vert
\\&\quad\leq \frac{3}{2}nb^2\P(\Vert \Delta \tilde{Z}_1 \Vert > a)\P(\Vert \Delta \tilde{Z}_1 \Vert \leq a)\lambda_{\max}(\C_\infty(B))\Big\},
\end{align*}
holds true for any $a\in\R$. The above allows us to apply Freedman's inequality for matrix martingales (see \cite{matfreed}), which gives for any $u>0$
\begin{align*}
   &\P\Bigg(\lambda_{\max}\Big(
\sum_{i=1}^n
X_{t_{i-1}} X_{t_{i-1}}^\top
\1^c_{\{\Vert \Delta \tilde{Z}_i \Vert \leq  \eta_+\}}\1_B \left( X_{t_{i-1}} \right)\Big)>nu,\mathfrak{C}_0\Bigg)
\\
&\leq d\exp\Bigg(-\frac{nu^2}{b^2\P(\Vert \Delta \tilde{Z}_1 \Vert > \eta_+)(3\P(\Vert \Delta \tilde{Z}_1 \Vert \leq \eta_+)\lambda_{\max}(\C_\infty(B))+\frac{2}{3}u)}\Bigg),
\end{align*}
and
\begin{align*}
   & \P\Bigg(\lambda_{\max}\Big(-
\sum_{i=1}^n
X_{t_{i-1}} X_{t_{i-1}}^\top
\1^c_{\{\Vert \Delta \tilde{Z}_i \Vert \leq \eta_-\}}\1_B \left( X_{t_{i-1}} \right)\Big)>nu,\mathfrak{C}_0\Bigg)
\\
&\leq d\exp\Bigg(-\frac{nu^2}{b^2\P(\Vert \Delta \tilde{Z}_1 \Vert \leq \eta_-)(3\P(\Vert \Delta \tilde{Z}_1 \Vert > \eta_-)\lambda_{\max}(\C_\infty(B))+\frac{2}{3}u)}\Bigg).
\end{align*}
Thus combining \eqref{eq: tr err 1} and \eqref{eq: tr err 2} with the above inequalities, we get
\begin{align*}
    &\P\Big(\lambda_{\max}(\hat{\C}^\eta_n(B))>2\lambda_{\max}(\C_\infty(B))\P(\Vert \Delta \tilde{Z}_1 \Vert \leq \eta_+)),\mathfrak{C}_0 \Big)
    \\
    &\leq d\exp\Bigg(-\frac{3n\lambda_{\max}(\C_\infty(B))\P(\Vert \Delta \tilde{Z}_1 \Vert \leq \eta_+))}{40b^2\P(\Vert \Delta \tilde{Z}_1 \Vert > \eta_+)}\Bigg)
    \\
    &= d\exp\Bigg(-\frac{3T\lambda_{\max}(\C_\infty(B))\P(\Vert \Delta \tilde{Z}_1 \Vert \leq \eta_+))}{40\Delta_nb^2\P(\Vert \Delta \tilde{Z}_1 \Vert > \eta_+)}\Bigg)
\end{align*}
and
\begin{align*}
    &\P\Big(\lambda_{\min}(\hat{\C}^\eta_n(B))<\frac{1}{4}\lambda_{\min}(\C_\infty(B))\P(\Vert \Delta \tilde{Z}_1 \Vert \leq \eta_-)),\mathfrak{C}_0\Big)
    \\
    &\leq d\exp\Bigg(-\frac{n\lambda_{\min}(\C_\infty(B))\P(\Vert \Delta \tilde{Z}_1 \Vert \leq \eta_-))}{16b^2(3\P(\Vert \Delta \tilde{Z}_1 \Vert > \eta_-)\kappa(\C_\infty(B))+\frac{1}{6}\P(\Vert \Delta \tilde{Z}_1 \Vert \leq \eta_-))}\Bigg)
    \\
    &= d\exp\Bigg(-\frac{T\lambda_{\min}(\C_\infty(B))\P(\Vert \Delta \tilde{Z}_1 \Vert \leq \eta_-))}{16\Delta_nb^2(3\P(\Vert \Delta \tilde{Z}_1 \Vert > \eta_-)\kappa(\C_\infty(B))+\frac{1}{6}\P(\Vert \Delta \tilde{Z}_1 \Vert \leq \eta_-))}\Bigg).
\end{align*}
The assertion now follows by observing that $T\geq T_1(\ep)$
implies by Corollary \ref{cor:rep disc}
\begin{align*}
   & \P(\mathfrak{C})
   \\
   &\geq \P(\mathfrak{C}^+_\eta)+ \P(\mathfrak{C}^-_\eta)- \P(\mathfrak{C}_0)
   \\
   &\geq \P(\mathfrak{C}_0)-d\exp\Bigg(-\frac{3T\lambda_{\max}(\C_\infty(B))\P(\Vert \Delta \tilde{Z}_1 \Vert \leq \eta_+))}{40\Delta_nb^2\P(\Vert \Delta \tilde{Z}_1 \Vert > \eta_+)}\Bigg)
   \\
   &\quad-d\exp\Bigg(-\frac{T\lambda_{\min}(\C_\infty(B))\P(\Vert \Delta \tilde{Z}_1 \Vert \leq \eta_-))}{16\Delta_nb^2(3\P(\Vert \Delta \tilde{Z}_1 \Vert > \eta_-)\kappa(\C_\infty(B))+\frac{1}{6}\P(\Vert \Delta \tilde{Z}_1 \Vert \leq \eta_-))}\Bigg)
   \\
   &\geq 1-\ep.
\end{align*}
\end{proof}
\subsection{Stochastic Error}
In this section we derive a uniform high probability estimate for the stochastic error
\[\mathcal{S}(\M) \coloneq \sum_{i=1}^n \left(\M X_{t_{i-1}}\right)^\top \Delta \tilde{Z}_i \1_{\{\Vert \Delta X_i \Vert< \eta\}} \1_B \left( X_{t_{i-1}} \right),\quad \M\in\R^{d\times d}.\]
The proof of the estimate relies on concentration inequalities for martingales combined with Talagrand's generic chaining device. However, the stochastic error $\mathcal{S}(\cdot)$ is not a martingale due to the truncation of the increments. We therefore firstly derive the following decomposition, which bounds the stochastic error by the sum of a martingale and another term, to which we refer as truncation bias.
\begin{lemma}\label{lemma: st decomp}
Assume that $\eta_->0$. Then the following bound holds true for any $c>0,$ 
\begin{align*}
&\frac{2}{T} \vert\mathcal{S}(\M) \vert
\\
&\leq \frac{2}{T}\vert\sum_{i=1}^n \left(\M X_{t_{i-1}}\right)^\top \Big(\Delta \tilde{Z}_i\1_{\{\Vert \Delta X_i \Vert< \eta\}} -\E[\Delta \tilde{Z}_i\1_{\{\Vert \Delta X_i \Vert< \eta\}}   \vert\mathcal{F}_{t_{i-1}}]\Big)\1_B \left( X_{t_{i-1}} \right)\vert+\frac{1}{c}\Vert \M X\Vert_{l^2_n(B,X)}^2
\\
&\quad +\frac{c}{\Delta_n} \lambda_{\max}(\C+\boldsymbol{\nu}_2)\exp(\Delta_n\Vert\A_0\Vert)\P(\Vert \Delta  \tilde{Z}_1\Vert\geq \eta_- ),
\end{align*}
where $\eta_-$ is defined in Lemma \ref{lemma: ou jumps, levy jumps}.
\end{lemma}
\begin{proof}
As $\Delta\tilde{Z}_i$ is centered and independent of $\mathcal{F}_{t_{i-1}}$ for any $i\in[n]$, it holds for any $i\in[n]$
\begin{align*}
    &\Delta\tilde{Z}_i \1_{\{\Vert \Delta X_i \Vert< \eta\}} \1_B \left( X_{t_{i-1}} \right)
    \\
    &=\Delta\tilde{Z}_i\1_{\{\Vert \Delta X_i \Vert< \eta\}} -\E[\Delta \tilde{Z}_i\vert \mathcal{F}_{t_{i-1}}]
     \\
    &= \Big(\Delta \tilde{Z}_i\1_{\{\Vert \Delta X_i \Vert< \eta\}} -\E[\Delta \tilde{Z}_i\1_{\{\Vert \Delta X_i \Vert< \eta\}}   \vert\mathcal{F}_{t_{i-1}}]\Big)
 - \E[\Delta \tilde{Z}_i\1_{\{\Vert \Delta X_i \Vert\geq \eta\}} \vert \mathcal{F}_{t_{i-1}}],
\end{align*}
such that we can bound the stochastic error as follows
\begin{align}
\begin{split}\label{eq:st decomp 1}
     \frac{2}{T} \vert\mathcal{S}(\M) \vert
    &=\vert\sum_{i=1}^n \left(\M X_{t_{i-1}}\right)^\top \Delta \tilde{Z}_i \1_{\{\Vert \Delta X_i \Vert< \eta\}} \1_B \left( X_{t_{i-1}} \right)\vert
    \\
    &\leq \frac{2}{T}\vert\sum_{i=1}^n \left(\M X_{t_{i-1}}\right)^\top \Big(\Delta \tilde{Z}_i\1_{\{\Vert \Delta X_i \Vert< \eta\}} -\E[\Delta \tilde{Z}_i\1_{\{\Vert \Delta X_i \Vert< \eta\}}   \vert\mathcal{F}_{t_{i-1}}]\Big)\1_B \left( X_{t_{i-1}} \right)\vert
    \\&\quad + \frac{2}{T}\vert\sum_{i=1}^n \left(\M X_{t_{i-1}}\right)^\top \E[\Delta \tilde{Z}_i\1_{\{\Vert \Delta X_i \Vert\geq \eta\}} \vert \mathcal{F}_{t_{i-1}}]\1_B \left( X_{t_{i-1}} \right)\vert. 
\end{split}
\end{align}
For proving the assertion it now suffices to bound the second summand. By combining Young's and H{\"o}lder's inequalities with the $\mathcal{F}_{t_i}$-measurability of $X_{t_i}$ and the Courant--Fischer theorem, we obtain for any $c>0$
\begin{align*}
    & \frac{2}{T}\vert\sum_{i=1}^n \left(\M X_{t_{i-1}}\right)^\top \E[\Delta \tilde{Z}_i\1_{\{\Vert \Delta X_i \Vert\geq \eta\}} \vert \mathcal{F}_{t_{i-1}}]\1_B \left( X_{t_{i-1}} \right)\vert
    \\
    &\leq \frac{\Delta_n}{Tc}\sum_{i=1}^n \Vert\M X_{t_{i-1}}\Vert^2+\frac{c}{T\Delta_n}\sum_{i=1}^n \E[(\M X_{t_{i-1}}/\Vert\M X_{t_{i-1}}\Vert)^\top\Delta \tilde{Z}_i\1_{\{\Vert \Delta X_i \Vert\geq \eta\}}\1_B \left( X_{t_{i-1}} \right) \vert \mathcal{F}_{t_{i-1}}]^2
    \\
    &\leq  \frac{1}{c}\Vert \M X\Vert_{l^2_n(B,X)}^2+\frac{c}{T\Delta_n}\sum_{i=1}^n \E[((\M X_{t_{i-1}}/\Vert\M X_{t_{i-1}}\Vert)^\top\Delta \tilde{Z}_i)^2\vert \mathcal{F}_{t_{i-1}}]\Vert\P(\Vert \Delta X_i \Vert\geq \eta, X_{t_{i-1}}\in B\vert \mathcal{F}_{t_{i-1}} )  
     \\
    &\leq  \frac{1}{c}\Vert \M X\Vert_{l^2_n(B,X)}^2+\frac{nc}{T\Delta_n} \lambda_{\max}\Big(\E[\Delta \tilde{Z}_0(\Delta \tilde{Z}_0)^\top]\Big)\P(\Vert \Delta  \tilde{Z}_1\Vert\geq \eta_- )  
         \\
    &\leq  \frac{1}{c}\Vert \M X\Vert_{l^2_n(B,X)}^2+\frac{c}{\Delta_n} \lambda_{\max}(\C+\boldsymbol{\nu}_2)\exp(\Delta_n\Vert\A_0\Vert)\P(\Vert \Delta  \tilde{Z}_1\Vert\geq \eta_- ),
\end{align*}
where we also used the stationarity of $\X$ and the increments of $\Z$ together with $n/(T\Delta_n)=\Delta_n^{-2},$ Markov's inequality and Lemmas \ref{lemma: ou jumps, levy jumps} and \ref{lemma: ew z}. This concludes the proof.
\end{proof}
In order to simplify the notation we introduce for $\M\in\R^{d\times d},i\in[n],$
\begin{align*}
  \Delta \Bar{Z}_i&\coloneq \Big(\Delta \tilde{Z}_i\1_{\{\Vert \Delta X_i \Vert< \eta\}} -\E[\Delta \tilde{Z}_i\1_{\{\Vert \Delta X_i \Vert< \eta\}}   \vert\mathcal{F}_{t_{i-1}}]\Big)
  \\
  \xi_i(\M)&\coloneq \left(\M X_{t_{i-1}}\right)^\top\Bar{Z}_i \1_B \left( X_{t_{i-1}} \right)
  \\
  \tilde{\mathcal{S}}(\M) &\coloneq \sum_{i=1}^n \xi_i(\M),
\end{align*}
and additionally we introduce $\mathcal{S}^{(d\times d)-1}\coloneq \{\M\in\R^{d\times d}\colon \Vert\M\Vert_2=1\}$. $\tilde{\mathcal{S}}(\M)$ denotes the martingale part of the upper bound obtained in Lemma \ref{lemma: st decomp}. To obtain a uniform bound for $\tilde{\mathcal{S}}(\M),$ we make use of the results of \cite{vanzanten01}, which show that martingales admit a subgaussian concentration behavior if their (predictable) quadratic variation process is well-behaved. We therefore introduce the event $Q_T$ in \eqref{eq: def q} on which said quadratic variation process is uniformly bounded by $T,$ and show that it occurs with high probability for sufficiently large $T$. The key to the proof is a combination of Freedman's inequality, which shows that the quadratic variation process admits a mixed tail in the sense of \cite{dirk13}, and subsequently applying the concentration inequalities for suprema of stochastic processes admitting a mixed tail in \cite{dirk13}.
\begin{proposition}\label{prop: qt}
Let $\ep\in(0,1)$ be given, and define the event
\begin{align}\label{eq: def q}
   &Q_T
\coloneq\Big\{\sup_{\M\in\mathcal{S}^{(d\times d)-1}}\Big(\sum_{i=1}^n\xi^2_i(\M)+\sum_{i=1}^n\E[\xi^2_i(\M)\vert\mathcal{F}_{t_{i-1}}]\Big)\leq  4\gamma(\Delta_n)T\Big\}\cap \mathfrak{C}, 
\end{align}
where
\[\gamma(\Delta_n)\coloneq\lambda_{\max}(\C+\boldsymbol{\nu}_2)(\lambda_{\max}(\C_\infty(B))\lor 1)\exp(\Delta_n\Vert\A_0\Vert).\] Then there exists a constant $c_{1,\star}>0,$ independent of $d,$ such that for all 
\[T\geq T_\star(\ep,\eta)\coloneq T_1(\ep/2,\eta)\lor c_{1,\star}\gamma^{-1}(\Delta_n) b^2\eta_+^2\Big(d+\sqrt{\log(2/\ep)}\Big)^2, \]
where $\eta_+$ is defined in Lemma \ref{lemma: ou jumps, levy jumps}, it holds
\[\P(Q_T)\geq 1-\ep.\]
\end{proposition}
\begin{proof}
First, we note that Lemma \ref{lemma: ew z} implies for any $\M\in\R^{d\times d},$
\begin{align*}
&\sum_{i=1}^n\E[\xi^2_i(\M)\vert\mathcal{F}_{t_{i-1}}]   
\\
&=\sum_{i=1}^n  \1_B \left( X_{t_{i-1}} \right)(\M X_{t_{i-1}})^\top\E[\Delta\bar{Z}_i(\Delta\bar{Z}_i)^\top\vert\mathcal{F}_{t_{i-1}}] \M X_{t_{i-1}}
\\
&\leq \sum_{i=1}^n  \1_B \left( X_{t_{i-1}} \right)(\M X_{t_{i-1}})^\top\E[\Delta\tilde{Z}_i(\Delta\tilde{Z}_i)^\top\vert\mathcal{F}_{t_{i-1}}] \M X_{t_{i-1}}
\\
&\leq \Delta_n \lambda_{\max}(\C+\boldsymbol{\nu}_2)\exp(\Delta_n\Vert\A_0\Vert)\sum_{i=1}^n  \1_B \left( X_{t_{i-1}} \right)(\M X_{t_{i-1}})^\top\M X_{t_{i-1}}
\\
&=T \lambda_{\max}(\C+\boldsymbol{\nu}_2)\exp(\Delta_n\Vert\A_0\Vert)\tr(\M\hat{\C}_n(B)\M)
\\
&\leq T \lambda_{\max}(\C+\boldsymbol{\nu}_2)\exp(\Delta_n\Vert\A_0\Vert)\Vert \M\Vert_2^2\lambda_{\max}(\hat{\C}_n(B)),
\end{align*}
which implies the following set inclusion
\begin{align*}
   &\Big\{\sup_{\M\in\mathcal{S}^{(d\times d)-1}}\Big(\sum_{i=1}^n\xi^2_i(\M)+\sum_{i=1}^n\E[\xi^2_i(\M)\vert\mathcal{F}_{t_{i-1}}]\Big)> 4T\gamma(\Delta_n)\Big\}\cap \mathfrak{C} 
   \\
   &\subseteq\Big\{\sup_{\M\in\mathcal{S}^{(d\times d)-1}}\sum_{i=1}^n(\xi^2_i(\M)-\E[\xi^2_i(\M)\vert\mathcal{F}_{t_{i-1}}])+2\sup_{\M\in\mathcal{S}^{(d\times d)-1}}\sum_{i=1}^n\E[\xi^2_i(\M)\vert\mathcal{F}_{t_{i-1}}]
> 4T\gamma(\Delta_n)\Big\}\cap \mathfrak{C}
   \\
   &\subseteq\Big\{\sup_{\M\in\mathcal{S}^{(d\times d)-1}}\sum_{i=1}^n(\xi^2_i(\M)-\E[\xi^2_i(\M)\vert\mathcal{F}_{t_{i-1}}])> T\lambda_{\max}(\C+\boldsymbol{\nu}_2)\lambda_{\max}(\C_\infty(B))\exp(\Delta_n\Vert\A_0\Vert)\Big\}\cap \mathfrak{C}.
\end{align*}
Thus, in order to prove the first assertion, it is enough to verify that for any $T\geq T_1(\ep)$
\[\P\Big(\sup_{\M\in\mathcal{S}^{(d\times d)-1}}\sum_{i=1}^n(\xi^2_i(\M)-\E[\xi^2_i(\M)\vert\mathcal{F}_{t_{i-1}}])> \gamma(\Delta_n)T,\mathfrak{C}\Big)\leq \ep. \]
For this we let arbitrary $\M_1,\M_2\in\mathcal{S}^{(d\times d)-1}$ be given. Then the following identity 
\begin{align*}
    \xi^2_i(\M_1)-\xi^2_i(\M_2)
    &=((\M_1-\M_2) X_{t_{i-1}})^\top\Delta\bar{Z}_i((\M_1+\M_2) X_{t_{i-1}})^\top\Delta\bar{Z}_i\1_B(X_{t_{i-1}}),
\end{align*}
implies by Lemma \ref{lemma: ou jumps, levy jumps}
\begin{align*}
    \vert \xi^2_i(\M_1)-\E[\xi^2_i(\M_1)\vert\mathcal{F}_{t_{i-1}}]-\xi^2_i(\M_2)+\E[\xi^2_i(\M_2)\vert\mathcal{F}_{t_{i-1}}]\vert
\leq 4b^2\eta_+^2\Vert\M_1-\M_2\Vert_2.
\end{align*}
Additionally, we obtain
\begin{align*}
    &\sum_{i=1}^n\E[( \xi^2_i(\M_1)-\E[\xi^2_i(\M_1)\vert\mathcal{F}_{t_{i-1}}]-\xi^2_i(\M_2)+\E[\xi^2_i(\M_2)\vert\mathcal{F}_{t_{i-1}}])^2\vert \mathcal{F}_{t_{i-1}}]
    \\
    &\leq\sum_{i=1}^n\E[(((\M_1-\M_2) X_{t_{i-1}})^\top\Delta\bar{Z}_i((\M_1+\M_2) X_{t_{i-1}})^\top\Delta\bar{Z}_i)^2\1_B(X_{t_{i-1}})\vert \mathcal{F}_{t_{i-1}}]
    \\
    &\leq 4b^2\eta_+^2\sum_{i=1}^n\E[(((\M_1-\M_2) X_{t_{i-1}})^\top\Delta\bar{Z}_i)^2\1_B(X_{t_{i-1}})\vert \mathcal{F}_{t_{i-1}}]
    \\
    &= 4b^2\eta_+^2\sum_{i=1}^n((\M_1-\M_2) X_{t_{i-1}})^\top\E[\Delta\bar{Z}_i(\Delta\bar{Z}_i)^\top\vert \mathcal{F}_{t_{i-1}}](\M_1-\M_2) X_{t_{i-1}}\1_B(X_{t_{i-1}})
    \\
    &\leq 4b^2\eta_+^2\sum_{i=1}^n((\M_1-\M_2) X_{t_{i-1}})^\top\E[\Delta\tilde{Z}_i(\Delta\tilde{Z}_i)^\top\vert \mathcal{F}_{t_{i-1}}](\M_1-\M_2) X_{t_{i-1}}\1_B(X_{t_{i-1}})
    \\
    &\leq 4b^2\eta_+^2\Delta_n \lambda_{\max}(\C+\boldsymbol{\nu}_2)\exp(\Delta_n\Vert\A_0\Vert)\sum_{i=1}^n((\M_1-\M_2) X_{t_{i-1}})^\top(\M_1-\M_2) X_{t_{i-1}}\1_B(X_{t_{i-1}})
       \\
    &= 4b^2\eta_+^2T \lambda_{\max}(\C+\boldsymbol{\nu}_2)\exp(\Delta_n\Vert\A_0\Vert)\tr((\M_1-\M_2)\hat{\C}_n(B)(\M_1-\M_2)^\top )
         \\
    &\leq  4b^2\eta_+^2T \lambda_{\max}(\C+\boldsymbol{\nu}_2)\exp(\Delta_n\Vert\A_0\Vert)\Vert \M_1-\M_2\Vert_2^2\lambda_{\max}(\hat{\C}_n(B) ).
\end{align*}
This yields the following set inclusion
\begin{align*}
    \mathfrak{C}&=\Big\{\Vert\hat{\C}_n(B)- \C_\infty(B)\Vert\leq \frac{\lambda_{\min}(\C_\infty(B))}{2} \Big\}
    \\
    &\subseteq \Big\{\forall \M_1,\M_2\in\mathcal{S}^{(d\times d)-1}: \sum_{i=1}^n\E[( \xi^2_i(\M_1)-\E[\xi^2_i(\M_1)\vert\mathcal{F}_{t_{i-1}}]-\xi^2_i(\M_2)+\E[\xi^2_i(\M_2)\vert\mathcal{F}_{t_{i-1}}])^2\vert \mathcal{F}_{t_{i-1}}]
    \\
    &\quad \leq 6b^2\eta_+^2T \gamma(\Delta_n)\Vert \M_1-\M_2\Vert_2^2\Big\}
\end{align*}
Hence, by Freedman's inequality
\begin{align*}
&\P\Big(\vert\sum_{i=1}^n \xi^2_i(\M_1)-\E[\xi^2_i(\M_1)\vert\mathcal{F}_{t_{i-1}}]-\xi^2_i(\M_2)+\E[\xi^2_i(\M_2)\vert\mathcal{F}_{t_{i-1}}]\vert\1_{\mathfrak{C}}>u\Big)
\\
&=\P\Big(\vert\sum_{i=1}^n \xi^2_i(\M_1)-\E[\xi^2_i(\M_1)\vert\mathcal{F}_{t_{i-1}}]-\xi^2_i(\M_2)+\E[\xi^2_i(\M_2)\vert\mathcal{F}_{t_{i-1}}]\vert>u,\mathfrak{C}\Big)
\\
&\leq 2\exp\Bigg(-\frac{u^2}{b^2\eta_+^2\Vert\M_1-\M_2\Vert_2(12\gamma(\Delta_n)T \Vert \M_1-\M_2\Vert_2+\frac{8}{3}u)}\Bigg)
\\
&\leq 2\exp\Bigg(-\frac{u^2}{2b^2\eta_+^2\Vert\M_1-\M_2\Vert_2(12\gamma(\Delta_n)T \Vert \M_1-\M_2\Vert_2\lor\frac{8}{3}u)}\Bigg)
\\
&= 2\exp\Bigg(-\frac{1}{b^2\eta_+^2\Vert\M_1-\M_2\Vert_2}\Big(\frac{u^2}{24\gamma(\Delta_n)T\Vert \M_1-\M_2\Vert_2}\land\frac{3u}{16}\Big)\Bigg).
\end{align*} 
For any $u>0,$ it thus holds
\begin{align*}
&\P\Big(\vert\sum_{i=1}^n \xi^2_i(\M_1)-\E[\xi^2_i(\M_1)\vert\mathcal{F}_{t_{i-1}}]-\xi^2_i(\M_2)+\E[\xi^2_i(\M_2)\vert\mathcal{F}_{t_{i-1}}]\vert\1_{\mathfrak{C}}>
\\&\hspace{10em}\sqrt{24\gamma(\Delta_n)}b\eta_+\Vert \M_1-\M_2\Vert_2\sqrt{Tu}+\frac{16}{3}b^2\eta_+^2\Vert\M_1-\M_2\Vert_2u\Big)
\leq 2\exp(-u),
\end{align*} 
which implies that $(\sum_{i=1}^n(\xi^2_i(\M)-\E[\xi^2_i(\M)\vert\mathcal{F}_{t_{i-1}}])\1_{\mathfrak{C}})_{\M\in\mathcal{S}^{(d\times d)-1}} $ has a mixed tail in the sense of definition (12) of \cite{dirk13}.
Thus, Theorem 3.5 of \cite{dirk13} implies that there exists some universal constant $c_3>0,$ such that
\begin{align}
\begin{split}\label{eq: cv chain 1}
     &\P\Bigg(\sup_{\M\in\mathcal{S}^{(d\times d)-1}}\sum_{i=1}^n(\xi^2_i(\M)-\E[\xi^2_i(\M)\vert\mathcal{F}_{t_{i-1}}])\1_{\mathfrak{C}}>c_3\Big(\gamma_2(\mathcal{S}^{(d\times d)-1},\sqrt{24\gamma(\Delta_n)T}b\eta_+\Vert\cdot\Vert_2)
    \\&\hspace{13em}+\gamma_1(\mathcal{S}^{(d\times d)-1},\frac{8}{3}b^2\eta_+^2\Vert\cdot\Vert_2)+2\sqrt{24u\gamma(\Delta_n)T}b\eta_++\frac{16}{3}ub^2\eta_+^2\Big)\Bigg)
    \\
    &\leq \e^{-u},
\end{split}
\end{align}
where $\gamma_1$ and $\gamma_2$ denote Talagrand's $\gamma_\alpha$ functionals (see Section 2 in \cite{dirk13}). To bound the $\gamma_2$ functional we employ Talagrand's majorizing measure Theorem (see Theorem 8.6.1 in \cite{vers18}). For this let $\mathsf{N}$ be a $d^2$-dimensional standard normal random vector.  Then it holds for some absolute constant $c_4>0,$
\begin{align}
\begin{split}\label{eq: cv chain 2}
     &\gamma_2(\mathcal{S}^{(d\times d)-1},\sqrt{24\gamma(\Delta_n)T}b\eta_+\Vert\cdot\Vert_2)  
  \\
  &\leq c_4b\eta_+\sqrt{\gamma(\Delta_n)T}\E\Big[\sup_{u\in\mathcal{S}^{d^2-1}}u^\top\mathsf{N}\Big]
  \\
  &\leq c_4b\eta_+\sqrt{\gamma(\Delta_n)T}d, 
\end{split}
\end{align}
where we used Example 7.5.7 in \cite{vers18} in the last step. In order to bound the $\gamma_1$ functional we make use of the fact (see e.g. equation (4) in \cite{dirk13}) that for some universal constant $c_5>0$
\[\gamma_1(\mathcal{S}^{(d\times d)-1},\frac{16}{3}b^2\eta_+^2\Vert\cdot\Vert_2)\leq c_5\int_0^\infty\log(\mathcal{N}(\mathcal{S}^{(d\times d)-1},\frac{16}{3}b^2\eta_+^2\Vert\cdot\Vert_2,u)\d u,\]
where for $\ep>0$ and $(T,d)$ being a semi-metric space $\mathcal{N}(T,d,\ep)$ denotes the covering number of $T$, i.e. the smallest number of balls of $d$-radius $\ep$ required to cover $T$. Hence, we obtain
\begin{align}
\begin{split}\label{eq: cv chain 3}
      &\gamma_1(\mathcal{S}^{(d\times d)-1},\frac{16}{3}b^2\eta_+^2\Vert\cdot\Vert_2)
  \\
  &\leq c_5\int_0^\infty\log(\mathcal{N}(\mathcal{S}^{(d\times d)-1},\frac{16}{3}b^2\eta_+^2\Vert\cdot\Vert_2,u)\d u  
  \\
  &=c_5\int_0^\infty\log(\mathcal{N}(\mathcal{S}^{(d\times d)-1},\Vert\cdot\Vert_2,\frac{3}{16b^2\eta_+^2}u)\d u
  \\
  &=c_5\frac{16}{3}b^2\eta_+^2\int_0^\infty\log(\mathcal{N}(\mathcal{S}^{(d\times d)-1},\Vert\cdot\Vert_2,u)\d u
   \\
  &\leq c_5\frac{16}{3}b^2\eta_+^2d^2\int_0^2\log\Big(1+\frac{2}{u}\Big)\d u
  \\
  &=c_5\frac{64}{3}\log(2)b^2\eta_+^2d^2,
\end{split}
\end{align}
where we used Corollary 4.2.13 in \cite{vers18} in the second to last step, together with the fact that $\Vert \M_1-\M_2\Vert_2\leq 2$ for any $\M_1,\M_2\in \mathcal{S}^{(d\times d)-1}$. Thus, combining \eqref{eq: cv chain 1}, \eqref{eq: cv chain 2} and \eqref{eq: cv chain 3} gives that for some universal constant $c_6>0$ it holds for any $\ep\in(0,1)$
\begin{align*}
    &\P\Bigg(\sup_{\M\in\mathcal{S}^{(d\times d)-1}}\sum_{i=1}^n(\xi^2_i(\M)-\E[\xi^2_i(\M)\vert\mathcal{F}_{t_{i-1}}])>
\\
&\hspace{4em}c_6\Big(b\eta_+\sqrt{\gamma(\Delta_n)T}\Big(d+\sqrt{\log(\ep^{-1})}\Big)+b^2\eta_+^2\Big(d^2+\log(\ep^{-1})\Big)\Big),\mathfrak{C}\Bigg)
\\
    &=\P\Bigg(\sup_{\M\in\mathcal{S}^{(d\times d)-1}}\sum_{i=1}^n(\xi^2_i(\M)-\E[\xi^2_i(\M)\vert\mathcal{F}_{t_{i-1}}])\1_{\mathfrak{C}}>
    \\
&\hspace{4em}c_6\Big(b\eta_+\sqrt{\gamma(\Delta_n)T}\Big(d+\sqrt{\log(\ep^{-1})}\Big)+b^2\eta_+^2\Big(d^2+\log(\ep^{-1})\Big)\Big)\Bigg)
    \\
&\leq \ep.
\end{align*}
Noting that $T\geq c_{1,\star}\gamma^{-1}(\Delta_n) b^2\eta_+^2(d+\sqrt{\log(\ep^{-1})})^2,$ where $c_{1,\star}=(2c_6\lor 4c_6^2),$ implies
\begin{align*}
&c_6\Big(b\eta_+\sqrt{\gamma(\Delta_n)T}\Big(d+\sqrt{\log(\ep^{-1})}\Big)+b^2\eta_+^2\Big(d^2+\log(\ep^{-1})\Big)\Big)
\\
&=\gamma(\Delta_n)Tc_6\Big(b\eta_+(\gamma(\Delta_n)T)^{-1/2}\Big(d+\sqrt{\log(\ep^{-1})}\Big)+(\gamma(\Delta_n)T)^{-1}b^2\eta_+^2\Big(d^2+\log(\ep^{-1})\Big)\Big)
\\
&\leq \gamma(\Delta_n)T,
\end{align*}
concludes the proof by applying Proposition \ref{prop:rep disc trunc}. 
\end{proof}
The results of \cite{vanzanten01} now imply that $\tilde{\mathcal{S}}$ has subgaussian increments with respect to the Frobenius norm on the event $Q_T$. We can therefore apply Talagrand's majorizing measure theorem to obtain uniform bounds for $\tilde{\mathcal{S}}$ depending on the Gaussian width of its indexing set (see Exercise 8.6.5 in \cite{vers18}). We recall that the Gaussian width of a set $\mathcal{D}\subset\R^{d\times d}$ is defined as
\[w(\mathcal{D})\coloneq \E\Big[\sup_{\M\in\mathcal{D}}\vec(\M)^\top Z\Big],\quad \mathrm{where} \quad Z\sim \mathcal{N}(0,\mathbb{I}_{d^2}). \]
\begin{theorem}\label{thm:st}
Let $\mathcal{D}\subset\R^{d\times d}$ be given.  Then there exists a universal constant $c_{2,\star}>0,$ such that for any $\ep\in(0,1),$ it holds
\[\P\Big(\sup_{\M\in\mathcal{D}}\vert\tilde{\mathcal{S}}(\M)\vert\1_{Q_T}> c_{2,\star} \sqrt{\gamma(\delta_n)T}\big(w(\mathcal{D})+\sqrt{\log(2/\ep)}\operatorname{rad}(\mathcal{D})\big)\Big)\leq \ep. \]
\end{theorem}
\begin{proof}
    Let $\M_1,\M_2\in\mathcal{D}$ be given and denote $\M\coloneq \M_1-\M_2$. Then Corollary 3.4 in \cite{vanzanten01} gives for any $u>0$
    \begin{align*}
        &\P\Big(\vert\tilde{\mathcal{S}}(\M_1)\1_{Q_T}-\tilde{\mathcal{S}}(\M_2)\1_{Q_T}\vert>u\Big)
        \\
        &=\P\Big(\vert\tilde{\mathcal{S}}(\M)\vert>u,Q_T\Big)
        \\
        &\leq \P\Big(\vert\tilde{\mathcal{S}}(\M)\vert>u,\sum_{i=1}^n\xi^2_i(\M)+\sum_{i=1}^n\E[\xi^2_i(\M)\vert\mathcal{F}_{t_{i-1}}]\leq  4\gamma(\Delta_n)T\Vert\M\Vert_2^2\Big)
        \\
        &\leq 2\exp\Big(-\frac{u^2}{8\gamma(\Delta_n)T\Vert\M\Vert_2^2}\Big).
    \end{align*}
    Proposition 2.5.2 in \cite{vers18} then implies that $(\tilde{\mathcal{S}}(\M)\1_{Q_T}/(c_7 \sqrt{\gamma(\Delta_n)T}))_{\M\in\mathcal{D}}$ is sub-Gaussian in the sense of Exercise 8.6.4 in \cite{vers18}, where $c_7>0$ is a universal constant. The assertion then follows from Exercise 8.6.5 in \cite{vers18}. 
\end{proof}
By the previous theorem, it now suffices to bound the Gaussian width and radius of a suitable set in order to obtain a uniform bound of the stochastic error. 
\begin{corollary}\label{cor:fin}
 Let $\ep_1,\ep_2\in(0,1)$ be given and set \[\mathcal{D}^*(\ep_1)\coloneq \{\M\in\R^{d\times d}\colon (\Vert \M\Vert_\star\lor \log(2/\ep_1)^{1/2}\Vert\M\Vert_2)=1\}.\]  
 Then there exists a universal constant $c_\star>0,$ such that for any $T\geq T_\star(\ep_2,\eta),$
 \[\P\Big(\sup_{\M\in\mathcal{D}^*(\ep_1)}\vert\tilde{\mathcal{S}}(\M)\vert\leq c_{\star} \sqrt{\gamma(\delta_n)T},Q_T\Big)\geq 1-\ep_1-\ep_2. \]
\end{corollary}
\begin{proof}
    In order to apply Theorem \ref{thm:st} we first show that the Gaussian width $w(\mathcal{D}^*(\ep_1)),$ and $\sqrt{\log(2/\ep_1)}\operatorname{rad}(\mathcal{D}^*(\ep_1))$ are upper bounded by universal constants, independent of $d$. By the triangle inequality it holds for any $\M_1,\M_2\in\mathcal{D}^*(\ep_1),$
    \begin{align*}
        \Vert\M_1-\M_2\Vert_2&\leq \Vert \M_1\Vert_2+\Vert\M_2\Vert_2
        \\
        &\leq 2\log(2/\ep_1)^{-1/2},
    \end{align*}
    which implies
    \begin{align*}
        \sqrt{\log(2/\ep_1)}\operatorname{rad}(\mathcal{D}^*(\ep_1))\leq 2.
    \end{align*}
    For the Gaussian width, we can argue as in the proof of Proposition 3.3 in \cite{dex24} to obtain
    \begin{align*}
        w(\mathcal{D}^*(\ep_1))\leq \sqrt{\frac{\pi}{\log(4)}}
+4.    \end{align*}
Thus, Proposition \ref{prop: qt} and Theorem \ref{thm:st} imply that there exists a universal constant $c_\star>0$ such that for any $T\geq T_\star(\ep_2)$ it holds
\begin{align*}
  &\P\Big(\sup_{\M\in\mathcal{D}^*(\ep_1)}\vert\tilde{\mathcal{S}}(\M)\vert\leq c_{\star} \sqrt{\gamma(\delta_n)T},Q_T\Big)
  \\
  &=\P(Q_T)-\P\Big(\sup_{\M\in\mathcal{D}^*(\ep_1)}\vert\tilde{\mathcal{S}}(\M)\vert>c_{\star} \sqrt{\gamma(\delta_n)T},Q_T\Big)
  \\
  &\geq 1-\ep_2-\P\Big(\sup_{\M\in\mathcal{D}^*(\ep_1)}\vert\tilde{\mathcal{S}}(\M)\vert\1_{Q_T}>c_{\star} \sqrt{\gamma(\delta_n)T}\Big)
  \\
  &\geq 1-\ep_1-\ep_2,
\end{align*}
which concludes the proof.
\end{proof}
Now, it only remains to bound the truncation bias in Lemma \ref{lemma: st decomp} to prove our main results. This is the content of the next section.
\subsubsection{Truncation Bias}
In order to bound the truncation bias we have to bound the probability of the increments of the BDLP $\Z$ being larger than the truncation level $\eta$. Since we assume the BDLP $\Z$ to be a martingale, we can apply Theorem 3.3 in \cite{vanzanten01}, which gives the following result.
\begin{lemma}\label{lemma:prob bound}
For any $a>0,u>0,L\geq2\Delta_n\exp(2\Delta_n\Vert\A_0\Vert)\lambda_{\max}(\C+\boldsymbol{\nu}_2)$ it holds
\[\P(\Vert \Delta  \tilde{Z}_1\Vert\geq u)\leq 2d\exp\Big(-\Big(\frac{u^2}{4dL}\land \frac{u}{\frac
        {4}{3}a\sqrt{d}}\Big)\Big)+2\Delta_nL^{-1}\exp(2\Delta_n\Vert\A_0\Vert)\int_{B(0,\exp(-\Delta_n\Vert\A_0\Vert )a)^{\mathsf{c}}}\Vert z\Vert^2\nu(\md z).\]
\end{lemma}
\begin{proof}
The proof of the assertion relies on Theorem 3.3 in \cite{vanzanten01}. For this we first note that for any $i\in[d]$ it holds for any $a>0$ 
\begin{align*}
    &\int_0^{\Delta_n}\int_{\R^d}(\exp(-(\Delta_n-s)\A_0 )z)_i^2\1(\vert(\exp(-(\Delta_n-s)\A_0 )z)_i\vert>a)N(\md s,\md z)
        \\&\quad+\int_0^{\Delta_n}\Big(\int_{\R^d}(\exp(-(\Delta_n-s)\A_0 )z)_i^2\nu(\md z)+(\exp(-(\Delta_n-s)\A_0 )\C\exp(-(\Delta_n-s)\A_0^\top ))_{ii}\d s
        \\
        &\leq \int_0^{\Delta_n}\int_{\R^d}\Vert\exp(-(\Delta_n-s)\A_0 )z\Vert^2\1(\Vert\exp(-(\Delta_n-s)\A_0 )z\Vert>a)N(\md s,\md z)
         \\&\quad+\lambda_{\max}\Bigg(\int_0^{\Delta_n}\Big(\int_{\R^d}(\exp(-(\Delta_n-s)\A_0 )z)(\exp(-(\Delta_n-s)\A_0 )z)^\top\nu(\md z)
         \\
         &\quad+(\exp(-(\Delta_n-s)\A_0 )\C\exp(-(\Delta_n-s)\A_0^\top ))_{ii}\d s\Bigg)
         \\
         &\leq \exp(2\Delta_n\Vert\A_0\Vert)\int_0^{\Delta_n}\int_{\R^d}\Vert z\Vert^2\1(\Vert z\Vert>\exp(-\Delta_n\Vert\A_0\Vert )a)N(\md s,\md z)
         \\&\quad+\Delta_n\exp(2\Delta_n\Vert\A_0\Vert)\lambda_{\max}(\C+\boldsymbol{\nu}_2)
         \\
         &\eqcolon Q(\Delta_n,a).
\end{align*}
Theorem 3.3 in \cite{vanzanten01} combined with equation (1.2) in \cite{vanzanten01} thus gives for any $a,u,L>0$
    \begin{align*}
        &\P(\Vert \Delta  \tilde{Z}_1\Vert\geq u)
        \\
        &\leq \P(\sqrt{d}\Vert \Delta  \tilde{Z}_1\Vert_\infty\geq u ,Q(\Delta_n,a)\leq L)+\P(Q(\Delta_n,a)>L)
        \\
        &\leq \P(Q(\Delta_n,a)>L)+\sum_{i=1}^d\P(\vert(\Delta  \tilde{Z}_1)_i\vert\geq u/\sqrt{d} ,Q(\Delta_n,a)\leq L)
        \\
        &\leq \P(Q(\Delta_n,a)>L)+2d\exp\Big(-\frac{u^2}{2dL}\frac{1}{1+\frac{au}{3\sqrt{d}L}}\big)\Big)
         \\
        &=\P(Q(\Delta_n,a)>L)+2d\exp\Big(-\frac{u^2}{2dL+\frac
        {2}{3}au\sqrt{d}}\Big)
        \\
        &\leq \P(Q(\Delta_n,a)>L)+2d\exp\Big(-\Big(\frac{u^2}{4dL}\land \frac{u}{\frac
        {4}{3}a\sqrt{d}}\Big)\Big).
    \end{align*}
    Then setting 
    \[\tilde{L}\coloneq L-\Delta_n\exp(2\Delta_n\Vert\A_0\Vert)\lambda_{\max}(\C+\boldsymbol{\nu}_2),\quad L>0,\]
   we obtain 
    \begin{align*}
        \P(Q(\Delta_n,a)>L)&=\P\Big(\exp(2\Delta_n\Vert\A_0\Vert)\int_0^{\Delta_n}\int_{B(0,\exp(-\Delta_n\Vert\A_0\Vert )a)^{\mathsf{c}}}\Vert z\Vert^2N(\md s,\md z)>\tilde{L}\Big)
        \\
        &\leq \Delta_n\tilde{L}^{-1}\exp(2\Delta_n\Vert\A_0\Vert)\int_{B(0,\exp(-\Delta_n\Vert\A_0\Vert )a)^{\mathsf{c}}}\Vert z\Vert^2\nu(\md z)
                \\
        &\leq 2\Delta_nL^{-1}\exp(2\Delta_n\Vert\A_0\Vert)\int_{B(0,\exp(-\Delta_n\Vert\A_0\Vert )a)^{\mathsf{c}}}\Vert z\Vert^2\nu(\md z).
    \end{align*}
    This concludes the proof.
\end{proof}
The previous result now allows us to bound the truncation bias under varying assumptions on the tail of the BDLP's L{\'e}vy measure. In some cases we introduce the technical assumption of the number of observations $n$ growing at most polynomially in the observation length $T$. This facilitates the proofs considerably and can always be fulfilled in practice by suppressing observations.
\begin{proposition}\label{prop: eta}
Assume that $n\leq c_\delta T^\delta$. Then the following hold true.
\begin{enumerate}
    \item Assume that $\Z$ is continuous, i.e. $\nu(\R^d)=0$. Then for any 
    \[\eta\geq\eta^{\mathrm{cont}}_0\coloneq \sqrt{32\delta \log(T)d\Delta_n\exp(2\Delta_n\Vert\A_0\Vert)\lambda_{\max}(\C+\boldsymbol{\nu}_2)}\] it holds
\[\Delta_n^{-1}\P(\Vert \Delta  \tilde{Z}_1\Vert\geq \eta/2)\leq \frac{2\e c_\delta d}{T}.\]
\item Assume that $\Z$ has bounded jumps, i.e. there exists an $a_0>0,$ such that $\nu(B(0,a_0)^{\mathsf{c}})=0$. Then for any 
\[\eta\geq \eta^{\mathrm{bounded}}_0\coloneq \sqrt{d\delta\log(T)}\exp(\Delta_n\Vert \A_0\Vert)\Big(\sqrt{32\Delta_n\lambda_{\max}(\C+\boldsymbol{\nu}_2)}\lor \frac{8}{3}a_0\sqrt{\delta\log(T)}\Big), \]
it holds
\[\Delta_n^{-1}\P(\Vert \Delta  \tilde{Z}_1\Vert\geq \eta/2)\leq \frac{2\e c_\delta d}{T}.\]
\item Assume that $\nu$ is sub-Weibull with parameter $\alpha>0$. Then for any
\[   \eta\geq \eta^{\mathrm{Wei},\alpha}_{0}\coloneq \sqrt{d}\exp(\Delta_n\Vert\A_0\Vert)\Big(\sqrt{32\delta\log(T)\lambda_{\max}(\C+\boldsymbol{\nu}_2)}\lor \frac{8\delta\log(T)^{1+1/\alpha}}{3c_{\alpha}^{1/\alpha}}\Big),\] it holds 
\[\Delta_n^{-1}\P(\Vert \Delta  \tilde{Z}_1\Vert\geq \eta/2)\leq \frac{2\e c_\delta d}{T}.\]
\item Assume that $\nu$ admits a $p$-th moment for some $p\geq 2$. Then for any
\[\eta\geq \eta^{\mathrm{poly},p}_0\coloneq T^{1/p}d^{1/2-1/p},\] it holds
\begin{align*}
   & \Delta_n^{-1}\P(\Vert \Delta  \tilde{Z}_1\Vert\geq \eta/2) 
   \\
   &\leq\frac{ 2^{2p-1}d}{T}\e^{p\Delta_n\Vert\A_0\Vert_{\mathrm{op}}}\Big(\Delta_n^{p/2-1}\big(\e\lambda_{\max}(\C)^{p/2} (p/d+1)^{p/2}+c_1(\nu_2/d)^{p/2}\big)+c_1\nu_pd^{-p/2}\Big).
\end{align*}
\end{enumerate}
\end{proposition}
\begin{proof}
We first investigate the continuous case, i.e. $\nu(\R^d)=0.$ Then, we can choose $a$ in Lemma \ref{lemma:prob bound} arbitrarily small to obtain
\begin{align*}
    &\Delta_n^{-1}\P(\Vert \Delta  \tilde{Z}_1\Vert\geq \eta/2)
    \\
    &\leq \frac{2d}{\Delta_n}\exp\Big(-\frac{\eta^2}{32d\Delta_n\exp(2\Delta_n\Vert\A_0\Vert)\lambda_{\max}(\C+\boldsymbol{\nu}_2)}\Big)
    \\
    &\leq \frac{2c_\delta d}{T}\exp\Big(\delta\log(T)-\frac{\eta^2}{32d\Delta_n\exp(2\Delta_n\Vert\A_0\Vert)\lambda_{\max}(\C+\boldsymbol{\nu}_2)}\Big),
\end{align*}
and thus choosing $\eta\geq\eta^{\mathrm{cont}}_0\coloneq \sqrt{32\delta \log(T)d\Delta_n\exp(2\Delta_n\Vert\A_0\Vert)\lambda_{\max}(\C+\boldsymbol{\nu}_2)}$ gives
\[\Delta_n^{-1}\P(\Vert \Delta  \tilde{Z}_1\Vert\geq \eta/2)\leq \frac{2 \e c_\delta d}{T}.\]
Now, assume that $\Z$ has bounded jumps, i.e. there exists $a_0>0,$ such that 
\begin{align*}
    \nu(B(0,a_0)^{\mathsf{c}})=0.
\end{align*}
Then we obtain by Lemma \ref{lemma:prob bound} 
\begin{align*}
    &\Delta_n^{-1}\P(\Vert \Delta  \tilde{Z}_1\Vert\geq \eta/2)
    \\
    &\leq \frac{2c_\delta d}{T}\exp\Big(\delta\log(T)-\Big(\frac{\eta^2}{32d\Delta_n\exp(2\Delta_n\Vert\A_0\Vert)\lambda_{\max}(\C+\boldsymbol{\nu}_2)}\land \frac{\eta}{\frac
        {8}{3}\exp(\Delta_n\Vert\A_0\Vert)a_0\sqrt{d}}\Big)\Big).
\end{align*}
Thus choosing
\[\eta\geq \eta^{\mathrm{bounded}}_0\coloneq \sqrt{d\delta\log(T)}\exp(\Delta_n\Vert \A_0\Vert)\Big(\sqrt{32\Delta_n\lambda_{\max}(\C+\boldsymbol{\nu}_2)}\lor \frac{8}{3}a_0\sqrt{\delta\log(T)}\Big), \]
concludes the proof.
Now assume that  $\nu$ is sub-Weibull with parameter $\alpha>0,$ i.e. there exists a constant $c_\alpha>0,$ such that
\begin{align*}
  \int \Vert z\Vert^2\exp(c_\alpha\Vert z\Vert^{\alpha})\nu(\md z)\leq 2\nu_2.
\end{align*}
Then we obtain by Lemma \ref{lemma:prob bound} for any $u,a>0, L>2\Delta_n\exp(2\Delta_n\Vert\A_0\Vert)\lambda_{\max}(\C+\boldsymbol{\nu}_2)$
\begin{align*}
    &\Delta_n^{-1}\P(\Vert \Delta  \tilde{Z}_1\Vert\geq \eta/2)
    \\
    &\leq 2\Delta_n^{-1}d\exp\Big(-\Big(\frac{\eta^2}{16dL}\land \frac{\eta}{\frac
        {8}{3}a\sqrt{d}}\Big)\Big)+2L^{-1}\exp(2\Delta_n\Vert\A_0\Vert)\int_{B(0,\exp(-\Delta_n\Vert\A_0\Vert )a)^{\mathsf{c}}}\Vert z\Vert^2\nu(\md z)
        \\
        &\leq 2\Delta_n^{-1}d\exp\Big(-\Big(\frac{\eta^2}{16dL}\land \frac{\eta}{\frac
        {8}{3}a\sqrt{d}}\Big)\Big)+4\nu_2L^{-1}\exp(2\Delta_n\Vert\A_0\Vert)\exp(-c_\alpha(\exp(-\Delta_n\Vert\A_0\Vert )a)^\alpha).
\end{align*}
Choosing $a=\exp(\Delta_n\Vert\A_0\Vert)(c_\alpha^{-1}\log(T))^{1/\alpha},L=2\exp(2\Delta_n\Vert\A_0\Vert)\lambda_{\max}(\C+\boldsymbol{\nu}_2)$ thus implies if $\Delta_n\leq 1$
\begin{align*}
    &\Delta_n^{-1}\P(\Vert \Delta  \tilde{Z}_1\Vert\geq \eta/2)
        \\
        &\leq 2\Delta_n^{-1}d\exp\Big(-\Big(\frac{\eta^2}{32\exp(2\Delta_n\Vert\A_0\Vert)\lambda_{\max}(\C+\boldsymbol{\nu}_2)d}\land \frac{\eta}{\frac
        {8}{3}\exp(\Delta_n\Vert\A_0\Vert)(c_\alpha^{-1}\log(T))^{1/\alpha}\sqrt{d}}\Big)\Big)
        \\
        &\quad+2\frac{\nu_2}{\lambda_{\max}(\C+\boldsymbol{\nu}_2)T}
             \\
        &\leq 2c_\delta\frac{d}{T}\exp\Big(\delta\log(T)-\Big(\frac{\eta^2}{32\exp(2\Delta_n\Vert\A_0\Vert)\lambda_{\max}(\C+\boldsymbol{\nu}_2)d}\land \frac{\eta}{\frac
        {8}{3}\exp(\Delta_n\Vert\A_0\Vert)(c_\alpha^{-1}\log(T))^{1/\alpha}\sqrt{d}}\Big)\Big)
        \\
        &\quad+2\frac{d}{T}
\end{align*}
Hence choosing
\begin{align*}
    \eta\geq \eta^{\mathrm{Wei},\alpha}_{0}\coloneq \sqrt{d}\exp(\Delta_n\Vert\A_0\Vert)\Big(\sqrt{32\delta\log(T)\lambda_{\max}(\C+\boldsymbol{\nu}_2)}\lor \frac{8\delta\log(T)^{1+1/\alpha}}{3c_{\alpha}^{1/\alpha}}\Big)
\end{align*}
concludes the proof.
Lastly, we inspect the polynomial case, i.e. we assume that there exists $p\geq 2,$ such that
\begin{align*}
\nu_p=\int\Vert z\Vert^p\nu(\md z)<\infty.
\end{align*}
For this case we directly apply Markov's inequality, which implies
\begin{align}
\begin{split}\label{eq: p mark}
     &\Delta_n^{-1}\P(\Vert \Delta  \tilde{Z}_1\Vert\geq \eta/2) 
   \\
   &\leq 2^p\Delta_n^{-1}\eta^{-p}\E[\Vert \Delta  \tilde{Z}_1\Vert^p]
   \\
   &\leq\frac{ 2^{2p-1}}{\Delta_n\eta^{p}}\Big(\E\Big[\Vert \int_0^{\Delta_n} \exp(-(\Delta_n-s)\A_0)\bSigma\d W_s\Vert^p\Big]+\E\Big[\Vert \int_0^{\Delta_n} \int_{\R^d} \exp(-(\Delta_n-s)\A_0)z\tilde{N}(\d s,\d z)\Vert^p\Big]\Big).  
\end{split}
\end{align}
As $\int_{0}^{\Delta_n}\e^{-(\Delta_n-s)\A_0}\bSigma\d W_s$ is distributed according to a centered multivariate normal distribution with covariance matrix $\int_{0}^{\Delta_n}\e^{-(\Delta_n-u)\A_0}\C\e^{-(\Delta_n-u)\A_0^\top}\d s,$ we obtain for $\mathsf{N}$ denoting a d-dimensional standard normal random vector
\begin{align*}
    &\E[\Vert \int_{0}^{\Delta_n}\e^{-(\Delta_n-s)\A_0}\bSigma\d W_s\Vert^p]
    \\
    &\leq \lambda_{\max}\Big(\int_{0}^{\Delta_n}\e^{-(\Delta_n-u)\A_0}\C\e^{-(\Delta_n-u)\A_0^\top}\d s\Big)^{p/2}  \E[\Vert N\Vert^p]
    \\
    &=\lambda_{\max}\Big(\int_{0}^{\Delta_n}\e^{-(\Delta_n-u)\A_0}\C\e^{-(\Delta_n-u)\A_0^\top}\d s\Big)^{p/2} 2^{p/2}\frac{\Gamma\big(\frac{p+d}{2}\big)}{\Gamma\big(\frac{d}{2}\big)}
    \\
    &\leq (2\Delta_n\lambda_{\max}(\C) \exp(2\Delta_n\Vert \A_0\Vert))^{p/2}\frac{\Gamma\big(\frac{p+d}{2}\big)}{\Gamma\big(\frac{d}{2}\big)}
    \\
    &\leq\e (2\Delta_n\lambda_{\max}(\C) \exp(2\Delta_n\Vert \A_0\Vert))^{p/2} \Big(\frac{p+d}{2\e}\Big)^{p/2}\Big(\frac{p+d}{d}\Big)^{d/2}
       \\
    &\leq \e(\Delta_n\lambda_{\max}(\C) \exp(2\Delta_n\Vert \A_0\Vert))^{p/2} (p+d)^{p/2},
\end{align*}
where we used Stirling's formula together with the inequality $1+x\leq\exp(x),$ which is valid for any $x\in\R,$ in the last steps.
Additionally, inspecting the proof of Theorem 4.4.23 in \cite{applebaum09} gives that it holds for some constant $c_1>0,$ depending only on $p$ and independent of the dimension, 
\begin{align*}
        &\E[\Vert \int_{0}^{\Delta_n}\int_{\R^d}\e^{-(\Delta_n-u)\A_0}z\tilde{N}(\d s,\d z)\Vert^p]
\\
    &\leq c_1\Big( \int_{0}^{\Delta_n}\int_{\R^d}\Vert\e^{-(\Delta_n-u)\A_0}z\Vert^p\nu(\md z)\d s+\Big( \int_{0}^{\Delta_n}\int_{\R^d}\Vert\e^{-(\Delta_n-u)\A_0}z\Vert^2\nu(\md z)\d s\Big)^{p/2}\Big)  
    \\
    &\leq c_1\Delta_n\e^{p\Delta_n\Vert\A_0\Vert_{\mathrm{op}}}(\nu_p+\Delta_n^{p/2-1}\nu_2^{p/2}). 
\end{align*}
Plugging these bounds into \eqref{eq: p mark} gives
\begin{align*}
   & \Delta_n^{-1}\P(\Vert \Delta  \tilde{Z}_1\Vert\geq \eta/2) 
   \\
   &\leq \frac{ 2^{2p-1}}{\eta^{p}}\e^{p\Delta_n\Vert\A_0\Vert_{\mathrm{op}}}\Big(\Delta_n^{p/2-1}\e\lambda_{\max}(\C)^{p/2} (p+d)^{p/2}+c_1(\nu_p+\Delta_n^{p/2-1}\nu_2^{p/2})\Big),
\end{align*}
and thus for any 
\[\eta\geq \eta^{\mathrm{poly},p}_0\coloneq T^{1/p}d^{1/2-1/p},\]
it holds
\begin{align*}
    & \Delta_n^{-1}\P(\Vert \Delta  \tilde{Z}_1\Vert\geq \eta/2) 
   \\
   &\leq \frac{ 2^{2p-1}d}{T}\e^{p\Delta_n\Vert\A_0\Vert_{\mathrm{op}}}\Big(\Delta_n^{p/2-1}\big(\e\lambda_{\max}(\C)^{p/2} (p/d+1)^{p/2}+c_1(\nu_2/d)^{p/2}\Big)+c_1\nu_pd^{-p/2}\Big).
\end{align*}
\end{proof}
\subsection{Proof of Theorem \ref{thm: main}}
We are now ready to state the proof of our main result.
\begin{proof}[Proof of Theorem \ref{thm: main}]
    Throughout the proof we assume that the event 
    \begin{equation}\label{eq: lambda}
        \   \Lambda\coloneq\Big\{\sup_{\M\in\mathcal{D}^*(\ep)}\vert\tilde{\mathcal{S}}(\M)\vert\leq c_{\star} \sqrt{\gamma(\delta_n)T}\Big\}\cap Q_T    
    \end{equation} occurs, which by Corollary \ref{cor:fin} happens with probability larger than $1-\ep_0-\ep$.
    First, we obtain by Lemma \ref{lemma: L2 frobenius disc} that
    \begin{align*}
        &\Vert(\hat{\A}-\A_0)X\|_{l^2_n(B,X,\eta)}^2 \\
    &\leq\|(\A-\A_0) X\|_{l^2_n(B,X,\eta)}^2-\|(\hat{\A}-\A)X\|_{l^2_n(B,X,\eta)}^2 + 2 \left(h(\A)-h(\hat{\A}) \right) \\ 
    & \quad + \frac{2}{T} \Big(  \mathcal{D}(\A-\hat{\A})+\mathcal{S}(\A-\hat{\A})\Big),
    \end{align*}
    where $\mathcal{D}$ and $\mathcal{S}$ are defined in \eqref{eq:def errs}.
    Now, applying Proposition \ref{prop: discretizationdex} and Lemma \ref{lemma: st decomp} with $c=3$ gives
    \begin{align*}
       &\Vert(\hat{\A}-\A_0)X\|_{l^2_n(B,X,\eta)}^2 \\
    &\leq\|(\A-\A_0) X\|_{l^2_n(B,X,\eta)}^2-\frac{1}{3}\|(\hat{\A}-\A)X\|_{l^2_n(B,X,\eta)}^2 + 2 \left(h(\A)-h(\hat{\A}) \right) \\ 
    & \quad+3\Delta_n^2\Vert \hat{\C}^\eta_n(B)\Vert\Vert\A_0\Vert_{2}^4\Big(\frac{1}{2}+\frac{1}{6}\exp(\Delta_n\Vert\A_0\Vert_{2})\Delta_n\Vert\A_0\Vert_{2}\Big)^2 
    \\
&\quad +\frac{3}{\Delta_n} \lambda_{\max}(\C+\boldsymbol{\nu}_2)\exp(\Delta_n\Vert\A_0\Vert)\P(\Vert \Delta  \tilde{Z}_1\Vert\geq \eta_- )  
    \\&\quad+ \frac{2}{T}\vert\tilde{\mathcal{S}}(\A-\hat{\A})\vert.
    \end{align*}
Since the event $\Lambda$ occurs, we can bound $\tilde{\mathcal{S}}(\A-\hat{\A})$ in the following way
  \begin{align}
  \begin{split}\label{eq: thm proof}
         &\Vert(\hat{\A}-\A_0)X\|_{l^2_n(B,X,\eta)}^2 +h(\A-\hat{\A})\\
    &\leq\|(\A-\A_0) X\|_{l^2_n(B,X,\eta)}^2-\frac{1}{3}\|(\hat{\A}-\A)X\|_{l^2_n(B,X,\eta)}^2  \\ 
    & \quad+3\Delta_n^2\lambda_{\max}(\C_\infty(B))\Vert\A_0\Vert_{2}^4\Big(\frac{1}{2}+\frac{1}{6}\exp(\Delta_n\Vert\A_0\Vert_{2})\Delta_n\Vert\A_0\Vert_{2}\Big)^2 
    \\
&\quad +\frac{3}{\Delta_n} \lambda_{\max}(\C+\boldsymbol{\nu}_2)\exp(\Delta_n\Vert\A_0\Vert)\P(\Vert \Delta  \tilde{Z}_1\Vert\geq \eta_- )  
    \\&\quad+ 2\Bigg(c_\star\sqrt{\frac{\gamma(\Delta_n)}{T}}(\Vert \A-\hat{\A}\Vert_\star\lor \log(2/\ep)^{1/2}\Vert\A-\hat{\A}\Vert_2)+\frac{1}{2}h(\A-\hat{\A})+h(\A)-h(\hat{\A})\Bigg),    
  \end{split}
    \end{align}
    where we also added $h(\A-\hat{\A})$ on both sides of the inequality. In order to prove the assertions, it now suffices to bound the last summand
    \[\Upsilon\coloneq 2\Bigg(c_\star\sqrt{\frac{\gamma(\Delta_n)}{T}}(\Vert \A-\hat{\A}\Vert_\star\lor \log(2/\ep)^{1/2}\Vert\A-\hat{\A}\Vert_2)+\frac{1}{2}h(\A-\hat{\A})+h(\A)-h(\hat{\A})\Bigg).\] For this we start with the Lasso case in which $h\equiv \Vert\cdot\Vert_1$. By the Cauchy--Schwarz inequality and equation (2.7) in \cite{belets18} we obtain for any $\M\in\R^{d\times d}$
    \begin{align*}
        \Vert \M\Vert_\star&\leq \sqrt{\sum_{i=1}^s\log\Big(\frac{2 d^2}{i}\Big)}\Vert \M\Vert_2+\sum_{i=s+1}^{d^2}\sqrt{\log\Big(\frac{2d^2}{i}\Big)}\vec(\A-\hat{\A})_i^{\#}
        \\
        &\leq \sqrt{s\log\Big(\frac{2\e d^2}{s}\Big)}\Vert \M\Vert_2+\sum_{i=s+1}^{d^2}\sqrt{\log\Big(\frac{2d^2}{i}\Big)}\vec(\A-\hat{\A})_i^{\#},
    \end{align*}
    and since the event $\Lambda$ occurs it also holds
\begin{align}
\begin{split}\label{eq: lambda cov}
   \Vert\M\Vert_2^2&\leq \lambda_{\min}(\hat{\C}^\eta_n)^{-1}\tr(\M\hat{\C}^\eta_n\M^\top)
\\
&=\lambda_{\min}(\hat{\C}^\eta_n)^{-1}\Vert \M\Vert^2_{l^2_n(B,X,\eta)}
\\
&\leq 4\lambda_{\min}(\C_\infty(B))^{-1}\P(\Vert\Delta\tilde{Z}_1\Vert\leq\eta_-)^{-1}\Vert \M\Vert^2_{l^2_n(B,X,\eta)}. 
\end{split}
\end{align}
Combining this with the assumption on the tuning parameter $\lambda_L(T)\geq 2c_\star\sqrt{\log(2ed^2/s)\gamma(\Delta_n)/T}$ and Lemma A.1 in \cite{belets18} gives
\begin{align}
\begin{split}\label{eq: thm 2}
      &\Upsilon
    \\
    &\leq \lambda_L\Bigg(\log(2\e d^2/s)^{-1/2}(\Vert \A-\hat{\A}\Vert_\star\lor \log(2/\ep)^{1/2}\Vert\A-\hat{\A}\Vert_2)+3\sqrt{s}\Vert\A-\hat{\A}\Vert_2-\sum_{i=s+1}^{d^2}\vec(\A-\hat{\A})^{\#}_i\Bigg)
    \\
    &\leq \lambda_L\Bigg(\Big(\sqrt{s}\Vert \A-\hat{\A}\Vert_2+\sum_{i=s+1}^{d^2}\vec(\A-\hat{\A})^{\#}_i\lor \frac{2\log(2/\ep)^{1/2}\Vert\A-\hat{\A}\Vert_{l^2_n(B,X,\eta)}}{(\log(2\e d^2/s)\lambda_{\min}(\C_\infty(B))\P(\Vert\Delta\tilde{Z}_1\Vert\leq \eta_-))^{1/2}}\Big)
    \\&\quad+3\sqrt{s}\Vert\A-\hat{\A}\Vert_2-\sum_{i=s+1}^{d^2}\vec(\A-\hat{\A})^{\#}_i\Bigg),  
\end{split}
\end{align}

We continue with considering the two different cases due to the maximum term. First, we have that if 
\[\sqrt{s}\Vert \A-\hat{\A}\Vert_2+\sum_{i=s+1}^{d^2}\vec(\A-\hat{\A})^{\#}_i\geq 2\log(2/\ep)^{1/2}\frac{\Vert\A-\hat{\A}\Vert_{l^2_n(B,X,\eta)}}{(\log(2\e d^2/s)\lambda_{\min}(\C_\infty(B))\P(\Vert\Delta\tilde{Z}_1\Vert\leq\eta_-))^{1/2}}\]
it holds
\begin{align*}
    \Upsilon&\leq 4\sqrt{s}\lambda_L\Vert \A-\hat{\A}\Vert_2
    \\
    &\leq 8\lambda_{\min}(\C_\infty(B))^{-1/2}\P(\Vert\Delta\tilde{Z}_1\Vert\leq\eta_-)^{-1/2}\sqrt{s}\lambda_L\Vert \A-\hat{\A}\Vert_{l^2_n(B,X,\eta)}
    \\
    &\leq 48\lambda_{\min}(\C_\infty(B))^{-1}\P(\Vert\Delta\tilde{Z}_1\Vert\leq\eta_-)^{-1}s\lambda^2_L+\frac{1}{3}\Vert \A-\hat{\A}\Vert^2_{l^2_n(B,X,\eta)},
\end{align*}
where we used Young's inequality in the last step.
In the opposite case, it follows that
\[\sqrt{s}\Vert \A-\hat{\A}\Vert_2\leq 2\log(2/\ep)^{1/2}\frac{\Vert\A-\hat{\A}\Vert_{l^2_n(B,X,\eta)}}{(\log(2ed^2/s)\lambda_{\min}(\C_\infty(B))\P(\Vert\Delta\tilde{Z}_1\Vert\leq\eta_-))^{1/2}},\]
which implies
\begin{align*}
    \Upsilon&\leq \lambda_L\Bigg( 2\log(2/\ep)^{1/2}\frac{\Vert\A-\hat{\A}\Vert_{l^2_n(B,X,\eta)}}{(\log(2\e d^2/s)\lambda_{\min}(\C_\infty(B))\P(\Vert\Delta\tilde{Z}_1\Vert\leq\eta_-))^{1/2}}
    +3\sqrt{s}\Vert\A-\hat{\A}\Vert_2\Bigg)
    \\
    &\leq 8\lambda_L\log(2/\ep)^{1/2}\frac{\Vert\A-\hat{\A}\Vert_{l^2_n(B,X,\eta)}}{(\log(2\e d^2/s)\lambda_{\min}(\C_\infty(B))\P(\Vert\Delta\tilde{Z}_1\Vert\leq\eta_-))^{1/2}}
       \\
    &\leq 48\frac{\lambda^2_L\log(2/\ep)}{\log(2\e d^2/s)\lambda_{\min}(\C_\infty(B))\P(\Vert\Delta\tilde{Z}_1\Vert\leq\eta_-)}+\frac{1}{3}\Vert \A-\hat{\A}\Vert^2_{l^2_n(B,X,\eta)},
\end{align*}
where we again applied Young's inequality. Inserting the bounds on $\Upsilon$ in \eqref{eq: thm proof} and noting that the assumption on $\eta$ implies $\eta_-\geq \eta/2$ completes the proof for the Lasso estimator. We continue by bounding $\Upsilon$ for the Slope estimator. By the assumption $\lambda_S(T)\geq 2c_\star\sqrt{\frac{\gamma(\Delta_n)}{T}}$ we get
\begin{align*}
    \Upsilon&\leq  \lambda_S\Bigg((\Vert \A-\hat{\A}\Vert_\star\lor \log(2/\ep)^{1/2}\Vert\A-\hat{\A}\Vert_2)+\Vert\A-\hat{\A}\Vert_\star+2\Vert\A\Vert_\star-2\Vert\hat{\A}\Vert_\star\Bigg),
\end{align*}
and by Lemma A.1 in \cite{belets18} it holds
\begin{align*}
   &\frac{1}{2}\Vert\A-\hat{\A}\Vert_\star+\Vert\A\Vert_\star-\Vert\hat{\A}\Vert_\star 
   \\
   &\leq \frac{3}{2}\sqrt{\sum_{i=1}^s\log\Big(\frac{2 d^2}{i}\Big)}\Vert \A-\hat{\A}\Vert_2-\frac{1}{2}\sum_{i=s+1}^{d^2}\sqrt{\log\Big(\frac{2d^2}{i}\Big)}\vec(\A-\hat{\A})_i^{\#}
   \\
   &\leq \frac{3}{2}\sqrt{s\log\Big(\frac{2\e d^2}{s}\Big)}\Vert \A-\hat{\A}\Vert_2-\frac{1}{2}\sum_{i=s+1}^{d^2}\sqrt{\log\Big(\frac{2 d^2}{i}\Big)}\vec(\A-\hat{\A})_i^{\#}
\end{align*}
where we used equation (2.7) in \cite{belets18} in the last step. Arguing as in the derivation of \eqref{eq: thm 2} we thus obtain
\begin{align*}
\Upsilon
&\leq \lambda_S\Bigg(\Big(\sqrt{s\log\Big(\frac{2\e d^2}{s}\Big)}\Vert \A-\hat{\A}\Vert_2+\sum_{i=s+1}^{d^2}\vec(\A-\hat{\A})^{\#}_i\lor \frac{2\log(2/\ep)^{1/2}\Vert\A-\hat{\A}\Vert_{l^2_n(B,X,\eta)}}{(\lambda_{\min}(\C_\infty(B))\P(\Vert\Delta\tilde{Z}_1\Vert\leq\eta_-))^{1/2}}\Big)
    \\&\quad+3\sqrt{s\log\Big(\frac{2\e d^2}{s}\Big)}\Vert \A-\hat{\A}\Vert_2-\sum_{i=s+1}^{d^2}\sqrt{\log\Big(\frac{2 d^2}{i}\Big)}\vec(\A-\hat{A})_i^{\#}\Bigg),
\end{align*}
and we again continue by investigating the two cases caused by the maximum operator. If
\[\sqrt{s\log\Big(\frac{2\e d^2}{s}\Big)}\Vert \A-\hat{\A}\Vert_2+\sum_{i=s+1}^{d^2}\vec(\A-\hat{\A})^{\#}_i\geq \frac{2\log(2/\ep)^{1/2}\Vert\A-\hat{\A}\Vert_{l^2_n(B,X,\eta)}}{(\lambda_{\min}(\C_\infty(B))\P(\Vert\Delta\tilde{Z}_1\Vert\leq\eta_-))^{1/2}},\]
we obtain by arguing as in the Lasso case
\begin{align*}
    \Upsilon&\leq 4\lambda_S\sqrt{s\log\Big(\frac{2\e d^2}{s}\Big)}\Vert \A-\hat{\A}\Vert_2
    \\
    &\leq 8\lambda_S\lambda_{\min}(\C_\infty(B))^{-1/2}\P(\Vert\Delta\tilde{Z}_1\Vert\leq\eta_-)^{-1/2}\sqrt{s\log\Big(\frac{2\e d^2}{s}\Big)}\Vert \A-\hat{\A}\Vert_{l^2_n(B,X,\eta)}
    \\
    &\leq 48\lambda^2_S\lambda_{\min}(\C_\infty(B))^{-1}\P(\Vert\Delta\tilde{Z}_1\Vert\leq\eta_-)^{-1}s\log\Big(\frac{2\e d^2}{s}\Big)+\frac{1}{3}\Vert \A-\hat{\A}\Vert^2_{l^2_n(B,X,\eta)}.
\end{align*}
Similarly as before, the other case implies
\[\sqrt{s\log\Big(\frac{2\e d^2}{s}\Big)}\Vert \A-\hat{\A}\Vert_2\leq \frac{2\lambda_S\log(2/\ep)^{1/2}\Vert\A-\hat{\A}\Vert_{l^2_n(B,X,\eta)}}{(\lambda_{\min}(\C_\infty(B))\P(\Vert\Delta\tilde{Z}_1\Vert\leq\eta_-))^{1/2}},\]
and thus
\begin{align*}
    \Upsilon&\leq \frac{8\lambda_S\log(2/\ep)^{1/2}\Vert\A-\hat{\A}\Vert_{l^2_n(B,X,\eta)}}{(\lambda_{\min}(\C_\infty(B))\P(\Vert\Delta\tilde{Z}_1\Vert\leq\eta_-))^{1/2}}
    \\
    &\leq 48\frac{\lambda_S^2\log(2/\ep)}{\lambda_{\min}(\C_\infty(B))\P(\Vert\Delta\tilde{Z}_1\Vert\leq\eta_-)}+\frac{1}{3}\Vert \A-\hat{\A}\Vert^2_{l^2_n(B,X,\eta)}.
\end{align*}
Inserting the above bounds in \eqref{eq: thm proof} together with the assumption on $\eta$ concludes the proof.
\end{proof}
Corollary \ref{cor: main} follows almost immediately from Theorem \ref{thm: main}.
\begin{proof}[Proof of Corollary \ref{cor: main}]
We only give the proof for the Slope estimator, the proof for the Lasso estimator is completely analogous.
Since the assumptions of Theorem \ref{thm: main} are satisfied, we have that for any $\ep>0$ the event $\Lambda$ defined in \eqref{eq: lambda} and the bound 
 \begin{align*}
                 &\Vert(\hat{\A}_{\mathsf{S}}-\A_0)X\|_{l^2_n(B,X,\eta)}^2 +\lambda_{\mathsf{S}}\Vert\hat{\A}_{\mathsf{S}}-\A\Vert_\star\\
    &\leq\|(\A-\A_0) X\|_{l^2_n(B,X,\eta)}^2 +3\Delta_n^2\lambda_{\max}(\C_\infty(B))\Vert\A_0\Vert_{2}^4\Big(\frac{1}{2}+\frac{1}{6}\exp(\Delta_n\Vert\A_0\Vert_{2})\Delta_n\Vert\A_0\Vert_{2}\Big)^2 
    \\
&\quad +\frac{3}{\Delta_n} \lambda_{\max}(\C+\boldsymbol{\nu}_2)\exp(\Delta_n\Vert\A_0\Vert)\P(\Vert \Delta  \tilde{Z}_1\Vert\geq \eta/2 )  
    \\&\quad+ \frac{48\lambda^2_{\mathsf{S}}}{\lambda_{\min}(\C_\infty(B))\P(\Vert\Delta\tilde{Z}_1\Vert\leq\eta/2)}\Bigg(s\log\Big(\frac{2\e d^2}{s}\Big)\lor\log(2/\ep)\Bigg),
    \end{align*}
    hold true with probability larger than $1-\ep_0-\ep$ for any $s$-sparse matrix $\A$. Now, since $\A_0$ is assumed to be $s$-sparse, we can choose $\A=\A_0$, which gives the following together with the fact that $\eta$ satisfies condition \ref{ass: eta}
     \begin{align*}
                 &\Vert(\hat{\A}_{\mathsf{S}}-\A_0)X\|_{l^2_n(B,X,\eta)}^2 +\lambda_{\mathsf{S}}\Vert\hat{\A}_{\mathsf{S}}-\A_0\Vert_\star\\
           &\leq3\Delta_n^2\lambda_{\max}(\C_\infty(B))\Vert\A_0\Vert_{2}^4\Big(\frac{1}{2}+\frac{1}{6}\exp(\Delta_n\Vert\A_0\Vert_{2})\Delta_n\Vert\A_0\Vert_{2}\Big)^2 
    \\
&\quad +3c_\eta \lambda_{\max}(\C+\boldsymbol{\nu}_2)\exp(\Delta_n\Vert\A_0\Vert)\P(\Vert \Delta  \tilde{Z}_1\Vert\geq \eta/2 )  
    \\&\quad+ \frac{48\lambda^2_{\mathsf{S}}}{\lambda_{\min}(\C_\infty(B))\P(\Vert\Delta\tilde{Z}_1\Vert\leq\eta/2)}\Bigg(s\log\Big(\frac{2\e d^2}{s}\Big)\lor\log(2/\ep)\Bigg)       
                 \\
    &\leq3\Delta_n^2\lambda_{\max}(\C_\infty(B))\Vert\A_0\Vert_{2}^4\Big(\frac{1}{2}+\frac{1}{6}\exp(\Delta_n\Vert\A_0\Vert_{2})\Delta_n\Vert\A_0\Vert_{2}\Big)^2 
    \\
 &\quad+\lambda_{\mathsf{S}}^2\Big(\frac{3c_\eta}{4c_\star^2}
+ \frac{48}{\lambda_{\min}(\C_\infty(B))\P(\Vert\Delta\tilde{Z}_1\Vert\leq\eta/2)}\Big)\Bigg(s\log\Big(\frac{2\e d^2}{s}\Big)\lor\log(2/\ep)\Bigg),
    \end{align*}
    where we used the assumption on $\lambda_{\mathsf{S}},$ together with the fact that assumption \ref{ass: ergodicity} implies $s\geq d$ in the last step. Now since \eqref{eq: lambda cov} is satisfied on $\Lambda$ and $\eta_-\geq \eta/2$ we obtain
    \begin{align*}
                       &\Vert(\hat{\A}_{\mathsf{S}}-\A_0)X\|_{2}^2 +4\lambda_{\mathsf{S}}\lambda_{\min}(\C_\infty(B))^{-1}\Vert\hat{\A}_{\mathsf{S}}-\A_0\Vert_\star     
                 \\
    &\leq24\Delta_n^2\kappa(\C_\infty(B))\Vert\A_0\Vert_{2}^4\Big(\frac{1}{2}+\frac{1}{6}\exp(\Delta_n\Vert\A_0\Vert_{2})\Delta_n\Vert\A_0\Vert_{2}\Big)^2 
    \\
 &\quad+4\lambda_{\mathsf{S}}^2\Big(\frac{3c_\eta}{2c_\star^2\lambda_{\min}(\C_\infty(B))}
+ \frac{192}{\lambda_{\min}(\C_\infty(B))^2}\Big)\Bigg(s\log\Big(\frac{2\e d^2}{s}\Big)\lor\log(2/\ep)\Bigg),
    \end{align*}
    where we also used that $\eta_-\geq\eta/2$ and the assumptions on $T$ imply
    \[\P(\Vert\Delta\tilde{Z}_1\Vert\leq\eta_-)\geq \frac 1 2. \]
    This concludes the proof by noting that $\Vert \M\Vert_\star\geq \sqrt{\log(2)}\Vert\M\Vert_1,$ for any $\M\in\R^{d\times d}.$
\end{proof}
\section{Simulation Study}\label{sec: sim}
The aim of this section is to study the practical performance of the studied estimators on synthetic data, as well as the impact of the tuning parameters $b$ and $\eta$. In all of the following experiments, the analysis is based on discrete observations of the OU process generated by an Euler--Maruyama scheme with discretization step equal to $10^{-2}$. We choose this value based on the numerical analysis provided in \cite{gama19}. Unless specified otherwise, we assume to have observed all generated data.

We compare the Lasso and Slope estimators to two likelihood-based reference estimators.
\begin{enumerate}
    \item The estimator minimizing the truncated negative log-likelihood $\mathcal{L}^D_n$ (see \eqref{eq: like disc}) employed in the theoretical analysis. For simplicity we refer to this estimator as ``truncated MLE''.
    \item The estimator minimizing 
\begin{equation}
    \label{eq: like disc mart part}
    \frac{1}{T}\sum_{i=1}^n \left(\A X_{t_{i-1}}\right)^\top \Delta X^c_i  +\frac{\Delta_n}{2T}\sum_{i=1}^n \Vert \A X_{t_{i-1}}\Vert^2,\quad\textrm{where}\quad \Delta X^c_i=X^c_{t_i}-X^c_{t_{i-1}},
\end{equation} with $(X^c_t)_{t\geq 0}$ denoting the continuous martingale part of $\X,$ see Section \ref{sec: preliminaries}. Due to \eqref{eq: like disc mart part} being a natural discretization of the negative log-likelihood given continuous observations (see Section \ref{subsec: est}), we refer to this estimator as ``true MLE''.
\end{enumerate}  The latter is used purely as a reference, since the continuous martingale part is unknown given discrete observations. Following the approach of \cite{dex24}, the tuning parameters $\lambda_{\mathsf L}$ and $\lambda_{\mathsf S}$ for Lasso and Slope are chosen via cross-validation. Each observed trajectory is split into two consecutive parts: the first 80\% of the sample is used to compute the candidate estimators, while the remaining 20\% is reserved for validation. Each estimator is computed on the training segment for different values of the respective tuning parameters on a logarithmic grid between $10^{-3}$ and $10$ and evaluated by computing the corresponding truncated negative log-likelihood on the validation segment. The parameter minimizing this validation criterion is then selected and used as the final estimator.

The performance of each estimator will be assessed by the $L_1$ and $L_2$ errors with respect to the true drift matrix $\A_0$. 

In order to generate $\A_0$ which is both $s$-sparse and in $M_+(\R^d)$ we proceed in the following way: We first choose $s-d$ off-diagonal positions, which are, unless specified otherwise, filled with uniformly sampled values from the interval $[-0.5, 0.5].$ Afterwards the diagonal is filled with the sum of the absolute values of the off-diagonal entries in the corresponding row plus a small offset of $0.1$, i.e., 
\[(\A_0)_{ii} = \sum_{j \neq i} |(\A_0)_{ij}| + 0.1, \quad i = 1, \ldots, d.\]
This produces a sparse and diagonally dominant matrix, whichs ensures stability of the simulated process.
\begin{figure}
    \centering
    \includegraphics[width=1\textwidth]{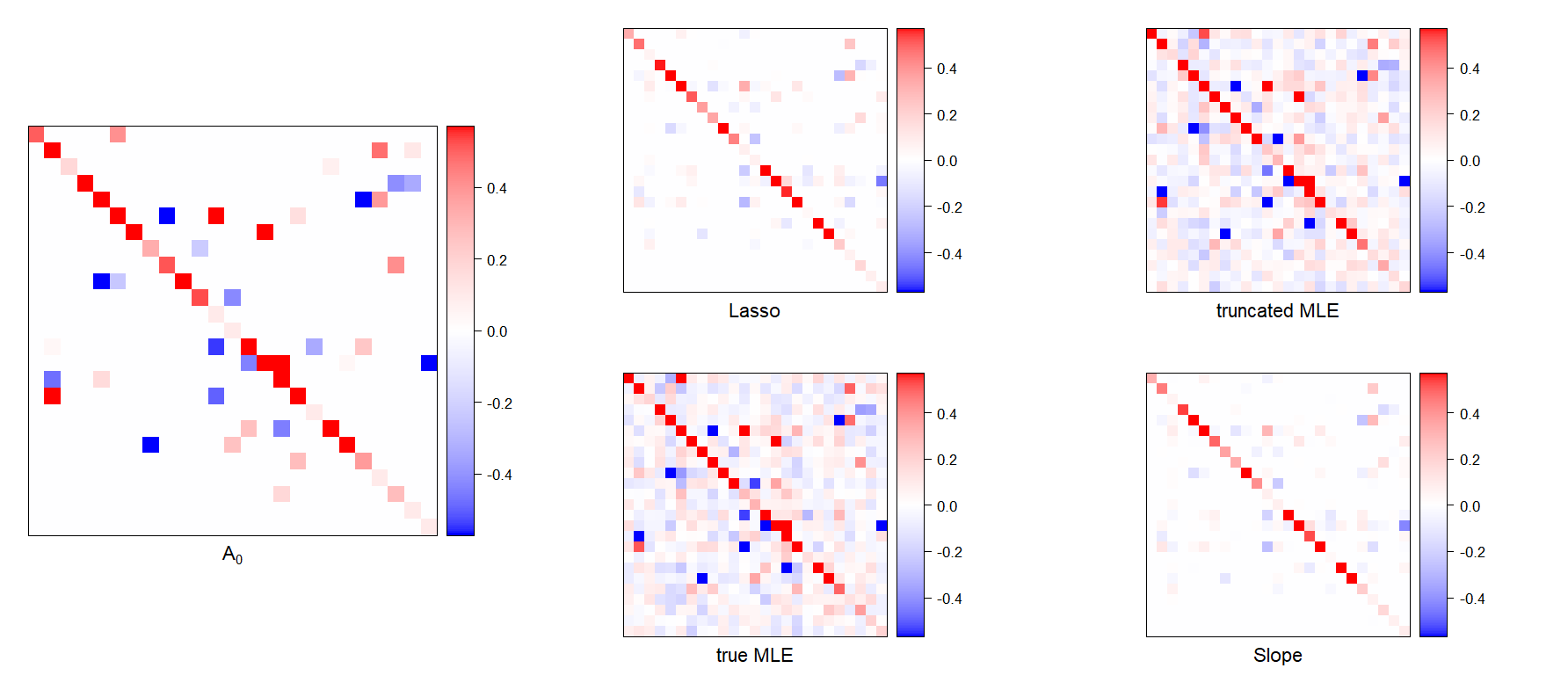}
    \caption{Comparison of the true parameter $A_0$ with Lasso and  Slope estimators, as well as true MLE and truncated MLE}
    \label{fig: heatmaps pareto int1}
\end{figure} 

Our estimators also require us to choose the set $B$ and the truncation parameter $\eta$, which are used to exclude extreme values of $\X,$ respectively its increments, from the estimation procedure. Regarding $B$ we always choose $B=B(0,b),$ where $b>0$ is a to be determined parameter. Based on the value of $b$ or $\eta$ we can compute the percentage of observations that were filtered out, see Figures \ref{fig: b pareto int10} and \ref{fig: eta laplace int10}.
We observe that the performance of the estimators stabilizes as soon as the parameters reach a certain level. Therefore, unless specified otherwise, we always choose the smallest value of $b$ and $\eta$ which filters out no more than 10\% of the observations.

\begin{figure}
\centering
\makebox[\textwidth][c]{%
    \includegraphics[width=1.1\textwidth, trim=50 0 0 0, clip]{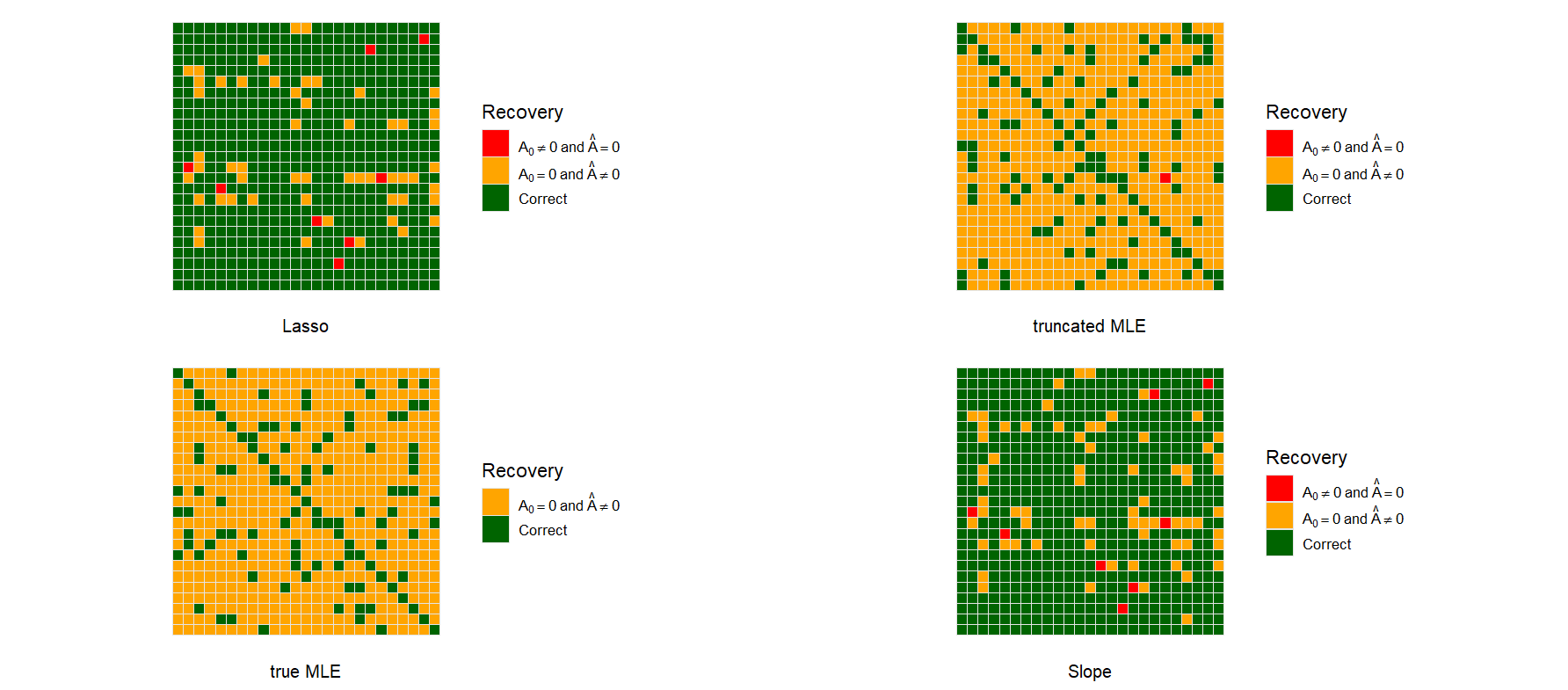}%
}
\caption{Support recovery of the investigated estimators. Green entries correspond to correctly classified coefficients of $\A_0$, meaning that the true and estimated values are either both zero or both non-zero. Red indicates entries where the estimator sets a coefficient to zero although the true value is non-zero. Orange indicates entries where a non-zero value is estimated even though the true coefficient is zero.}
\label{fig: support pareto int1}
\end{figure}

A first comparison of the different estimators can be found in Figure \ref{fig: heatmaps pareto int1}, where the drift matrix $\A_0$ and the different estimators are depicted as heat maps.
The dimension in this example is $d=25,$ the matrix $\A_0$ is $s$-sparse with $s=60$ and observations up to time $T=100$ are available. All non-zero off-diagonal entries of $\A_0$ are drawn uniformly from the interval $[-0.8, 0.8].$ 
The BDLP in this example is a sum of a standard Brownian motion and a compound Poisson process with jump intensity $1$ and Laplace-distributed jump sizes.

 We observe that the Lasso and Slope estimators recover the sparsity of the true drift matrix much better than the MLE-type estimators. The support recovery of the discussed estimators is presented in detail in Figure \ref{fig: support pareto int1}.

In order to assess the dependence of the studied estimators on the dimension $d$ in sparse settings, we compute their respective $L_1$ and $L_2$ errors for $d=10$ to $50$ while keeping the sparsity $s$ constant at $s=55$. For each value of $d$ we generate a L{\'e}vy-driven OU process with time horizon $T=100$ and repeat this procedure over $10$ iterations. The BDLP is chosen to be the same as in Figures \ref{fig: heatmaps pareto int1} and \ref{fig: support pareto int1}. The results of this study can be found in Figure \ref{fig: L1 laplace int1}. We observe that Lasso and Slope exhibit significantly smaller errors than the MLE-type estimators. Moreover, the error is almost constant over all dimensions for Lasso and Slope, which is consistent with our theoretical results. In contrast the errors of the MLE-type estimators grow as the dimension is increased and therefore fail at exploiting the sparsity of the drift matrix.

\begin{figure}
    \centering
    \includegraphics[width=0.7\textwidth]{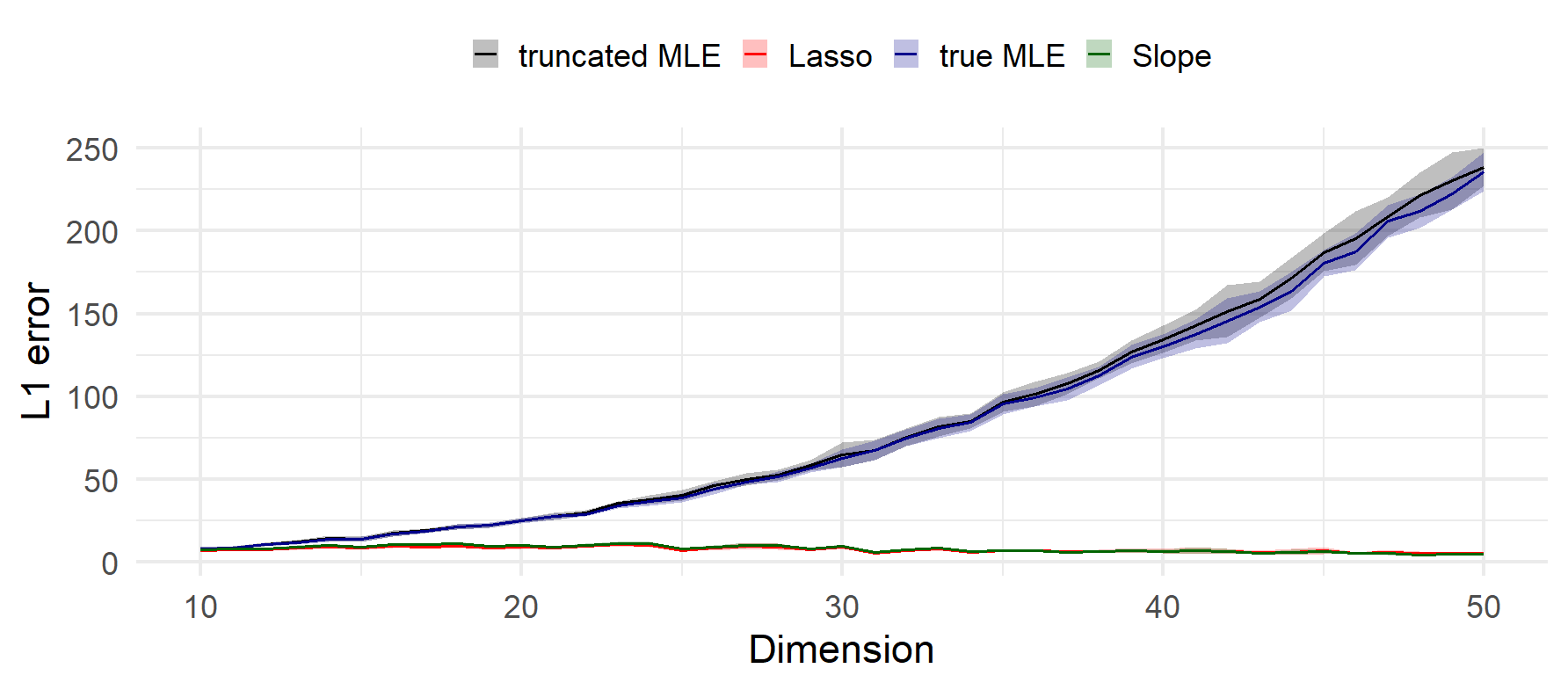}
    \\
       \includegraphics[width=0.7\textwidth]{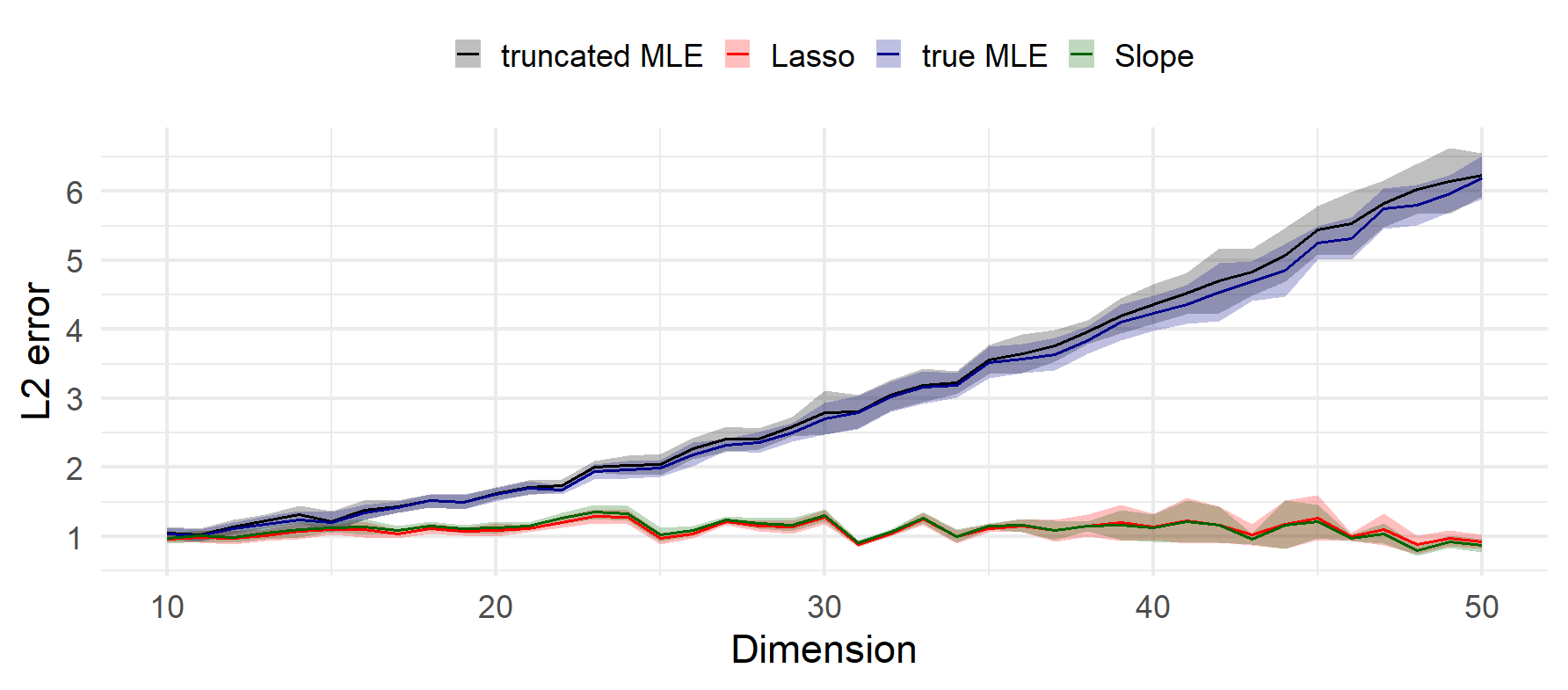}
    \caption{$L_1$ (top) and $L_2$ (bottom) errors of true and truncated MLE, Lasso and Slope $\pm$ one standard deviation.}
    \label{fig: L1 laplace int1}
\end{figure} 

 We also carried out the same experiment in a low-frequency setting, i.e. with a larger step size $\Delta_n$. More specifically,  we still simulate $\X$ on a grid with $T \cdot10^2 = 10^4$ points to ensure the precision of the Euler--Maruyama scheme, but only use $n=80$ approximately equidistant observations for the estimation. Hence the stepsize $\Delta_n = T/n$ increases from $0.01$ in the previous experiment to $1.25$. The results of this simulation can be found in Figure \ref{fig: n L1 laplace int1}. We do not observe a significant change in the behavior of Lasso and Slope estimators despite being in a low-frequency setting. A more thorough analysis of the impact of the inter-observation distance $\Delta_n$ can be found in Figure \ref{fig: n laplace int1}.
 
\begin{figure}
    \centering
    \includegraphics[width=0.7\textwidth]{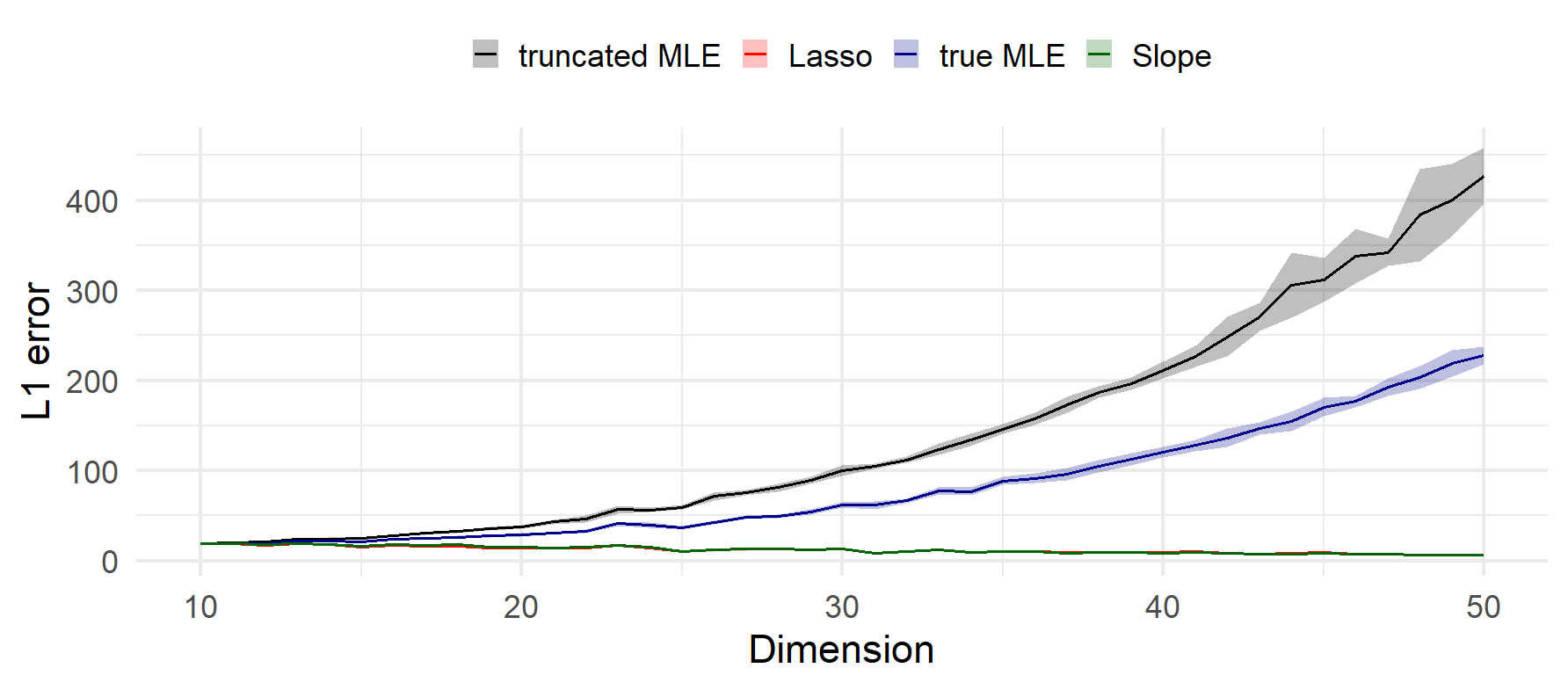}
    \\
    \includegraphics[width=0.7\textwidth]{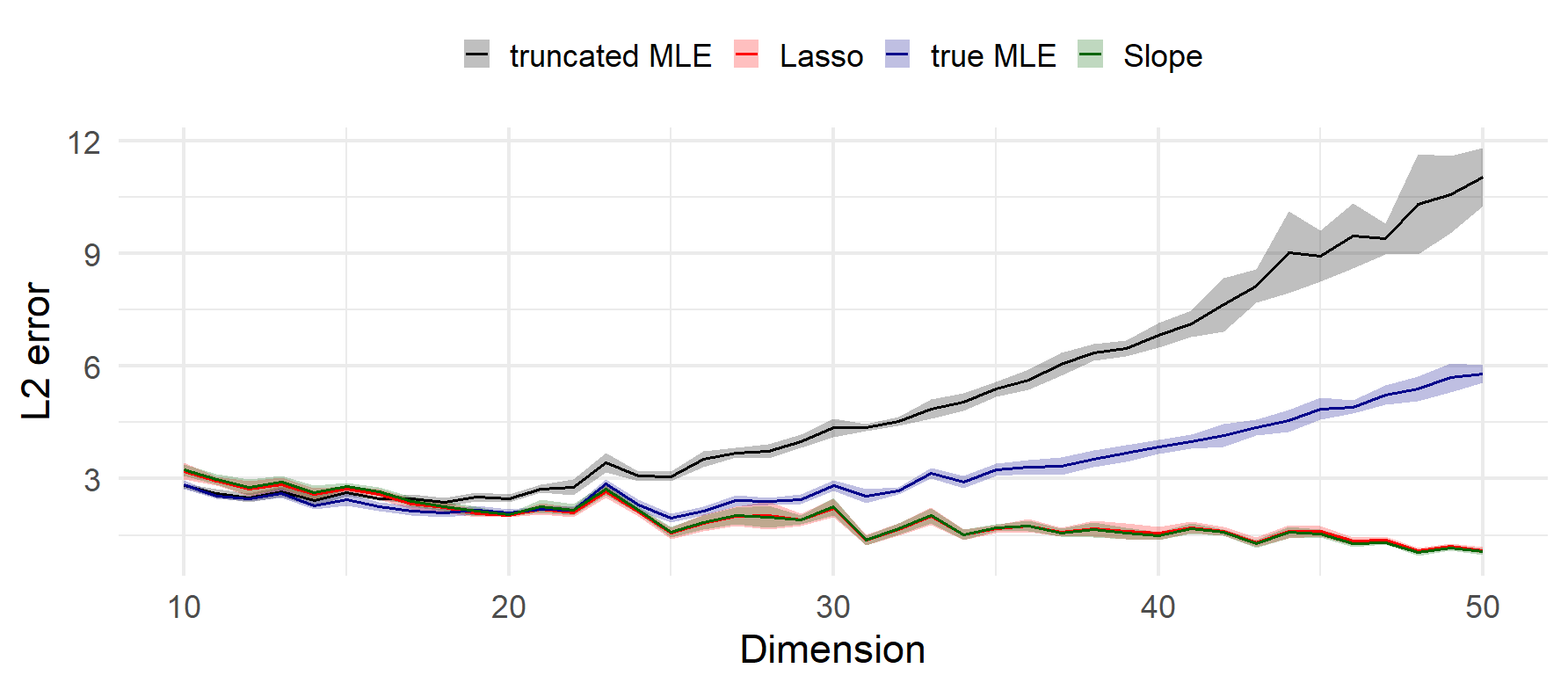}
    \caption{$L_1$ (top) and $L_2$ (bottom) errors of true and truncated MLE, Lasso and Slope $\pm$ one standard deviation. Simulations are conducted in a low-frequency setting.}\label{fig: n L1 laplace int1}  
\end{figure} 

We complement the simulation study with an additional investigation of how key model parameters affect the performance of the estimators. In particular, we analyze the impact of the truncation parameters $b$ and $\eta$, as well as the discretization level $\Delta_n$. Since Lasso and Slope estimators behaved almost identically in the previous experiments, we only conduct the experiments for the Lasso and the MLE-type estimators.

In these experiments, we expect to observe a stabilization effect. No observations are removed once $b$ and $\eta$ exceed certain levels, and the estimators' performance should therefore plateau. Similarly for sufficiently small $\Delta_n$, the discretization error becomes dominated by the stochastic error and the truncation bias, which again leads to stabilization. These theoretical expectations are confirmed by the empirical results. Moreover, the simulations allow us to examine the transitional regime prior to stabilization, illustrating how the estimators behave before reaching their limiting performance.

All subsequent simulations are conducted in dimension $d = 25$. The true drift matrix $\A_0$ is constructed to have $s=125$ non-zero entries, which corresponds to a sparsity of $20\%$. All reported results are averaged over 10 independent simulation runs.

For the truncation parameter $b$, which controls the removal of observations with large norm, Figure \ref{fig: b pareto int10} shows that the estimation errors for both Lasso and MLE decrease as $b$ increases. This behavior is expected, since larger values of $b$ allow the estimators to rely on a greater portion of the available data. The effect is considerably more pronounced for the MLE, indicating that the Lasso estimator exhibits greater robustness to extreme observations. Comparing the errors with the proportion of retained observations shows that once about $75\%$ of the data is used, the effect of the $b$-truncation becomes negligible. The BDLP in the discussed simulations was given as a sum of a standard Brownian motion and a compound Poisson process with intensity $1$ and symmetric 
Pareto-distributed jumps with parameter $4.5$. In order to focus on the impact of $b$, we choose $\eta=1000$, i.e. very large.

\begin{figure}
\centering
\includegraphics[width=.7\linewidth]{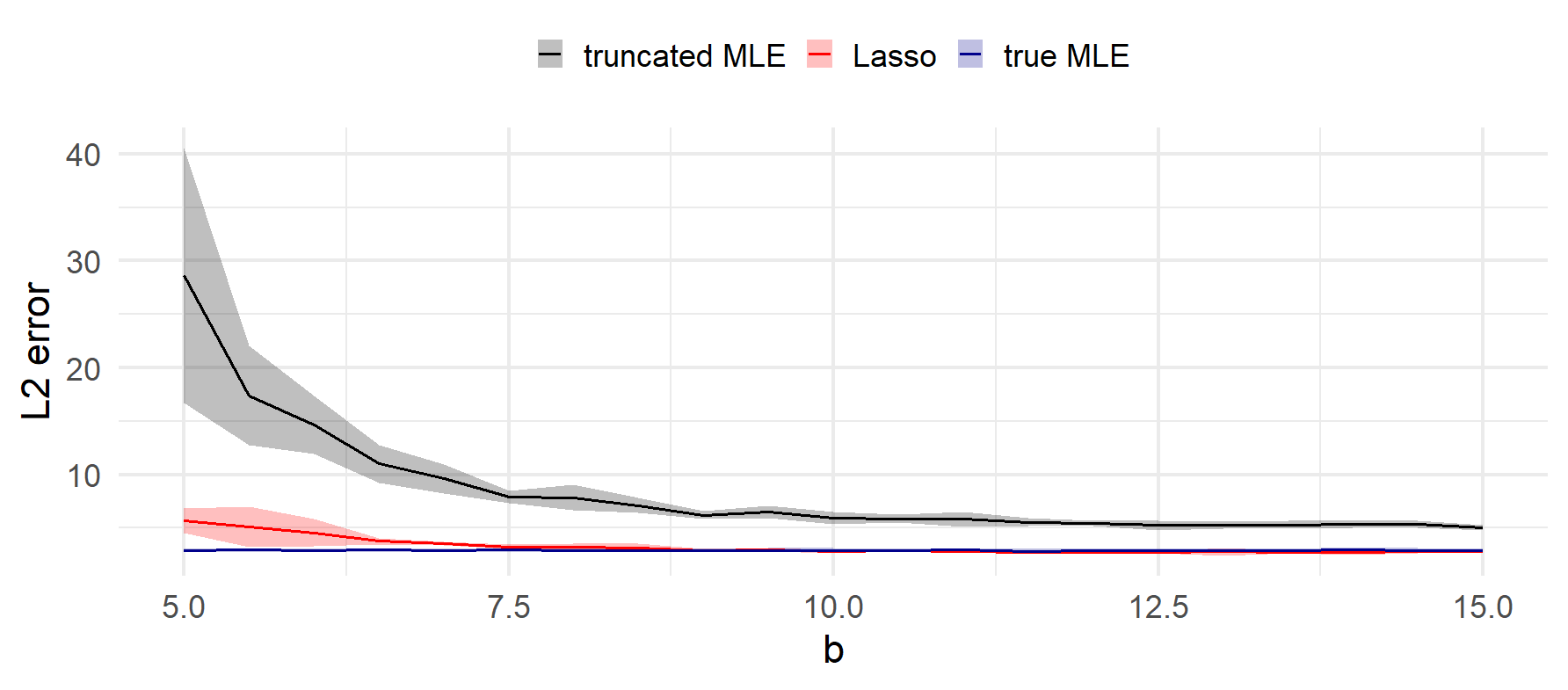}
\\
\includegraphics[width=.7\linewidth]{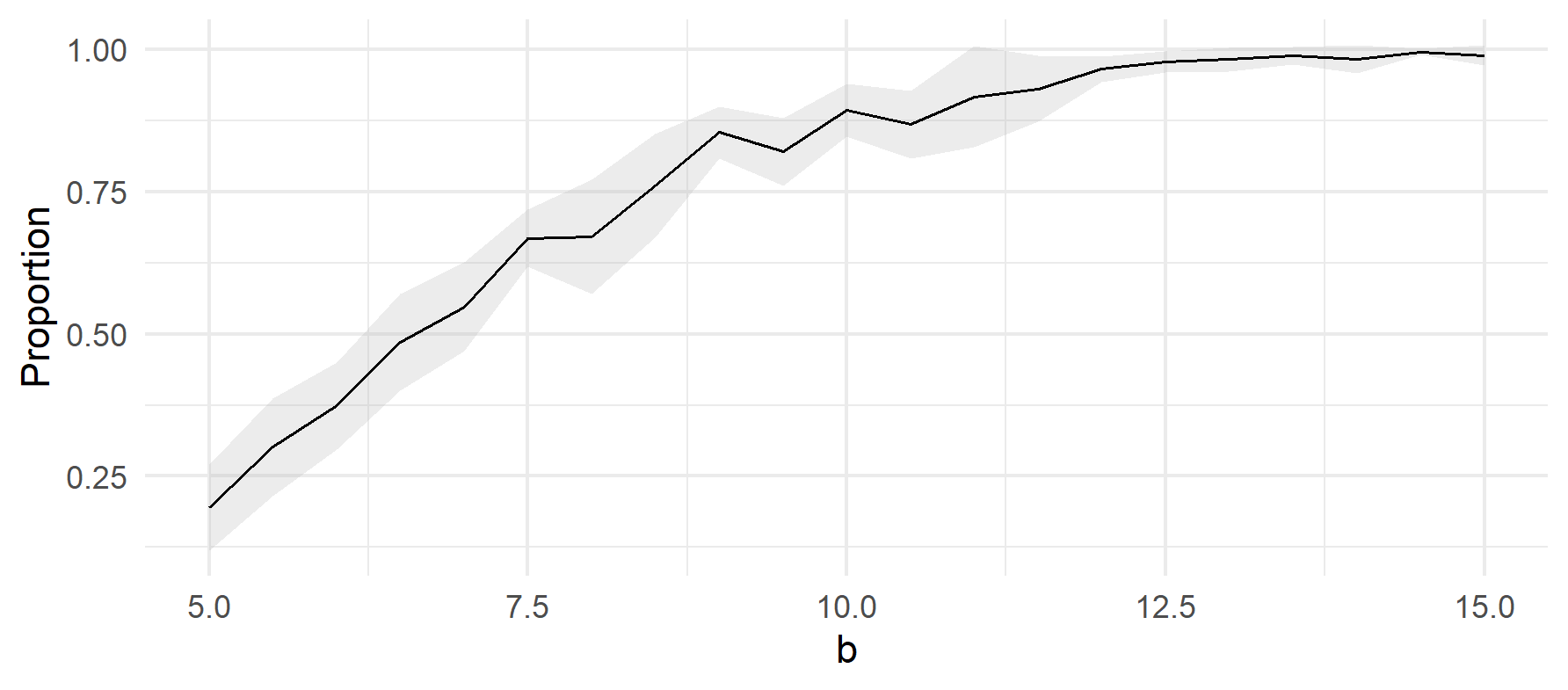}
\caption{$L_2$ errors (top) and proportion of used observations (bottom) for truncated MLE and Lasso ($\pm$ one standard deviation) for different values of $b$. The $L_2$ error of the true MLE is added as a reference in the plot on the top.}
    \label{fig: b pareto int10}
\end{figure}
We now focus on examining the impact of $\eta$, the main factor responsible for the truncation bias. We set $b=1000$ to reduce the impact of the $b$-truncation and choose the BDLP as sum of a Brownian motion plus a compound Poisson process with independent, Laplace-distributed jumps and intensity 10.
The results can be found in Figure \ref{fig: eta laplace int10} and can be summarized as follows. After an initial noisy regime, two distinct phases emerge. The first phase corresponds to using approximately 90\% of the observations (up to about $\eta = 5$), while the second phase occurs when nearly all observations are included ($\eta > 8$). In the former regime, all errors are substantially smaller, and the truncated MLE outperforms the Lasso. Its performance in particular matches the true MLE. Beyond the transition, however, the error of the truncated MLE increases significantly, whereas the error of the Lasso estimator grows only slightly. This again indicates that the Lasso is less sensitive to changes in the underlying parameter than the truncated MLE. 
\begin{figure}
\centering

\includegraphics[width=.7\linewidth]{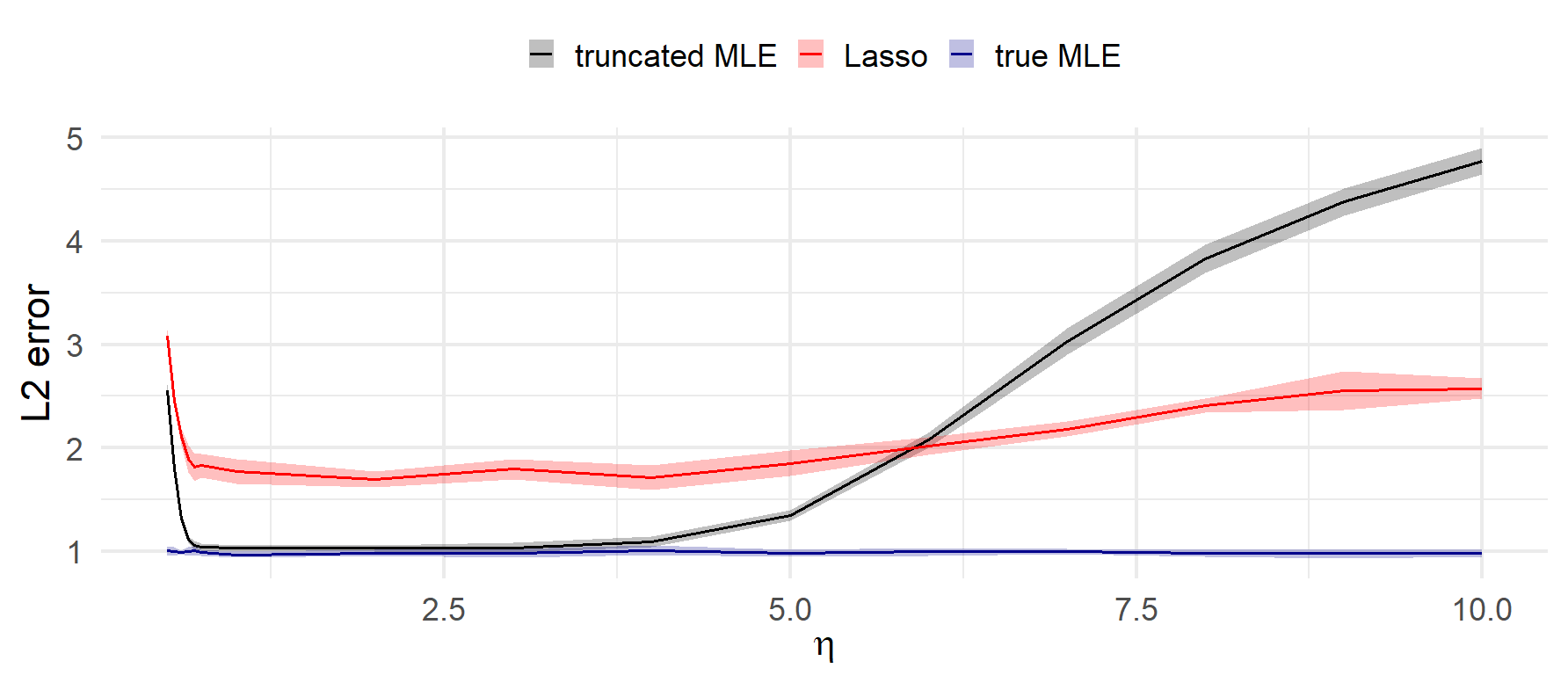}
\\
\includegraphics[width=.7\linewidth]{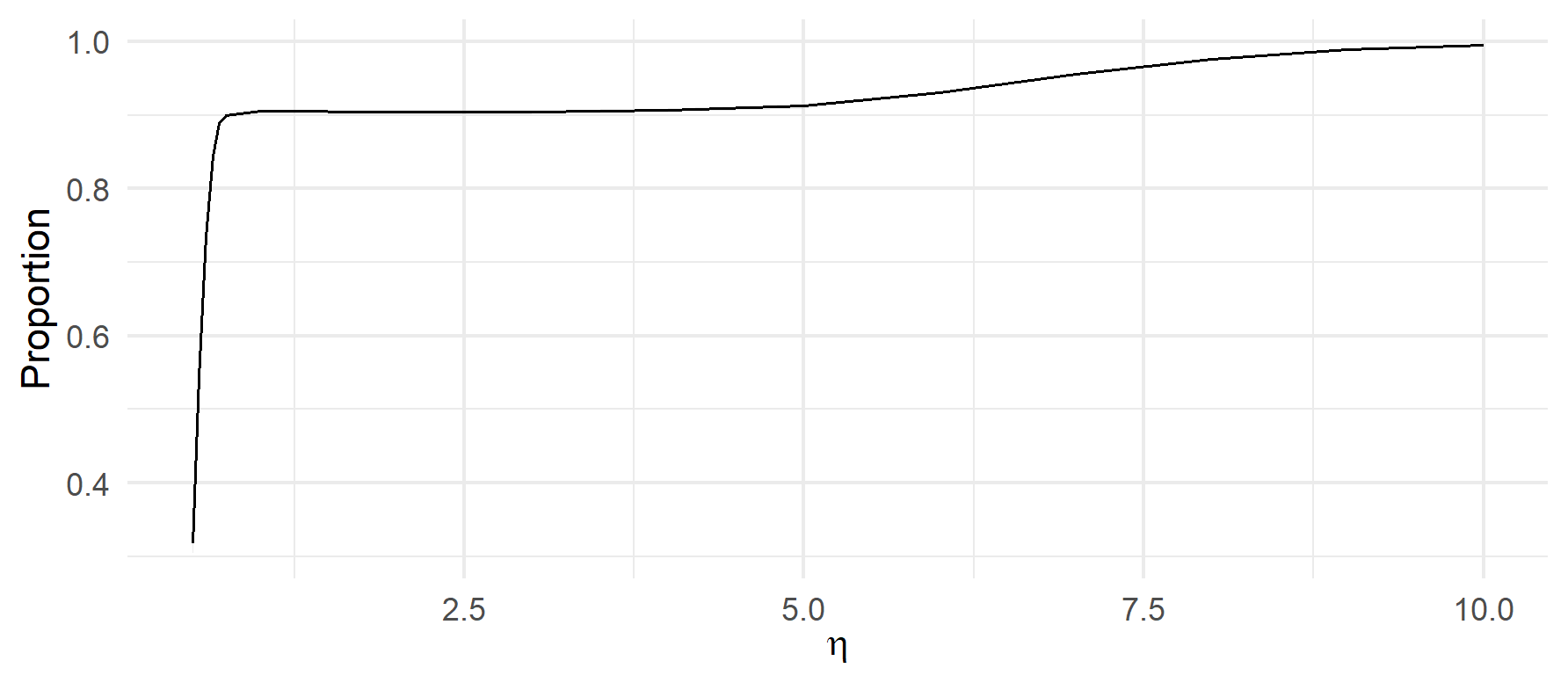}
\caption{$L_2$ errors $\pm$ one standard deviation of truncated MLE and Lasso (top) and proportion of observations used (bottom) for different values of $\eta$. As a reference the $L_2$ error of the true MLE has been added in the plot on the top.}
    \label{fig: eta laplace int10}
\end{figure}

\begin{figure}
\centering
\includegraphics[width=.7\linewidth]
{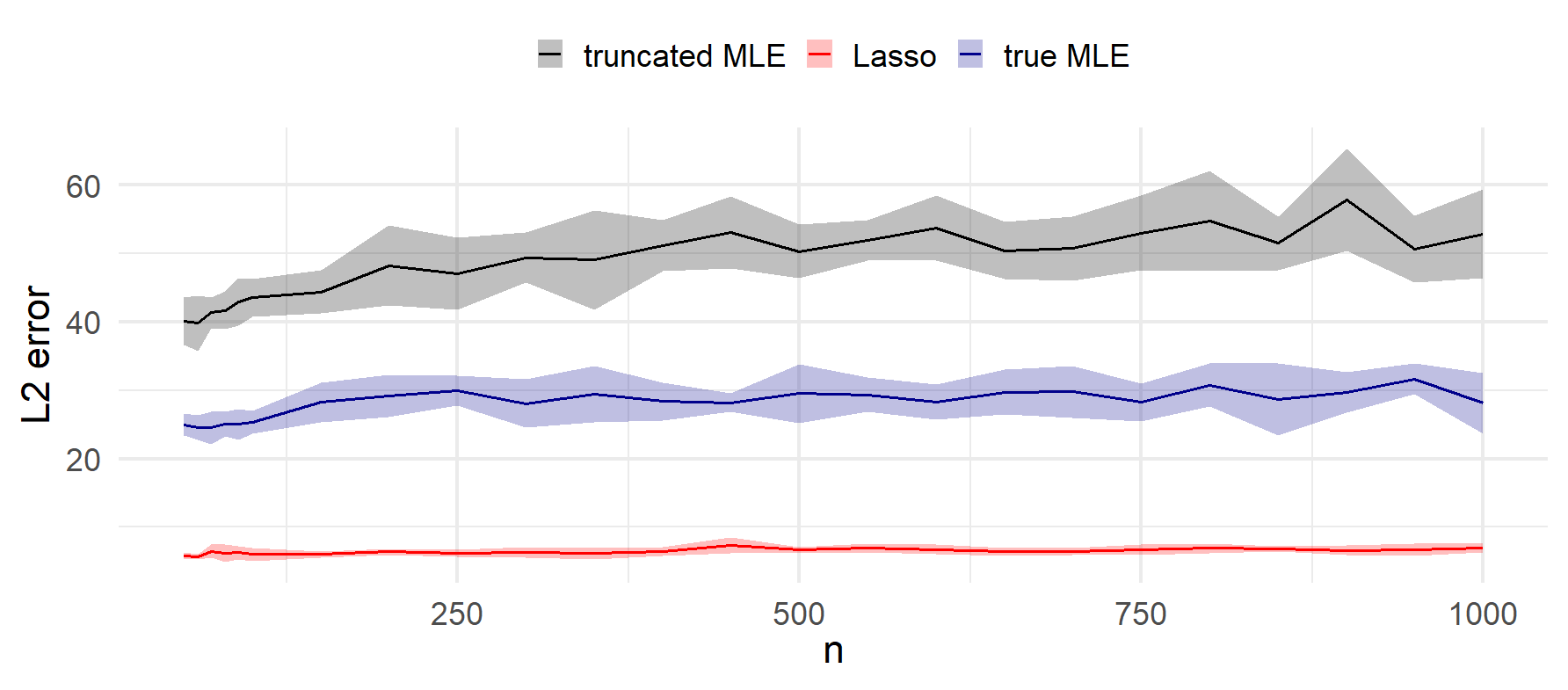}
\includegraphics[width=.7\linewidth]{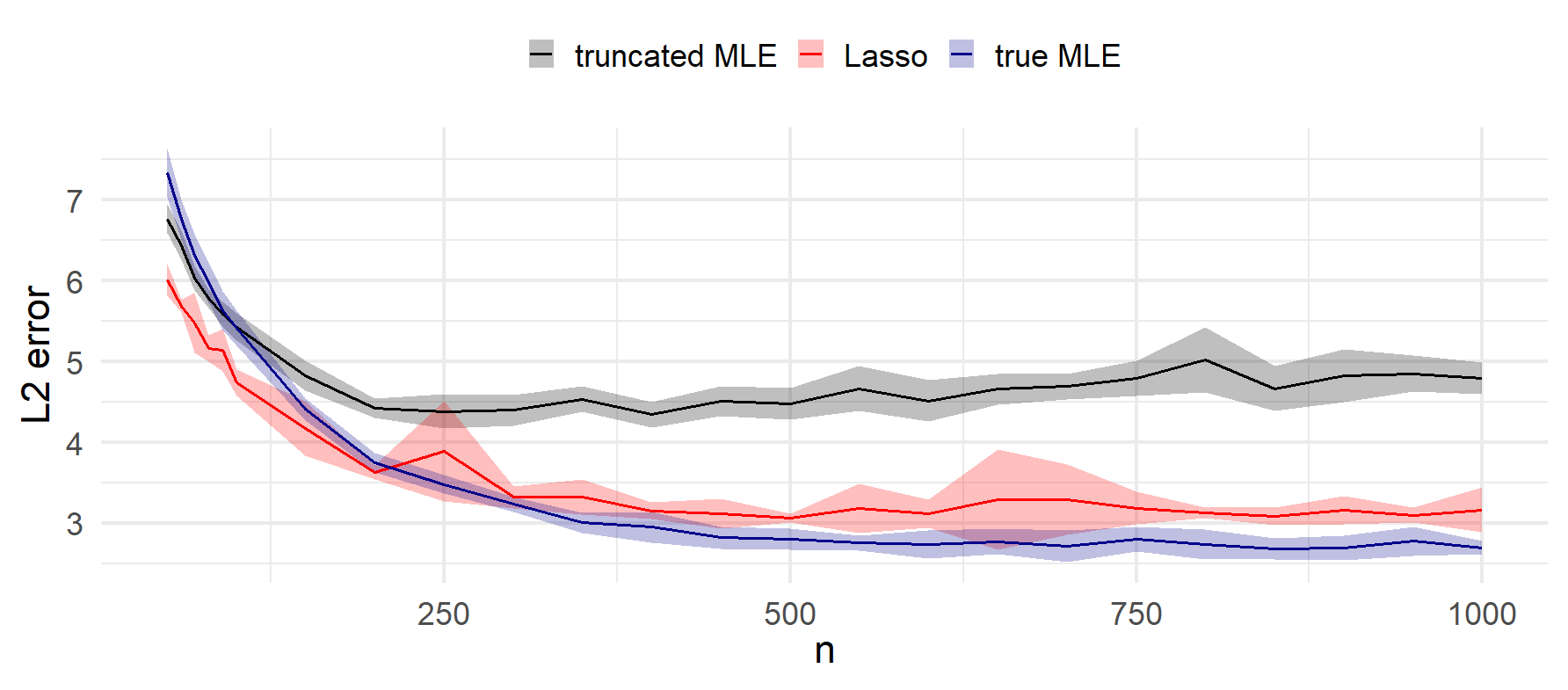}
\caption{$L_2$ errors $\pm$ one standard deviation of truncated MLE, true MLE and Lasso for different values of $n$ and $T=10$ (top) and $T=100$ (bottom).}
    \label{fig: n laplace int1}
\end{figure}
Lastly, we investigate  the impact of the inter-observation distance $\Delta_n$. For this we again choose a very liberal truncation setting of $b=\eta=1000$. As before we simulate $T\cdot 10^2$ points along the trajectory of the process, but now assume to have observed only $n$ approximately
equidistant of them. Therefore, increasing $n$ from 50 to 1000 implies higher-frequency observations and smaller $\Delta_n$. 

The BDLP is again chosen to be the sum of a standard Brownian motion and a compound Poisson process with intensity $1$ and Laplace-distributed jumps 
and we investigate the impact of $\Delta_n$ for $T=10$ and $T=100$. The results of this study can be found in Figure \ref{fig: n laplace int1}. For $T=10$ we see that the errors of both MLE and Lasso estimator stay almost constant, even for very small values of $n$ corresponding to large values of $\Delta_n$. This indicates that the stochastic error dominates in this regime as predicted by our theory. Similarly, for $T=100,$ we observe that the errors are stable as soon as $n\geq 300$ or $\Delta_n\leq 1/3$. However, for larger values of $\Delta_n$ the errors decay, which also agrees with our theoretical results. 
\section{Conclusion and Outlook}\label{sec: conclusion}
In this work we studied high-dimensional drift estimation for L{\'e}vy-driven OU processes based on discrete observations. Under sparsity assumptions on the drift matrix we showed that Lasso and Slope estimators can achieve minimax optimal results in a high frequency regime by deriving sharp nonasymptotic oracle inequalities. These bounds quantify bias, truncation error, discretization error and stochastic fluctuations. We also identified the required sample complexity depending on the tails of the BDLP's L{\'e}vy measure under very general assumptions. 

There are several possible generalizations for our work. A natural approach is extending the analysis from L{\'e}vy-driven OU processes to more general L{\'e}vy-driven diffusion processes, analogously to \cite{amorino2025sampling} for continuous diffusion processes. This requires a more in-depth analysis of the discretization error; since we focused on OU processes we could exploit the fact that their dynamics are explicitly known. Another possible extension is to the non-ergodic case. Indeed, the ergodicity assumption \ref{ass: ergodicity} requires the sparsity $s$ to be larger than the dimension $d$, which does not permit settings with extreme sparsity. For scalar OU processes the non ergodic case has for example been investigated in Chapter 3.5 of \cite{kutoyants2004statistical}. However a high-dimensional analysis needs a more refined approach; the concentration of the covariance matrix being possibly the most challenging component. One specific structural assumption for the non-ergodic case is a low-rank structure of the drift matrix. In the classical i.i.d. case this setting has been investigated by \cite{7b0bc9f7-163f-35f7-a646-75dc9ea6a85f} and possibly similar results are obtainable for OU processes. 

Since extending results from classical OU processes to L{\'e}vy-driven OU processes is closely related to generalizing statistics with sub-Gaussian noise to heavy tailed errors, another promising approach for generalizing our results is to follow work from this field. A common theme of this field is to replace the usual $L_2$ criterion in the definition of the Lasso estimator, as it penalizes the outliers occurring due to heavy-tailed noise, too strictly. One example is \cite{fan14}, which employs a $L_1$-penalized form of quantile regression, which does not require any assumptions on the tails of the noise. However, contrary to this work, in our setting it is not only the observation noise that is heavy-tailed but also the design distribution. To the best of our knowledge this has only been explored in the recent paper \cite{zhou2026minimaxoptimalrobustsparse}, of which we were unaware during the writing of this work.
\section{Acknowledgements}
This work is supported by ERC grant A2B (grant agreement no. 101124751). Part of this work
has been carried out while the authors visited the Simons Institute for the Theory of Computing
in Berkeley.
\appendix
\section{Auxiliary Results}\label{app}
\begin{lemma}\label{lemma: basic ou}
    For any time points $0 \leq t_1 < t_2 \leq T$ it holds that
    $$
    X_{t_2} = \e^{-(t_2 - t_1)\A_0}X_{t_1} + \int_{t_1}^{t_2} \e^{-(t_2 - u)\A_0}\d Z_u.
    $$
\end{lemma}

\begin{proof}
    Recall the explicit form of the solution to the Ornstein--Uhlenbeck equation \eqref{eq: ou explicit} from which it follows that
    $$
    X_0 = \e^{t_1 \A_0} \left( X_{t_1} - \int_0^{t_1} \e^{-(t_1 - u)\A_0} \d Z_u \right).
    $$
    Substituting the above again to the formula \eqref{eq: ou explicit} for the solution but now at point $t_2$, we obtain
    \begin{align*}
    X_{t_2}&=\e^{-{t_2}\A_0}\e^{t_1 \A_0}\left( X_{t_1} - \int_0^{t_1} \e^{-(t_1 - u)\A_0} \d Z_u \right) +\int_0^{t_2}\e^{-(t_2-u)\A_0}\d Z_u \\
    &= \e^{-(t_2 - t_1)\A_0}X_{t_1} + \int_{t_1}^{t_2} \e^{-(t_2 - u)\A_0}\d Z_u.
    \end{align*}
\end{proof}
\begin{lemma}\label{lemma: ou jumps, levy jumps}
    For any $\eta >0$ it holds 
    $$
\Big\{\Vert \Delta \tilde{Z}_i\Vert< \eta_-\Big\}  \cap\{X_{t_{i-1}}\in B \}\subseteq    \{\Vert \Delta X_i \Vert< \eta\}\cap\{X_{t_{i-1}}\in B \} \subseteq  \Big\{\Vert \Delta \tilde{Z}_i\Vert< \eta_+\Big\}  \cap\{X_{t_{i-1}}\in B \} $$
    and
\begin{align*}
  \{\Vert \Delta X_i \Vert\geq \eta\}\cap\{X_{t_{i-1}}\in B \}
  \subseteq \Big\{\Vert \Delta \tilde{Z}_i\Vert\geq \eta_-\Big\},
\end{align*}
where
\begin{align}    \label{eq: eta+ eta-}\eta_{+}&\coloneq\eta+b(\exp(\Delta_n\Vert \A_0\Vert_{\operatorname{op}})-1),\\
    \eta_{-}&\coloneq  \eta-b(\exp(\Delta_n\Vert \A_0\Vert_{\operatorname{op}})-1).\notag
\end{align}
\end{lemma}
\begin{proof}
Note that by Lemma \ref{lemma: basic ou} it holds
\begin{align*}
    \Vert \Delta X_{i}\Vert &=\Vert X_{t_{i}}-X_{t_{i-1}}\Vert
    \\
    &=\Vert (\exp(-\Delta_n\A_0)-\mathbb{I}_d)X_{t_{i-1}}+\int_{t_{i-1}}^{t_{i}}\exp(-(t_{i}-u))\A_0\d Z_u\Vert
    \\
    &\geq \Vert \int_{t_{i-1}}^{t_{i}}\exp(-(t_{i}-u)\A_0)\d Z_u\Vert-\Vert (\exp(-\Delta_n\A_0)-\mathbb{I}_d)X_{t_{i-1}}\Vert
    \\
    &=\Vert \Delta \tilde{Z}_i\Vert-\Vert (\exp(-\Delta_n\A_0)-\mathbb{I}_d)X_{t_{i-1}}\Vert
\end{align*}
This implies 
\begin{align*}
  &\{\Vert \Delta X_i \Vert< \eta\}\cap\{X_{t_{i-1}}\in B \}
  \\
  &\subseteq \Big\{\Vert \Delta \tilde{Z}_i\Vert< \eta+\Vert (\exp(-\Delta_n\A_0)-\mathbb{I}_d)X_{t_{i-1}}\Vert\Big\}\cap\{X_{t_{i-1}}\in B \}
    \\
  &\subseteq \Big\{\Vert \Delta \tilde{Z}_i\Vert< \eta+b\Vert (\exp(-\Delta_n\A_0)-\mathbb{I}_d)\Vert_{\mathrm{op}}\Big\}\cap\{X_{t_{i-1}}\in B \}
      \\
  &\subseteq \Big\{\Vert \Delta \tilde{Z}_i\Vert< \eta+b(\exp(\Delta_n\Vert \A_0\Vert_{\operatorname{op}})-1)\Big\}\cap\{X_{t_{i-1}}\in B \}.
\end{align*}
The other set inclusions follow by analogous arguments.
\end{proof}
\begin{lemma}\label{lemma: ew z}
For any $i\in[n]$ it holds  
\[\lambda_{\max}\Big(\E[(\Delta \tilde{Z}_i)(\Delta \tilde{Z}_i)^\top \vert \mathcal{F}_{t_{i-1}}]\Big)\leq \Delta_n \lambda_{\max}(\C+\boldsymbol{\nu}_2)\exp(\Delta_n\Vert\A_0\Vert).\]
\end{lemma}
\begin{proof}
    Combining the Courant--Fischer theorem with It{\^o}'s isometry gives, since the $\Delta \tilde{Z}_i,$ are i.i.d.\ and $\Delta \tilde{Z}_i$ is independent of $\mathcal{F}_{t_{i-1}}$
    \begin{align*}
        &\lambda_{\max}\Big(\E[(\Delta \tilde{Z}_i)(\Delta \tilde{Z}_i)^\top \vert \mathcal{F}_{t_{i-1}}]\Big)
        \\
        &=\max_{u\in\R^d:\Vert u\Vert=1}\E[(u^\top\Delta \tilde{Z}_1)^2]
        \\
        &=\max_{u\in\R^d:\Vert u\Vert=1}\Bigg(\int_0^{\Delta_n}\Vert u^\top\exp(-(\Delta_n-s)\A_0)\bSigma\Vert^2\d s +\int_0^{\Delta_n}\int_{\R^d}\Vert u^\top\exp(-(\Delta_n-s)\A_0)z\Vert^2\nu(\md z)\d s\Bigg)
          \\
        &=\max_{u\in\R^d:\Vert u\Vert=1}u^\top\Bigg(\int_0^{\Delta_n}\exp(-(\Delta_n-s)\A_0)\Big(\C+\int_{\R^d}zz^\top\nu(\md z)\Big)\exp(-(\Delta_n-s)\A_0^\top)\d s \Bigg)u
                  \\
        &\leq \lambda_{\max}(\C+\boldsymbol{\nu}_2)\max_{u\in\R^d:\Vert u\Vert=1}\int_0^{\Delta_n}\Vert\exp(-(\Delta_n-s)\A_0^\top)u\Vert ^2\d s 
                   \\
        &\leq\Delta_n \lambda_{\max}(\C+\boldsymbol{\nu}_2)\exp(\Delta_n\Vert\A_0\Vert),
    \end{align*}
    which concludes the proof.
\end{proof}

\begin{proposition}\label{prop: mix}
Assume that the L{\'e}vy triplet of the BDLP $\Z$ is given as $(0,\C,0).$ Then for $T\geq \lambda_{\max}(\C_\infty)\lambda_{\min}(\C)^{-1}\log(2d \kappa(\C_{\infty})^2\kappa(\C)),$ it holds 

 \begin{align*}
        &\beta_{\X}(T)
        \leq \frac{3}{2}d\kappa(\C_{\infty})^{2}\kappa(\C)   \exp\left(-\frac{T\lambda_{\min}(\C)}{\lambda_{\max}(\C_\infty)} \right)
        +\frac{1}{\sqrt{2}} d\kappa(\C_{\infty}) \exp\left(-\frac{T\lambda_{\min}(\C)}{2\lambda_{\max}(\C_\infty)} \right).
    \end{align*}
\end{proposition}
\begin{proof}
First recall that  $\C_{\infty}=\int_0^\infty \e^{-s\A_0}\C\e^{-s\A_0^\top}\d s,$ solves the Lyapunov equation
    \[-\A_0 \C_\infty-\C_\infty \A_0^\top=-\C, \]
    and thus Corollary 5 in \cite{xu97} gives for any $T>0$
    \begin{align}
        \left\Vert\e^{-T\A_0}\right\Vert^2_{2}\notag
        &\leq \kappa(\C_{\infty}) \tr\left(\exp\left(-T\lambda_{\max}(\C_\infty)^{-1}\C \right) \right)\notag
        \\
        &\leq d \kappa(\C_{\infty}) \exp\left(-T\lambda_{\max}(\C_\infty)^{-1}\lambda_{\min}(\C) \right).\label{eq: xu bound}
    \end{align}
    By Propositions 2.1 and 2.2 in \cite{mas04}, respectively Theorems 3.1, 4.1 and 4.2 in \cite{sato84}, $\mu$ and $P_T(x,\cdot)$ are infinitely divisible for any $x\in\R^d$ with respective L{\'e}vy triplets $(0,\C_{\infty}, 0)$ and $(\e^{-T\A_0}x,\C_T, 0),$ where
    \[\C_T\coloneq\int_0^T \e^{-s\A_0}\C\e^{-s\A_0^\top}\d s,\] i.e.\ they follow a normal distribution. Now for any $x\in\R^d,T\geq 0$ it holds,
    \begin{equation}\label{eq: mix triangle}
        \Vert P_T(x,\cdot)-\mu\Vert_{\operatorname{TV}}
        \leq \Vert P_T(x,\cdot)-P_T(0,\cdot)\Vert_{\operatorname{TV}}+\Vert P_T(0,\cdot)-\mu\Vert_{\operatorname{TV}},
    \end{equation}
    and by Theorem 1.1 and equation (2) of \cite{devroye2023total} and \eqref{eq: xu bound} we get
        \begin{align}
      &\Vert P_T(0,\cdot)-\mu\Vert_{\operatorname{TV}}\notag
      \\
      &\leq \frac{3}{2}\left\Vert \C_\infty^{-1/2}\int_T^\infty \e^{-s\A_0}\C\e^{-s\A_0^\top}\d s \C_\infty^{-1/2} \right\Vert_2\notag
      \\  
      &\leq \frac{3}{2}\Vert \C_\infty^{-1/2} \Vert^2\left\Vert \int_T^\infty \e^{-s\A_0}\C\e^{-s\A_0^\top}\d s \right \Vert_2 \notag
      \\ 
      &\leq \frac{3}{2}\lambda_{\min}\left(\C_\infty\right)^{-1}\int_T^\infty \left\Vert \e^{-s\A_0}\C\e^{-s\A_0^\top}\right\Vert_2\d s  \notag
      \\
      &\leq \frac{3}{2}\lambda_{\min}\left(\C_\infty\right)^{-1}\lambda_{\max}(\C)\int_T^\infty \left\Vert \e^{-s\A_0}\right\Vert^2_2\d s\notag
      \\
      &\leq \frac{3}{2}d\lambda_{\min}(\C_\infty)^{-1}\lambda_{\max}(\C)\kappa(\C_{\infty})\int_T^\infty   \exp\left(-s\lambda_{\max}(\C_\infty)^{-1}\lambda_{\min}(\C) \right)\d s\notag
      \\
      &=\frac{3}{2}d\kappa(\C)\kappa(\C_{\infty})^2   \exp\left(-T\lambda_{\max}(\C_\infty)^{-1}\lambda_{\min}(\C)\right).\label{eq: mix bound 1}
    \end{align}
    Additionally, Pinsker's inequality and classical results for the Kullback--Leibler divergence of normal distributions give together with \eqref{eq: xu bound}
    \begin{align*}
        &\Vert P_T(x,\cdot)-P_T(0,\cdot)\Vert_{\operatorname{TV}}
        \\
        &\leq \sqrt{\frac{1}{2}\operatorname{KL}(P_T(x,\cdot),P_T(0,\cdot))}
        \\
        &=\frac{1}{2}\sqrt{x^\top\e^{-T\A^\top_0}\C_T^{-1}\e^{-T\A_0}x}
        \\
        &=\frac{1}{2} \left\Vert\C_T^{-1/2}\e^{-T\A_0}x\right\Vert 
        \\
        &\leq \frac{1}{2}  \Vert x\Vert \left\Vert\C_T^{-1/2}\e^{-T\A_0}\right\Vert
        \\
        &\leq \frac{1}{2}  \Vert x\Vert \lambda_{\min}\left(\C_T\right)^{-1/2}\left\Vert\e^{-T\A_0}\right\Vert_{2}
         \\
        &\leq \frac{1}{2}  \Vert x\Vert \lambda_{\min}\left(\C_T\right)^{-1/2}\left( d\kappa(\C_{\infty}) \exp\left(-T\lambda_{\max}(\C_\infty)^{-1}\lambda_{\min}(\C) \right)\right)^{1/2}.
    \end{align*}

Furthermore the min-max Theorem implies together with \eqref{eq: xu bound}
\begin{align*}
    \lambda_{\min}\left(\C_T\right)
    &=\min_{x\in\R^d:\Vert x\Vert=1}\int_0^T x^\top\e^{-s\A_0}\C\e^{-s\A_0^\top}x\d s
    \\
    &\geq\min_{x\in\R^d:\Vert x\Vert=1}\int_0^\infty x^\top\e^{-s\A_0}\C\e^{-s\A_0^\top}x\d s- \max_{x\in\R^d:\Vert x\Vert=1}\int_T^\infty x^\top\e^{-s\A_0}\C\e^{-s\A_0^\top}x\d s
      \\
    &\geq \lambda_{\min}(\C_\infty)-\lambda_{\max}(\C)\int_T^\infty \Vert \e^{-s\A_0^\top}\Vert_2^2\d s
    \\
    &\geq \lambda_{\min}(\C_\infty)-d \kappa(\C_{\infty})\lambda_{\max}(\C)\int_T^\infty  \exp\left(-s\lambda_{\max}(\C_\infty)^{-1}\lambda_{\min}(\C) \right)\d s
    \\
    &= \lambda_{\min}(\C_\infty)-d \lambda_{\max}(\C_\infty)\kappa(\C_{\infty})\kappa(\C) \exp\left(-T\lambda_{\max}(\C_\infty)^{-1}\lambda_{\min}(\C) \right),
\end{align*}
which gives for $T\geq \lambda_{\max}(\C_\infty)\lambda_{\min}(\C)^{-1}\log(2d \kappa(\C_{\infty})^2\kappa(\C))\eqcolon T_0,$
\[\lambda_{\min}\left(\C_T\right)\geq \frac{1}{2}\lambda_{\min}(\C_\infty), \]
and thus for $T\geq T_0$ it holds 
 \begin{align}
        &\Vert P_T(x,\cdot)-P_T(0,\cdot)\Vert_{\operatorname{TV}}\notag
         \\
        &\leq   \Vert x\Vert (2\lambda_{\min}\left(\C_{\infty}\right))^{-1/2}\left( d\kappa(\C_{\infty})\exp\left(-T\lambda_{\max}(\C_\infty)^{-1}\lambda_{\min}(\C) \right)\right)^{1/2}.\label{eq: mix bound 2}
    \end{align}
    Combining \eqref{eq: mix triangle}, \eqref{eq: mix bound 1} and \eqref{eq: mix bound 2}, we obtain for $T\geq T_0$ by H{\"o}lder's inequality
    \begin{align*}
        \beta_{\X}(T)&=\int \Vert P_T(x,\cdot)-\mu\Vert_{\operatorname{TV}}\mu(\d x)
        \\
        &\leq \frac{3}{2}d\kappa(\C)\kappa(\C_{\infty})^2   \exp\left(-T\lambda_{\max}(\C_\infty)^{-1}\lambda_{\min}(\C) \right)
        \\
        &\quad+ (2\lambda_{\min}\left(\C_{\infty}\right))^{-1/2}\left( d\kappa(\C_{\infty})\exp\left(-T\lambda_{\max}(\C_\infty)^{-1}\lambda_{\min}(\C) \right)\right)^{1/2}\E[\Vert X_0\Vert^2]^{1/2}
                \\
        &\leq \frac{3}{2}d\kappa(\C)\kappa(\C_{\infty})^2   \exp\left(-T\lambda_{\max}(\C_\infty)^{-1}\lambda_{\min}(\C) \right)
        \\
        &\quad+ \frac{1}{\sqrt{2}}d\kappa(\C_{\infty}) \exp\left(-T\lambda_{\max}(\C_\infty)^{-1}\lambda_{\min}(\C)/2 \right),
    \end{align*}
    which concludes the proof.
\end{proof}

\printbibliography
\end{document}